\documentclass[aos]{imsart}

\RequirePackage{amsthm,amsmath,amsfonts,amssymb}
\RequirePackage[numbers]{natbib}
\RequirePackage[colorlinks,citecolor=blue,urlcolor=blue]{hyperref}
\RequirePackage{graphicx}
\usepackage[dvipsnames]{xcolor}

\startlocaldefs

\usepackage{amssymb,amsmath,bm,amsthm,dsfont}
\usepackage{mathrsfs}

\usepackage{hyperref}
\theoremstyle{plain}

\newtheorem{theorem}{Theorem}[section]
\newtheorem{lemma}[theorem]{Lemma}

\theoremstyle{remark}

	\newcommand{\beq}{\begin{eqnarray}}
	\newcommand{\eeq}{\end{eqnarray}}
\newcommand{\beas}{\begin{eqnarray*}}
	\newcommand{\enas}{\end{eqnarray*}}
\newcommand{\bea}{\begin{eqnarray}}
	\newcommand{\ena}{\end{eqnarray}}
\newcommand{\bms}{\begin{multline*}}
	\newcommand{\ems}{\end{multline*}}
\newcommand{\qmq}[1]{\quad \mbox{#1} \quad}
\newcommand{\qm}[1]{\quad \mbox{#1}}

\newcommand{\ignore}[1]{}
\newcommand{\Bvert}{\left\vert\vphantom{\frac{1}{1}}\right.}
\newcommand{\bsy}[1]{\boldsymbol #1}
\newcommand{\R}{\mathbb{R}}
\newcommand{\Id}{\operatorname{Id}}
\newtheorem{corollary}{Corollary}[section]

\newtheorem{proposition}{Proposition}[section]
\newtheorem{remark}{Remark}[section]
\newtheorem{example}{Example}[section]

\newtheorem{model}{Model}[section]
\newtheorem{assumption}{Assumption}[section]

\newcommand{\transpose}{{\mbox{\scriptsize\sf T}}}
\newcommand{\Tk}{T}
\newcommand{\norm}[1]{\left\lVert#1\right\rVert}
\newcommand{\HS}{}
\newcommand{\commentout}[1]{}

\endlocaldefs

\begin{document}

\begin{frontmatter}
\title{Relaxing the Gaussian assumption in Shrinkage and SURE in high dimension}
\runtitle{Shrinkage and SURE in high dimension}

\begin{aug}
\author[A]{\fnms{Max} \snm{Fathi}\ead[label=e1]{mfathi@lpsm.paris}},
\author[B]{\fnms{Larry} \snm{Goldstein}
	\ead[label=e2]{larry@math.usc.edu}},
\author[C]{\fnms{Gesine} \snm{Reinert}\ead[label=e3]{reinert@stats.ox.ac.uk}}
\and
\author[D]{\fnms{Adrien} \snm{Saumard}\ead[label=e4]{adrien.saumard@ensai.fr}}

\address[A]{LJLL \& LPSM, Universit\'e de Paris,
\printead{e1}}

\address[B]{Department of Mathematics,
University of Southern California,
\printead{e2}}

\address[C]{Department of Statistics,
	 University of Oxford, 
	\printead{e3}}

\address[D]{Universit\'e de Rennes, Ensai, CREST-UMR 9194, Rennes F-35000, France
	\printead{e4}}

\end{aug}

\begin{abstract}
Shrinkage estimation is a fundamental tool of modern statistics, pioneered by Charles Stein upon his discovery of the famous paradox involving the multivariate Gaussian. A large portion of the subsequent literature only considers the efficiency of shrinkage, and that of an associated procedure known as Stein's Unbiased Risk Estimate, or SURE, in the Gaussian setting of that original work.
We investigate what extensions to the domain of validity of shrinkage and SURE can be made away from the Gaussian through the use of tools developed in the probabilistic area now known as Stein's method. We show that shrinkage is efficient away from the Gaussian under very mild conditions on the distribution of the noise. SURE is also proved to be adaptive under similar assumptions, and in particular in a way that retains the classical asymptotics of Pinsker's theorem. Notably, shrinkage and SURE are shown to be efficient under mild distributional assumptions, and particularly for general isotropic log-concave measures. 
\end{abstract}

\begin{keyword}[class=MSC2020]
\kwd[Primary ]{62F12}
\kwd{62F35}
\end{keyword}

\begin{keyword}
\kwd{Shrinkage estimation}
\kwd{Stein Kernel}
\kwd{Zero bias}
\kwd{Unbiased risk estimation}
\end{keyword}

\end{frontmatter}

\section{Introduction}The breakthrough, counter-intuitive results of the works \cite{St56} and \cite{JaSt61} showed that the `natural' estimate of the unknown mean $\bsy{\theta} \in \R^d $ of an observation ${\bm X}$ having the normal distribution ${\mathcal N}_d(\bsy{\theta},\sigma^2 \Id)$ in dimensions $d \ge 3$ is not admissible under mean squared error loss. In particular,  with $\| \cdot \|$ denoting the Euclidean norm, for $d \ge 3$ it was demonstrated that for
\begin{align} \label{eq:shrinkage}
S_\lambda({\bm X}) = {\bm X}\left( 1-\frac{\lambda }{\|{\bm X}\|^2} \right)
\qmq{for $\lambda \ge 0$}
\end{align}
there exists a range of positive values for $\lambda$ for which $S_\lambda({\bm X})$ has a strictly smaller mean squared error than $S_0({\bm X})$. Here, by the properties of the Gaussian, ${\bm X}$ could represent the mean of an independent sample from this same Gaussian distribution with $\sigma^2$ properly re-scaled. 

In \cite{St81}, related ideas, and in particular the use of Stein's lemma, were applied to construct what is now known as SURE, for Stein's Unbiased Risk Estimate, that provides an unbiased estimator for the mean squared error of a nearly arbitrary estimator of a multivariate mean, again in the Gaussian context. For ${\bm f}:\mathbb{R}^d \rightarrow \mathbb{R}^d$ let $\nabla {\bm f}$ and $\nabla \cdot {\bm f} $ denote the Jacobian matrix, and divergence of ${\bm f}$, respectively; 
precisely, with $\partial_j$ denoting taking the partial derivative with respect to the $j^{th}$ coordinate variable, 
$[\nabla {\bm f}]_{i,j} = \partial_j  f_i$.  With $\nu$ generally denoting the distribution of ${\bm X}$, which for now is the normal ${\mathcal N}_d(\bsy{\theta},\sigma^2 \Id)$, let $W^{1,2}(\nu)$ denote the natural (weighted) Sobolev space induced by the (squared) Sobolev norm
\begin{align}\label{def:W.norm.ker}
||{\bm f}||_{W^{1,2}(\nu)}^2 := ||{\bm f}||_{L^2(\nu)}^2 + ||\nabla {\bm f}||_{L^2(\nu)}^2,
\end{align}
where the second term is the usual (squared) Hilbert-Schmidt norm induced by the scalar product $\langle A, B \rangle =\operatorname{Tr}(AB^\top)$ on matrices $A$ and $B$, (see for instance \cite{CFP19}). 
Stein's identity gives the characterization of the multivariate normal distribution 
that ${\bm X} \sim {\mathcal N}_d(\bsy{\theta},\Sigma)$ if and only if
\begin{align} \label{eq:Stein.Identity.d}
E[\langle {\bm X}-\bsy{\theta},{\bm f}({\bm X})\rangle] =  E[\langle \Sigma, \nabla {\bm f}({\bm X})\rangle ] \qmq{for all ${\bm f} \in W^{1,2}(\nu)$.}
\end{align}
In particular, via \eqref{eq:Stein.Identity.d}, for ${\bm f} \in W^{1,2}(\nu)$ and ${\bm X} \sim {\mathcal N}_d(\bsy{\theta},\sigma^2 \operatorname{Id})$ we have that
\begin{align} \label{eq:intro.sure}
{\rm SURE}({\bm f},{\bm X}) :&= d\sigma^2 + \|{\bm f}({\bm X})\|^2 + 2\sigma^2  \nabla \cdot {\bm f}({\bm X})\\ \nonumber
&\mbox{is unbiased for the risk of} 
 \quad S({\bm X})={\bm X}+{\bm f}({\bm X}),
\end{align}
that is, unbiased for the expectation of $\|S({\bm X})-\bsy{\theta}\|^2$. In particular, taking
\begin{align}\label{def:f.g0}
{\bm f}({\bm x}) = - \lambda
{\bm g}_0({\bm x}) 
\qmq{where} 
{\bm g}_0({\bm x}) = \frac{{\bm x}}{|| {\bm x}||^2}
\end{align}
in \eqref{eq:intro.sure}
gives an unbiased estimator for the risk of the shrinkage estimator \eqref{eq:shrinkage}.

Since shrinkage estimation and SURE first appeared, they have been extensively studied in the statistical literature and applied in practice in many contexts, see, for instance \cite{DoJo95,Zh98,CaSiTr13, BaErMo13}.
Regarding shrinkage, previous result for non-Gaussian distributions have appeared in the works \cite{FSW18} and \cite{CWH11},\cite{CFR89},\cite{SrBi89} that consider the estimation of high dimensional covariance matrices under spherically and elliptically symmetric distributional assumptions. Compared to those works, an advantage of our approach is that a number of our main results  completely avoid any assumption of symmetry.

In addition, the work \cite{ES96} considers shrinkage estimation of the mean ${\bm \theta}$ based on the observation ${\bm X}={\bm Y}+{\bm \theta}$ under the assumption that $d \ge 3, E[{\bm Y}]=0,E[||{\bm Y}||^2]<\infty$, and that there exists a (possibly randomized) stopping time ${\bm t}$ for an $\mathbb{R}^d$ valued Brownian motion ${\rm B}_{s \ge 0}$ such that the distribution of ${\rm B}_{\bm t}$ is that of ${\bm Y}$. However, the proof of the main result of \cite{ES96} appears to be in error, in that it does not take into account that the stopping time involved depends on the path, and that its variation must be taken into account when taking a derivative with respect to the parameter $\epsilon$ that controls the magnitude of the drift of the perturbed Brownian motion constructed when deriving \cite[Equation (3)]{ES96}.

Regarding the use of SURE in non-Gaussian settings,  
\cite{Eldar08} extends SURE to exponential families by exploiting the fact that in the natural parametrization the score function is linear in the unknown $\bsy{\theta}$, allowing linear functions of this unknown to be unbiasedly estimated using quantities that do not depend on it. The approach taken in \cite{Eldar08} is unrelated to the methods we consider, and the results obtained are presently not subsumed by ours.

Assuming independence of the coordinates, \cite{Li85} considers consistency of SURE and of Stein shrinkage type estimators in the context of linear estimation, and makes an appealing link with Generalized Cross-Validation. Precise comparison with our results are presented in Remark \ref{rem:Li85}.
This present work makes the case that SURE can be extended beyond the currently known settings. Though unbiased for the Gaussian, SURE can be applied in many cases `as is' at the cost of a bias of order small enough to be able to, say, consistently choose good tuning parameters. In particular, we show that under our conditions SURE remains adaptive, in that the classical asymptotics of Pinsker's theorem for the Gaussian case still apply. We propose to distinguish this estimate by using the term ASSURE for the non-Gaussian cases where the procedure is Approximately the Same as SURE.

We verify properties of our proposed extensions using tools having their origins in Stein's method, in particular,
Stein kernels and the zero bias distribution.
We present a review of shrinkage and SURE in the Gaussian case, followed by background needed for the application of the methods we apply; technical results used for the zero bias technique in multi-dimension are compiled in Section \ref{sec:prop:mvariate.zbias}. In Section \ref{sec:shrinkage} we present a number of non-Gaussian models, to our knowledge not previously discussed in the literature, under which shrinkage is shown to be advantageous.  In Section \ref{sec_SURE} we assess SURE in non-Gaussian cases, concluding that the bias incurred by applying the estimate used in the Gaussian case can be small enough as to make this estimate useful in many instances. In addition, in 
Section \ref{eq:SureToSoftThresh} we consider the use of SURE for soft-thresholding, and in Section 
\ref{subsec:adaptivity} consider the adaptivity of the shrinkage estimator, and extensions of Pinsker's theorem, 
for non-Gaussian cases.  Section \ref{ssec_inv_norms} gives some technical results needed in Sections \ref{sec:shrinkage} and  \ref{sec_SURE} on the boundedness  in mean of inverse norms.

The main proofs of our results and 
some technical details on a few illustrative examples are presented in a section of
supplementary material at the end of this document. In addition, in part F of the supplement we give conditions under which our Stein kernel methods may be applied to functions other than the special case of ${\bm g}_0({\bm x})$ in 
\eqref{def:f.g0}
that is specific to shrinkage, and which is handled, for that case, by Assumption \ref{assumptionW.kernel}.

Though our results cover classes of distributions previously not treated in the literature, such non-elliptical, or discrete, distributions, some specific applications of our methods give an introduction to our results. For instance Examples \ref{ex:Student} and \ref{eq:student.zb} cover the multivariate Student $t$ distribution, Example \ref{ex:spherical} the uniform distribution on the sphere $S^{d-1}$ and with support over an ellipsoid, and
Example \ref{ex:noise.corruption} considers corrupted Gaussian observations. 

For instance,  
for ${\bm X}$ following a $d$-dimensional multivariate Student $t$ distribution with $k \ge 5$ degrees of freedom, having unknown mean $\bsy{\theta}$ and known covariance matrix $\Sigma$ with largest eigenvalue $\kappa$, for even dimensions $d=2m \ge 6$,  we apply two approaches, yielding the two different bounds 
\beas
 24 \lambda \sqrt{ \frac{2 d^2 (k+2)}{  (d-2)(d-4) 
 (d+k-2)  k (k-4)} }  \qmq{and}
 16 \lambda \frac{(d+k-2)}{(d-2)k}
\enas
on the {\it excess risk,} that is, the mean squared error above its value in  the Gaussian case.
The first bound, from Example  \ref{ex:Student}, uses a Stein kernel and Theorem \ref{thm:shrinkage.non_indep}, while the second, from Example  \ref{eq:student.zb}, applies a zero bias approach in conjunction with Theorem \ref{zerobiasshrinkage}. When 
$\lambda \in [0,2 (\operatorname{Tr}(\Sigma) -2 \kappa)]$ the risk of $S_\lambda$ is no larger than that of $S_0$ asymptotically in 
the first case when $kd \rightarrow \infty$, and for the second when $k \rightarrow \infty$ as $d \rightarrow \infty$. As in general it may be the case that only one of the two results applied here may be invoked, having access to both these approaches is advantageous even though the conclusions drawn for this particular case are nearly the same. 

In the following, densities of random vectors are with respect to Lebesgue measure, and when ${\bm X}$ has measure $\nu$ we will refer to the measure of ${\bm X}+{\bm \theta}$ as the translation of $\nu$ by ${\bm \theta}$.

\section{Stein's Identity and Two Extensions}\label{sec:SteinIdentity}
Stein's identity \cite{St81}, also known as Stein's lemma,  for characterizing the one dimensional normal distribution states that a random variable $X$ has law ${\cal N}(\theta,\sigma^2)$ if and only if \begin{align} \label{eq:Stein.Identity.d=1}
E[(X-\theta)f(X)] = \sigma^2 E[f'(X)] \qmq{for all $f \in {\mathcal F}$}
\end{align}
where ${\mathcal F}$ is the class of all real valued functions that are absolutely continuous on compact intervals, and for which the  expectation on the left hand side of  \eqref{eq:Stein.Identity.d=1} exists; an extension to $d$ dimensions is given in \eqref{eq:Stein.Identity.d}.
We consider two generalizations of Stein's lemma that will be used for the relaxation of the normal assumptions in Shrinkage and SURE.

One way that Stein's lemma may be generalized in one dimension for a mean zero random variable $X$ is through the use of a Stein kernel $\Tk$, a random variable for which \begin{align*}
E[Xf(X)]=E[T f'(X)]\qmq{for all $f \in 
{\mathcal F}$.}
\end{align*}
Stein kernels were first introduced in \cite{CaPa92}, and further developed in the univariate setting in \cite{Cha09}.
By replacing $T$ by $E[T|X]$ we may assume that $T$ is some function of $X$.

When $X$ has mean $\theta$ and $T_{X-\theta}$ is the Stein kernel of $X-\theta$, we obtain
\begin{multline*} 
E[(X-\theta)f(X)]=E[(X-\theta)f((X-\theta)+\theta)] \\
=E[T_{X-\theta}f'((X-\theta)+\theta)] = E[T_{X-\theta}f'(X)].
\end{multline*} 
Hence, if we are given $X=Y+\theta$ for an unknown $\theta$ and some mean zero random variable $Y$ with known distribution, we usually cannot compute the Stein kernel $T_{X-\theta}$ for $X-\theta$ without knowledge of $\theta$. However, under  natural assumptions we can get estimates on norms of $T_{X-\theta}$ that are uniform in $\theta$, for example as in the setting of \cite{CFP19}.

Another way to generalize Stein's lemma to non-Gaussian cases, following \cite{GR97} and \cite{D015}, is to use the fact that for any random variable $X$ with finite, non-zero variance $\sigma^2$ and mean $\theta$, the $X$-zero bias distribution $X^*$ exists, which is characterized by the condition that
\begin{align} \label{eq:def.non.zero.b}
E[(X-\theta)f(X)]=\sigma^2 E[f'(X^*)] \qmq{for all $f \in {\mathcal F}$.}
\end{align}
Hence, Stein's lemma \eqref{eq:Stein.Identity.d=1} can be restated as saying that the univariate normal distributions are the unique fixed points of the zero bias transformation that produces the distribution of $X^*$ from that of $X$. 

We highlight a relation between the zero bias distributions in the centered and non-centered cases by noting that if we define $X^*$ via \eqref{eq:def.non.zero.b} restricting to the case where $\theta=0$ then for the general case we obtain \eqref{eq:def.non.zero.b} by letting 
\begin{align} \label{eq:X*.gen.theta}
X^*:= (X-\theta)^* + \theta.
\end{align}
This distinction is important. With $=_d$ denoting equality in distribution, if we are given that $X$ is from a location family, specifically, that if $X=_d Y+ \theta$ where the distribution of some mean zero variable $Y$ is specified but $\theta$ is not, then, similar to this phenomenon for Stein kernels, though we may sample from $Y^*$, we are not able to sample from $X^*$ without knowledge of $\theta$.

The intricacy of the Stein kernel $T$ and the zero bias distribution is not illustrated in the normal case, as there $T$ is simply the variance, and the transformed distribution unchanged from the original, respectively. We now move on to multivariate generalizations of Stein kernels and the zero bias distribution.

\subsection{Multidimensional Stein kernels} \label{subsec:mult.dim.ker}
Stein kernels can be defined in the multivariate setting of vectors with dependent coordinates. This notion, originating in \cite{Cha09}, is at the core of the Nourdin-Peccati approach to Stein's method \cite{NP12}, making a powerful link to Malliavin calculus. Given a random vector ${\bm X} \in \R^d $ with mean $\bsy{\theta}$ and distribution $\nu$ which is absolutely continuous with respect to Lebesgue measure on $\R^d$, a Stein kernel ${\Tk}_{{\bm X} - \bsy{\theta}}$ for the mean-zero vector ${\bm X} - \bsy{\theta}$ is a matrix-valued function such that 
\bea \label{def:multivariate_stein_kernel}
E[\langle {\bm X} - \bsy{\theta}, {\bm f}({\bm X} ) \rangle ] = E [ \langle \Tk_{{\bm X} - \bsy{\theta}}, \nabla {\bm f}({\bm X})\rangle]
\qmq{for all ${\bm f} \in W^{1,2}(\nu)$.} 
\ena 
We remark that other works on Stein kernels may use other classes of test functions, see for example the discussion in \cite{MRS}. In addition, we only consider situations where the Stein kernel has a finite second moment, so that the right-hand side is always finite for test functions in $W^{1,2}(\nu)$.

As in the univariate setting, 
 the Stein characterization \eqref{eq:Stein.Identity.d} of the normal distribution translates as saying that a random vector ${\bm X}$ has a Gaussian distribution with mean $\bsy{\theta}$ and covariance matrix $\Sigma$ iff ${\bm X} - \bsy{\theta}$ admits $\Sigma$ as a Stein kernel.

Construction of multidimensional Stein kernels as in \eqref{def:multivariate_stein_kernel} has been considered, for example, in \cite{CFP19, MRS, Fathi19}. Existence and uniqueness of Stein kernels in higher dimensions are not in general guaranteed. In \cite{CFP19} it is shown that if $\nu$ is centered and satisfies a Poincar\'{e} inequality, then there is a Stein kernel that is the gradient of an element of $W_0^{1,2}(\nu)$, the set of functions in $W^{1,2}(\nu)$ with $\nu-$mean zero, and it is unique in that class. When the components of ${\bm Y}=(Y_1,\ldots,Y_d)$ are  independent, with mean zero,  finite variances and admit Stein kernels $T_i,i=1,\ldots,d$,  \cite[Example 3.9]{MRS} shows that the diagonal matrix ${\Tk}={\rm diag}(T_1,\ldots,T_d)$, 
satisfies \eqref{def:multivariate_stein_kernel} with $\bsy{\theta}=0$.

\begin{remark}\label{rem:isotropic} 
If $T_{\bm Y}$ is a Stein kernel for a centered isotropic random vector ${\bm Y}$ and $A$ is an invertible matrix, then
\begin{equation}\label{iso} \mathbb{E}[\langle A{\bm Y}, {\bm f}(A{\bm Y} )\rangle ] = \mathbb{E}[\langle AT_{\bm Y} A^\transpose, \nabla {\bm f}(A{\bm Y}) \rangle]\end{equation} 
so that  ${\bm y} \longrightarrow A T_{ {\bm Y}}(A^{-1}{\bm y})A^\transpose$ is a Stein kernel for $A{\bm Y}$.  In particular, this transformation with $A={\rm Cov}({\bm Y})^{-1/2}$ allows us to reduce many statements for non-isotropic random vectors ${\bm Y}$
 to the isotropic case, 
 as long as the covariance matrix of ${\bm Y}$ is invertible. 
\end{remark}

\begin{example}\label{elliptical} 
We say that an absolutely continuous $\mathbb{R}^d$ valued random vector ${\bm X}$ has a multivariate elliptical
distribution $E_d({\bm \theta}, \Upsilon, \phi)$ if it admits a density of the form
\begin{equation*} 
  p({\bm x}) = \kappa|\Upsilon|^{-1/2} \phi \left( \frac{1}{2}
    ({\bm x}-{{\bm \theta}})^\transpose\Upsilon^{-1}({\bm x}-{{\bm \theta}}) \right), \, {\bm x} \in \R^d,
\end{equation*}
for $\phi : [0,\infty) \to [0,\infty)$ a measurable function called a \emph{density
  generator}, ${\bm \theta} \in \R^d$ the location parameter, $\kappa$ the
normalising constant and $\Upsilon$ a symmetric positive definite
$d \times d$ dispersion matrix.  Here we assume that the model is chosen such that ${\bm \theta}$ is the mean vector and $\Upsilon$ is the covariance matrix $\Sigma$ of ${\bm X}$; see for example \cite{LVY15} for suitable conditions. 

 The cases $E_d(0, \Id, \phi)$ are the
   \emph{spherical distributions}. They are centred and isotropic and hence \eqref{iso} applies, so that 
   if $T_{\phi}$ is a Stein kernel for 
   $E_d(0, \Id, \phi)$  then 
 $T_{{\bm X}-{\bm \theta}} ({\bm x})= \Sigma^{1/2}  T_{{\phi}} (\Sigma^{-1/2} ( {\bm x} -{\bm \theta} ))\Sigma^{1/2}$ 
  is a Stein kernel for $E_d({\bm {\bm \theta}}, \Sigma, \phi)$, 
see also \cite{MRS}.
In particular, by \cite{MRS} and Proposition 2 in \cite{LVY15}, a Stein kernel for $E_d({\bm \theta}, \Sigma, \phi)$ is given by 
\begin{equation}
  \label{eq:elliptickernel}T_{{\bm X}-{\bm \theta}}({\bm x})  = \left( \frac{1}{\phi(({\bm x}-{\bm \theta})^\transpose\Sigma^{-1}({\bm x}-{{\bm \theta}})/2)}
\int_{({\bm x}-{\bm \theta})^\transpose\Sigma^{-1}({\bm x}-{{\bm \theta}})/2}^{+\infty}\phi(u) \mathrm{d}u \right) \Sigma.
\end{equation}
\end{example}

Moving forward, when considering the shrinkage estimator \eqref{eq:shrinkage} for non-Gaussian models using Stein Kernels, for integrability we will require, unless other conditions are explicitly mentioned, that the following assumption is in force: 
\begin{assumption}\label{assumptionWonly.kernel}
The function ${\bm g}_0({\bm x})={\bm x}/\|{\bm x}\|^2$ is an element of $W^{1,2}(\nu)$.
\end{assumption}
We note for later use that
\begin{align} \label{eq:Jacg0} 
\nabla {\bm g}_0({\bm x}) 
  =
\frac{1}{\|{\bm x}\|^2}{\Id}-\frac{2}{\|{\bm x}\|^4}{\bm x}{\bm x}^\transpose.
\end{align}

Lemma \ref{lem:ker.sat.assumption} provides the following simple sufficient condition for the satisfaction of Assumption \ref{assumptionWonly.kernel}.
Its proof is given in Supplement A.

\begin{lemma}\label{lem:ker.sat.assumption}
When $d \ge 5$, Assumption \ref{assumptionWonly.kernel} is satisfied by the measure  $\nu$ when it has a density bounded almost everywhere in some neighborhood of the origin, and by the translates of $\nu$ by any $\bm{\theta} \in \mathbb{R}^d$ when $\nu$ has a density bounded almost everywhere from above. \end{lemma}

\subsection{Multidimensional zero-bias transform}

For a given probability measure $\nu$ on $\mathbb{R}^d$, and probability measures $\nu^{ij},i,j=1,\ldots,d$ on  $\mathbb{R}^d$ depending on $\nu$, using notation to parallel  \eqref{def:W.norm.ker},
define the Sobolev-like norm and its corresponding function space respectively by
\begin{align*} 
||{\bm f}||_{W_z^{1,2}(\nu)}^2 :
= ||{\bm f}||_{L^2(\nu)}^2 + \sum_{i,j=1}^d ||\partial_j f_i||_{L^2(\nu^{ij})}^2 \,
\mbox{and} \,W_z^{1,2}(\nu)=\{{\bm f}: 
||{\bm f}||_{W_z^{1,2}(\nu)}^2
< \infty\}.
\end{align*}

The multivariate extension for zero biasing below in \eqref{eq:mzb.identity.gen.cov} complements the generalization from \cite{GR05} which takes an approach different from the one introduced here. For our extension, given a mean zero random vector ${\bm Y}$ in $\mathbb{R}^d$ with positive definite covariance matrix $\Sigma$ having entries $\sigma_{ij}={\rm Cov}(Y_i,Y_j)$, we say the collection of vectors $\{{\bm Y}^{ij}: \mbox{$i,j$ such that}\,\, \sigma_{ij} \not = 0\}$ in $\mathbb{R}^d$ has the multivariate ${\bm Y}$-zero bias distribution when
\begin{align} \label{eq:mzb.identity.gen.cov}
	E[\langle {\bm Y}, {\bm f}({\bm Y})\rangle] = E\left[\sum_{i,j=1}^d \sigma_{ij} \partial_j f_i({\bm Y}^{ij})\right]
	=: E[ \langle \Sigma, \nabla {\bm f}({\bm Y}^*) \rangle]\,\,\, \mbox{for all 
${\bm f} \in W_z^{1,2}(\nu)$,}
\end{align}
where in the second equality we define $\nabla {\bm f}({\bm Y}^*)$ to be the matrix with $i,j^{th}$ entry $\partial_j f_i({\bm Y}^{ij})$. When this identity is satisfied, we say that the zero bias vectors of ${\bm Y}$ exist.

Though in point \ref{prop:pos.mult.trans} of Proposition \ref{prop:mvariate.zbias}, and in 
Example \ref{ex:spherical} we consider the zero bias approach where \eqref{eq:mzb.identity.gen.cov} holds for a distribution with non-diagonal covariance matrix,
below, in view of Remark \ref{rem:isotropic}, we focus primarily on the case where $\Sigma={\rm diag}(\sigma_1^2,\ldots,\sigma_d^2)$, an invertible diagonal matrix. In this instance, the collection of zero bias vectors appearing in the identity \eqref{eq:mzb.identity.gen.cov} reduce to the $d$ vectors ${\bm Y}^i={\bm Y}^{ii}$, for which, now also letting $\sigma_i^2=\sigma_{ii}$, satisfy
\begin{align} \label{eq:mzb.identity.theta.zero}
E[\langle {\bm Y}, {\bm f}({\bm Y})\rangle] = E\left[\sum_{i=1}^d \sigma_i^2 \partial_i f_i({\bm Y}^i)\right] \qmq{for all 
${\bm f} \in W_z^{1,2}(\nu)$.}
\end{align}

Part \ref{prop:mvariate.zbias.ex.un} of Proposition \ref{prop:mvariate.zbias} in 
Section \ref{sec:prop:mvariate.zbias} shows that the zero bias vectors exist for any mean zero ${\bm Y}$ with non-singular diagonal covariance matrix if and only if 
\begin{align} \label{eq:mgalelike}
E[Y_i|Y_j, j \not = i]=0 \quad
\forall i=1,\ldots, d,
\end{align}
and, under this condition, provides a construction. Part \ref{prop:mvariate.zbias.support} of Proposition \ref{prop:mvariate.zbias} specifies the support of the zero bias vectors, Part \ref{prop:mvariate.zbias.sum} considers independent sums, 
Part \ref{prop:mvariate.zbias.mix} handles mixtures, 
Part \ref{prop:pos.mult.trans} provides the existence of the zero bias vectors for a class of distributions with non-diagonal covariance matrices and Part \ref{prop:mvariate.zbias.density} considers the special case where ${\bm Y}$ possesses a density function.

Generally, to encompass vectors ${\bm X}$ with arbitrary means $\bsy{\theta}$, extending \eqref{eq:def.non.zero.b} and 
\eqref{eq:X*.gen.theta},  \eqref{eq:mzb.identity.gen.cov}
implies that
\begin{align}\label{eq:mzb.identity}
E[\langle {\bm X}-\bsy{\theta}, {\bm f}({\bm X})\rangle] = E\left[\sum_{i,j=1}^d \sigma_{ij} \partial_j f_i({\bm X}^{ij})\right]
\qmq{for ${\bm X}^{ij}=({\bm X}-\bsy{\theta})^{ij}+\bsy{\theta}$.}
\end{align}

\begin{example} 
Following on from Example \ref{elliptical}, for ${\bm Y} \sim E_d(0, \Id, \phi)$ 
it is shown implicitly in Proposition 2 of \cite{LVY15} that the collection ${\bm Y}^i =  {\bm Y}^* \sim  E_d(0, \Id, \Phi)$, for $i=1, \ldots, d$
has the ${\bm Y}$-zero bias distribution, where 
$\Phi(x) = \int_x^\infty \phi (u) du$. Example \ref{ex:spherical}, which considers the spherical distribution resulting by placing uniform measure on the surface of a sphere, is not covered by the referenced results as it is not absolutely continuous with respect to Lebesgue measure on $\R^d$. 
\end{example}

When $\sigma_{ij} \ge 0$ for all $1 \le i,j \le d$, the right hand side of \eqref{eq:mzb.identity.gen.cov} may be written more compactly as a mixture via the use of a pair of random indices $(I,J)$, independent of $\{{\bm Y}^{ij},i,j=1,\ldots,d\}$, with distribution 
\begin{align*}
P(I=i,J=j) = \frac{\sigma_{ij}}{\sigma^2} \qmq{where} \sigma^2 = {\rm Var}\left(\sum_{i=1}^d Y_i \right) = \sum_{i,j=1}^d \sigma_{ij}.
\end{align*}
Then, starting with the first equality of \eqref{eq:mzb.identity.gen.cov}, we obtain
\begin{multline*} 
E[\langle {\bm Y}, {\bm f}({\bm Y})\rangle] = E\left[\sum_{i,j=1}^d \sigma_{ij} \partial_j f_i({\bm Y}^{ij})\right] 
\\= \sigma^2 E\left[\sum_{i,j=1}^d P(I=i,J=j) \partial_j f_i({\bm Y}^{ij})\right] = \sigma^2 E[\partial_J f_I({\bm Y}^{IJ})].
\end{multline*}
In particular, taking $g({\bm y})=\sum_i y_i$, the sum of the coordinates of ${\bm y} \in \mathbb{R}^d$, $W=g({\bm Y})$ and ${\bm f}({\bm y}) = (f(g({\bm y})),\ldots,f(g({\bm y})))$ for smooth $f$ yields
\begin{align} \label{eq:EWfW.from.mult}
E[Wf(W)]=E[\langle {\bm Y}, {\bm f}({\bm Y})\rangle] = \sigma^2 E[f'(W^{IJ})],
\end{align}
demonstrating that $W^{IJ}$, the sum of the coordinates of ${\bm Y}^{IJ}$, has the $W$-zero biased distribution. We note that the condition that $\sigma_{ij}$ be non-negative always holds when ${\bm Y}$ has a diagonal covariance matrix.

As was done for Stein kernels by Assumption \ref{assumptionWonly.kernel} here we shall adopt Assumption \ref{assumptionWonly.zero_bias} to guarantee that the zero bias Stein identity can be applied to the function ${\bm g}_0({\bm x})$ that is used in the shrinkage estimator.

\begin{assumption}\label{assumptionWonly.zero_bias} The function
${\bm g}_0({\bm x})={\bm x}/\|{\bm x}\|^2$ is
an element of $W_z^{1,2}(\nu)$.
\end{assumption}

Similar to the sufficient condition provided by Lemma \ref{lem:ker.sat.assumption} for the satisfaction of Assumption \ref{assumptionWonly.kernel} when applying kernels, here we provide some simple conditions that guarantee the validity of Assumption \ref{assumptionWonly.zero_bias}; we restrict to the diagonal covariance case. The proof of Lemma \ref{lem:zb.sat.assumption} can be found in Supplement A.

\begin{lemma}\label{lem:zb.sat.assumption}
Let $d \ge 5$ and $\nu$ the measure of a mean zero distribution with finite second moment,
that satisfies \eqref{eq:mgalelike}. Then
Assumption \ref{assumptionWonly.zero_bias} is satisfied for the measure $\nu$ and all its translates when $\nu$
has a density $p({\bm y})$ such that for each $i=1,\ldots,d$ there exists an $L^1$ function $g_i$ such that 
$|y_i|p({\bm y}) \le g_i(y_i)$ for all ${\bm y} \in \mathbb{R}^d$.

In addition, for any $d \ge 2$, letting 
\begin{align} \label{eq:X.neg.i}
{\bm x}^{\neg i}=(x_1,\ldots,x_{i-1},x_{i+1}, \ldots, x_d),
\end{align}
Assumption \ref{assumptionWonly.zero_bias} also holds when there exists some positive $\delta$ such that the supports of $\nu$ and $\nu^i,i=1,\ldots,d$ have empty intersection with a ball around the origin of radius $\delta$, or if the support $S$ of $\nu$ satisfies 
\begin{align} \label{eq:x.neg.i.bdbl.2delta}
S \subset \bigcap_{i=1,\ldots,d} \{{\bm x}:\|{\bm x}^{\neg i}\|_\infty \ge 2\delta \} = \{{\bm x}: \exists k \not = l \,\,|x_k| \ge 2\delta, |x_l| \ge 2 \delta\},
\end{align}
that is, if every element of the support has at least two non-zero coordinates whose absolute values are larger than some (uniform) constant.
\end{lemma}

\begin{example} \label{ex:finite.support} 
When the distribution of ${\bm Y}$ has support $S_{\bm Y}$ then \eqref{eq:x.neg.i.bdbl.2delta} is satisfied for the support $S$ of ${\bm X}={\bm Y}+{\bm \theta}$ whenever for some $\delta>0$,
\begin{align*} 
{\bm \theta} \in \bigcap_{i=1}^d \left\{
{\bm \psi}: \| {\bm \psi}^{\neg i} + {\bm y}^{\neg i}\|_\infty \ge 2 \delta\,\, \forall {\bm y} \in S_{\bm Y}
\right\}.
\end{align*}
When $S_{\bm Y}$ is finite the collection of shifts ${\bm \theta}$ that are excluded can be written as a finite union of sets whose probability with respect to any given absolutely continuous probability measure can be made arbitrarily small by choice of $\delta$. 
\end{example}

\begin{example} To see the importance of conditions, such as \eqref{eq:x.neg.i.bdbl.2delta}, that require 
the supports of $\nu^i$ to be bounded away from zero, 
let ${\bm Y}$ have the uniform distribution on $ \left\{(-1,0),(1,0),(0,-1),(0,1)\right\}$. Then ${\bm Y}^1=_d (U,0)$  where $U \sim \mathcal{U}[-1,1]$, and specializing identity \eqref{eq:mzb.identity.theta.zero} to the case ${\bm f}=({\bm g}_{0,1},0)$ with ${\bm g}_0$ as in \eqref{def:f.g0}, we obtain
\begin{multline*}
E[\sigma_1^2 \partial_1 {\bm g}_{0,1} ({\bm Y}^1)]=E[Y_1 {\bm g} _{0,1} ( {\bm Y} )], \qmq{which produces}\\
\frac{1}{4} \int_{-1}^{1} \partial_1 {\bm g}_{0,1} (u,0) du =  \frac{1}{4} ( {\bm g} _{0,1} (1,0) -  {\bm g} _{0,1} (-1,0)).
\end{multline*}
But the latter identity can not hold since the left-hand side, with the integral over $[-1,1]$ of $-1/u^2$ being infinite, and the right hand side taking the value $1/2$. This example demonstrates that Assumption \ref{assumptionWonly.zero_bias} may not hold when the support of $\nu^i$ includes zero for some $i$.
\end{example}

\subsection{Models} 

We end this section with a discussion of the various models and to provide some context for the assumptions made for the observation error $\bm{Y}$. 

\begin{model}[Log-concave]\label{model:log concave}
With $\phi$ denoting a convex function, a (non-degenerate) vector ${\bm Y}$ is log-concave when it has a density of the form $\exp(-\phi)$ with respect to the Lebesgue measure;  $\phi$ is called the potential. The vector ${\bm Y}$ is strictly log-concave if the function $\phi$ is strictly convex, and strongly log-concave when $\phi$ is strongly convex. When the potential $\phi$ is two times differentiable, strong log-concavity amounts to the assumption that ${\rm Hess}(\phi)({\bm x})\geq \sigma^{-2} {\rm Id}$ for all ${\bm x} \in \mathbb{R}^d$ and for some $\sigma^2>0$. Another characterization of strong log-concavity cooreponds to assuming the existence of a density $p=\exp(-\psi)\gamma_{\sigma^2}$, where $\psi$ is a convex function and $\gamma_{\sigma^2}$ is the density of a Gaussian vector with mean zero and covariance $\sigma^2 {\rm Id}$ (\cite{SauWell}).

This class of measures generalizes uniform measures over convex sets to a non-uniform setting. They include of course Gaussian measures, but also many other examples, such as exponential, logistic and Gamma distributions. They are a common class of distributions where one can seek to generalize properties of Gaussian distribution, in particular in high-dimensional settings. They play a role in several areas of applied mathematics, such as convex optimization and optimal transport theory . We refer to \cite{SauWell} for a survey, and to \cite{Vem10} for a panorama of applications. Our work here shall use recent developments in the study of high-dimensional log-concave measures, such as the recent (almost) solution to the KLS conjecture \cite{Chen21, LV17b}, to prove that shrinkage estimation works nicely in this setting.  
\end{model}

\begin{model}[Unconditional] \label{model:unconditional}
Condition \eqref{eq:mgalelike} shares a strong formal similarity with the 
notion of a martingale difference random field \cite{NaPe92}, although in the latter case, the indices of the collection of random variables live in a lattice. 
A main class of examples of martingale difference random fields depend on ``superparity potentials''
\cite[Definition 3]{NaPe92}.

Superparity is related to the term ``unconditional'' that appears in literature linked to high-dimensional geometry (see e.g. \cite{BoNa03}).
 We say the vector ${\bm Y}=(Y_1,\ldots,Y_d)$ is {\em unconditional}  when it satisfies
\begin{align*}
(\epsilon_1 Y_1,\ldots,\epsilon_d Y_d) =_d (Y_1,\ldots,Y_d)
\qm{for all $(\epsilon_1,\ldots,\epsilon_d) \in \{-1,1\}^d$};
\end{align*}
such vectors are easily seen to have a diagonal covariance matrix and to satisfy \eqref{eq:mgalelike}.

Zero biasing of unconditional vectors was considered in \cite{Go07} and
\eqref{eq:EWfW.from.mult} recovers a construction used there for this special case.
Such vectors arise naturally as the uniform distribution over bodies in $\mathbb{R}^d$
that have a high degree of symmetry, and include spherically-symmetric distributions. However, Equation \eqref{eq:mgalelike} can hold for many non-elliptical examples. Up to an affine transformation, an elliptical distribution has a density of the form $p(||x||_2)$, while, subject only to the existence of the necessary conditional expectations, Equation \eqref{eq:mgalelike} holds for \emph{any} distribution with density of the form $p(|x_1|,...,|x_d|)$, for $p$ with argument in $\mathbb{R}_+^d$ instead of just $\mathbb{R}_+$. In particular,  there is no need to have an underlying $\ell^2$ structure. An already relevant class of measures are those with a density that is a function of some $\ell^p$ norm. These distributions are elliptical if and only if $p = 2$, while Condition \eqref{eq:mgalelike} applies to any value of $p$, including $p=\infty$.

Furthermore, Poincar\'e inequalities for unconditional distributions have been investigated, for example in \cite{Kla13}, which proved, under an additional assumption of log-concavity, that the Poincar\'e constant scales at most logarithmically in the dimension; this result allows one to bound the squared Stein discrepancy $E[||\Tk - \Sigma||_{\HS}^2]$ in the kernel approach (see also \cite[Proposition 2.21]{CG19} for a simplified proof); this quantity has been used as a measure of how far away a given distribution is from the Gaussian \cite{LNP15,NP12}, typically in the case $\Sigma = \Id$.
A further relevant feature worthy of note here is that if we assume in addition that the distribution of ${\bm Y}$ is strictly log-concave, then the maximum likelihood estimator of ${\bm \theta}$ is ${\bm X}$, as in the Gaussian case.
\end{model}

\begin{model}[Mixture]\label{model:indep} The most basic multidimensional model for a mean zero random vector ${\bm Y}$ that satisfies the conditional expectation condition \eqref{eq:mgalelike}
is the one that assumes independence among coordinates; for this model, a Stein kernel was discussed above Remark \ref{rem:isotropic}. 
For zero biasing, taking ${\bm Y}$ to have a non-singular covariance matrix to exclude trivial cases, one may easily verify that the vectors 
\bea \label{eq:defY^i.indep}
{\bm Y}^i = (Y_1,\ldots,Y_{i-1},Y_i^*,Y_{i+1},\ldots,Y_d) \qm{for $i=1,\ldots,d$}
\ena 
satisfy \eqref{eq:mzb.identity.theta.zero},
where $Y_i^*$ is independent of $\{Y_j, j \not =i\}$, and has the $Y_i$-zero biased distribution. In particular, ${\bm Y}$ and ${\bm Y}^i$ may be put on a joint space by specifying that $Y_i^*$ is independent of $Y_j, j \not = i$, and fixing any coupling for $(Y_i,Y_i^*)$. 

The independent model can be 
be extended through the use of mixtures, as follows.
Let  $(\mathcal{S},\Sigma)$ be a measurable space
and let $\{m_s\}_{s \in \mathcal{S}}$ be a collection of probability measures on $\mathbb{R}^d$ such that for each Borel
subset $A \subset \mathbb{R}^d$ the function $s \rightarrow m_s(A)
$ from $\mathcal{S}$ to $[0,1]$ 
is measurable. When $\mu$ is a probability (mixing) measure on
$(\mathcal{S},\Sigma)$, the set function given by
\begin{align*} 
m_\mu(A)=\int_{\mathcal{S}} m_s(A)\mu(d s)
\end{align*}
is a probability measure, called the $\mu$ m of $\{m_s\}_{s \in \mathcal{S}}$. We may also refer to this distribution as the $\mu$ mixture of the distributions of random variables ${\bm X}_s \sim m_s, s  \in \mathcal{S}$, and write the equivalent identity 
\begin{align}\label{eq:mix.s.over.mu}
E[f({\bm X})] = \int_{\mathcal{S}} E_s[f({\bm X}_s)]d\mu
\end{align}
for real valued, bounded continuous functions $f$ on $\mathbb{R}^d$, and $E_s$ denoting expectation with respect to $m_s$.

Using such mixtures we extend 
 the basic independent coordinate model to the case where
the vector ${\bm X}$ is of the form ${\bm Y}+\bsy{\theta}$ for some unknown $\bsy{\theta} \in \mathbb{R}^d$ and ${\bm Y}$ is the $\mu$ mixture of the mean zero distributions ${\bm Y}_s, s \in \mathcal{S}$ with non-singular covariance matrices ${\bm \Sigma}_s, s \in  \mathcal{S}$, each satisfying \eqref{eq:mgalelike}, and where $\Sigma={\rm Var}({\bm Y})$ is finite.
This extension gives
 rise to a wide collection of distributions that generally have dependent coordinates. 
It also handles the case where a Gaussian observation has positive probability of being corrupted by noise, see Example \ref{ex:noise.corruption}.

In particular, under this model, condition \eqref{eq:mgalelike} holds for ${\bm Y}$, and in the special case where ${\bm \Sigma}_s$ does not depend on $s \in \mathcal{S}$ 
then Proposition \ref{prop:mvariate.zbias} shows that for $i=1,\ldots,d$, the distribution of the zero bias vector ${\bm Y}^i$ exists, and is simply the $\mu$ mixture of ${\bm Y}_s^i$.
Included are cases where for all $s \in \mathcal{S}$ the $d$ components of ${\bm Y}_s$ are independent with variance not depending on $s$, 
in which case the zero bias vectors for the distributions in the mixture can be constructed as in \eqref{eq:defY^i.indep}.

It is possible to build a Stein kernel for mixtures of centered measures that each have a Stein kernel. In particular, when $T_s, s \in \mathcal{S}$ is a Stein kernel for ${\bm Y}_s$, then it is easy to see that $T$ is a Stein kernel for ${\bm Y}$ when $({\bm Y} ,T)$ is the $\mu$ mixture of $({\bm Y} _s,T_s)$.  However, here we shall mostly discuss mixtures in the context of zero-bias transforms. 

\end{model}

\section{Shrinkage for non-Gaussian models} \label{sec:shrinkage}
Consider the shrinkage estimator
\begin{align} \label{eq:shrinkage.est}
S_\lambda({\bm X}) = {\bm X}\left(1 -
\frac{\lambda}{||{\bm X}||^2} \right) ,
\qmq{$\lambda \ge 0$}
\end{align}
of an unknown mean $\bsy{\theta}$ of a random vector ${\bm X} \in \mathbb{R}^d$. We have $S_\lambda({\bm X})={\bm X}+{\bm f}({\bm X})$, dropping the dependence of ${\bm f}$ on $\lambda$. Using \eqref{eq:Jacg0},
\begin{align} \label{eq:f.jacobian.f.shrinkage}
{\bm f}({\bm x})=-\lambda\frac{{\bm x}}{\|{\bm x}\|^2}
\qmq{satisfies}
\nabla {\bm f}({\bm x}) 
  =
 -\lambda \left( \frac{1}{\|{\bm x}\|^2}{\Id}-\frac{2}{\|{\bm x}\|^4}{\bm x}{\bm x}^\transpose \right).
\end{align}

To explore the mean squared error of $S_\lambda({\bm X})$, when the corresponding expectations exist, expansion yields
\bea
	E_{\bsy{\theta}} || S_\lambda({\bm X})-\bsy{\theta}||^2
	 \label{eq:MSE.expansion.3.1+3.2}
	=E_{\bsy{\theta}}\left\{|| {\bm X}-\bsy{\theta} ||^2 - 2\lambda \left\langle {\bm X}-\bsy{\theta},\frac{{\bm X}}{\|{\bm X}\|^2} \right\rangle  +
	\frac{\lambda^2}{||{\bm X}||^2} \right\},
\ena 
where we now emphasize the goal of estimating the unknown mean $\bsy{\theta}$ of ${\bm X}$  by
including it as a subscript.
The mean squared error of $S_0({\bm X})$ is given by the first term. Thus $S_\lambda({\bm X})$ has smaller mean squared error than  $S_0({\bm X})$ if and only if the sum of the expectations of the two last terms is negative. 
We apply the two extensions of the Stein identity in Section \ref{sec:SteinIdentity}
to reformulate the expectation 
$E_{\bsy{\theta}} \langle {\bm X}-\bsy{\theta},{\bm f}({\bm X}) \rangle$ of the second term, namely, using the Stein kernel identity \eqref{def:multivariate_stein_kernel}, and the multidimensional zero-bias transform identity \eqref{eq:mzb.identity}. These two approaches yield qualitatively similar results, though they depend on different properties of the ${\bm X}$ distribution. Exploring both approaches, we show how the use of $S_\lambda$ provides advantages for estimating the mean in some non-Gaussian settings. 

\subsection{Stein Kernels}
In this subsection we present three theorems that offer  generalizations of the Gaussian case, under different assumptions on the Stein kernels of the distribution $\nu$. In Theorem \ref{thm_shrinkage_SLC} it is assumed that $\nu$ admits a Stein kernel that is uniformly bounded from above and below.  Theorem \ref{thm:shrinkage.non_indep} holds for general distributions with positive definite covariance matrices, while Theorem \ref{thm:shrinkage.poinca} addresses distributions for which a Poincar\'{e} inequality holds. The proofs of these results are deferred to the end of this subsection.

Considering non-isotropic random vectors, Theorem \ref{thm_shrinkage_SLC} that follows offers a non-parametric generalization of the Gaussian case, and does not require Assumption \ref{assumptionWonly.kernel}.

\begin{theorem} \label{thm_shrinkage_SLC}
Consider the measure $\nu$ of a random vector ${\bm X}-\bsy{\theta}$ with mean zero such that
$E_{\bsy{\theta}}\left[ \|{\bm X}\|^{-2} \right]<\infty$. Assume that $\nu$ admits a Stein kernel $T$ that is uniformly bounded from below and above, in the sense that $\alpha_-{\rm Id}\leq T \leq \alpha_+{\rm Id}$ $a.s.$ for the usual partial ordering of symmetric matrices, with $\alpha_-$ and $\alpha_+$ positive constants. Then
\bea \label{eq:risk_T_infinity}
E_{\bsy{\theta}}|| S_\lambda ({\bm X}) - \bsy{\theta}||^2   &\le & E|| {\bm X} - \bsy{\theta}||^2 
- \lambda  E_{\bsy{\theta}} \left[\frac{1}{|| {\bm X}||^2} \right] \left(2d\alpha_- - 4\alpha_+ - \lambda \right),
\ena
and if $d\geq 1+\lfloor 2\alpha_+/\alpha_-\rfloor$ and $\lambda \in (0, 2d\alpha_- - 4\alpha_+)$, then the shrinkage estimator $S_\lambda$ has a smaller risk than the least-squares estimator $S_0$.
In particular, if $\nu$ is the measure
of a strongly log-concave random vector ${\bm X}-\bsy{\theta}$ with full support, mean zero and twice differentiable potential $\phi$ satisfying $c_+ {\rm Id} \geq {\rm Hess}(\phi)({\bm x}) \geq c_{-} \, {\rm Id}$ for any ${\bm x}\in \mathbb{R}^d$ for some positive constants $c_{+},c_{-} > 0$,  then $E_{\bsy{\theta}}\left[ \|{\bm X}\|^{-2} \right]<\infty$ for $d\geq 3$ and \eqref{eq:risk_T_infinity} holds with $\alpha_-=1/c_+$ and $\alpha_+=1/c_-$.
\end{theorem}

This result, which proof is detailed in Supplement B, has three main features. First, the classical result for $d\geq 3$ and ${\bm X} - \bsy{\theta}$ having a normal distribution is recovered from (\ref{eq:risk_T_infinity}) with $\alpha_-=\alpha_+=1$ (or $c_+=c_-=1$ in the strongly log-concave case), since in this case one can take $T=\Sigma={\rm Id}$. Second, no condition is needed concerning the behavior of $\|\bsy{\theta}\|$ with respect to dimension $d$. Third, the result in the strongly log-concave case shows that, for sufficiently large dimensions, there exist shrinkage estimators that are asymptotically better
than the least-squares estimator (which is also the MLE for an unconditional strictly log-concave noise vector), even in non-isotropic situations and without knowledge of the covariance matrix, or of any need to estimate it. Moreover, we do not require any form of symmetry, unlike previous
results on elliptical distributions. 

The next result, Theorem \ref{thm:shrinkage.non_indep}, is much more general than Theorem  \ref{thm_shrinkage_SLC}, and demonstrates that shrinkage can  improve
the MSE by providing a bound involving a term that depends on $\lambda$ that can be negative, plus a term $2 B_\lambda$ that measures the discrepancy between the given distribution and the Gaussian, which will be of smaller order under certain conditions. We formalize some conditions under which shrinkage is to advantage in the following remark. 

\begin{remark} \label{rem:shrinkage.succeeds}
For ${\bm X}$ with mean $\bsy{\theta}$ and covariance matrix $\Sigma$, the risk of the estimator $S_0$ is ${\rm Tr}(\Sigma)$, which is the first term in the bound 
\eqref{eq:Slambdarisk.upper.bound.quantitative} below, while 
the second term becomes
\begin{align} \label{eq:shrink.gain.gen}
-E_{\bsy{\theta}}\left[
\frac{({\rm Tr}(\Sigma)-2\kappa)^2}{\|{\bm X}\|^2} \right] \qmq{when} \lambda={\rm Tr}(\Sigma)-2\kappa.
\end{align}
When 
\begin{align} \label{eq:conditions.for.shrinakge}
\|\bsy{\theta}\|^2=O({\rm Tr}(\Sigma)) \qmq{and}
\kappa = o({\rm Tr}(\Sigma))
\end{align}
then using Jensen's inequality to obtain
\begin{align}\label{eq:Jen.Enorm.sq}
E_{\bsy{\theta}}\left[ \frac{1}{\|{\bm X}\|^2} \right] \ge \frac{1}{\|\bsy{\theta}\|^2+\operatorname{Tr}(\Sigma)},
\end{align}
we see that \eqref{eq:shrink.gain.gen} is negative with absolute value at least on the order of ${\rm Tr}(\Sigma)$. Hence, shrinkage with this value of $\lambda$ will improve the mean squared error over that of $S_0$ when $B_\lambda=o({\rm Tr}(\Sigma))$. Similar remarks apply in general when $\lambda/2({\rm Tr}(\Sigma)-2\kappa)$ is bounded away from 0 and 1. In the canonical case where growth is of order $d$, the set of conditions
\begin{align}\label{bds.theta.traceSigma}
\|\bsy{\theta}\|^2 + {\rm Tr}(\Sigma)=O(d) \qmq{and} d=O({\rm Tr}(\Sigma)-2\kappa)
\end{align}
are equivalent to the two in \eqref{eq:conditions.for.shrinakge} along with the additional assumption that ${\rm Tr}(\Sigma)/d$ is bounded away from zero and infinity. 
\end{remark}

\begin{theorem} \label{thm:shrinkage.non_indep}
Let a random vector ${\bm X}$  have mean $\bsy{\theta}$, positive semidefinite  covariance matrix $\Sigma$ with largest eigenvalue $\kappa$,
and Stein kernel $\Tk= \Tk_{{\bm X} - {\bf \bsy{\theta}}}$. 
Then
\bea  \label{eq:Slambdarisk.upper.bound.quantitative}
{E_{\bsy{\theta}}|| S_\lambda ({\bm X}) - \bsy{\theta}||^2 }  &\le & E|| {\bm X} - \bsy{\theta}||^2 
+  E_{\bsy{\theta}} \left[\frac{\lambda}{|| {\bm X}||^2} \left(\lambda - 
 2 \left( \operatorname{Tr}(\Sigma) - 2 \kappa \right) \right) \right] 
 + 2 B_\lambda
\ena 
where
\begin{align} \label{eq:defB}
B_\lambda= |E_{\bsy{\theta}}[\langle \Sigma-\Tk_{{\bm X}-{\bsy{\theta}}},\nabla {\bm f}({\bm X})\rangle]|
\end{align}
with ${\bm f}$ as in \eqref{eq:f.jacobian.f.shrinkage}, and
\begin{align} \label{eq:upper.bound.B.lambda}
B_\lambda \le
\frac{\lambda}{d} \sqrt{E_{\bsy{\theta}}[d^2\|{\bm X}\|^{-4}]} \left\{  \sqrt{\operatorname{Var}(\operatorname{Tr}(\Tk))}+2
\sqrt{ E[\| \Tk-\Sigma\|^2]} \right\}.
\end{align}
If for some $d_0$
\bea \label{eq:3conditions.thm:shrinkage.non_indep}
\,\sup_{d \geq d_0}E_{\bsy{\theta}}[d^2||{\bm X}||^{-4}] < \infty,  \hspace{1mm} \operatorname{Var}(\operatorname{Tr}(\Tk)) = o(d^2)\,\,\mbox{and}\,\, E[||\Tk - \Sigma||_{\HS}^2] = o(d^2)
\ena
and $\lambda =O(d)$,
then $B_\lambda=o(d)$, and over the range  $\lambda \in [0,2 (\operatorname{Tr}(\Sigma) -2 \kappa)]$ the risk of $S_\lambda$ is no larger than that of $S_0$ asymptotically. If in addition $\lambda/2(\operatorname{Tr}(\Sigma) -2 \kappa)$ is bounded away from zero and \eqref{bds.theta.traceSigma} holds
then the shrinkage estimator \eqref{eq:shrinkage.est}
has strictly smaller mean squared error than $S_0({\bm X})={\bm X}$ for all $d$ sufficiently large.
\end{theorem}

The proof of Theorem \ref{thm:shrinkage.non_indep} can be found in Supplement B. The quantity $\operatorname{Tr}(\Sigma) -2 \kappa$ in \eqref{eq:Slambdarisk.upper.bound.quantitative} is used to bound the trace of a matrix product that appears in the proof, and could be refined at the cost of an expression involving higher moments. In the special case when $\Sigma=  \sigma^2 \Id$,  \eqref{eq:Slambdarisk.upper.bound.quantitative} simplifies to
\bea \label{eq:Blambda.sigma^2I}
{E_{\bsy{\theta}}|| S_\lambda ({\bm X}) - \bsy{\theta}||^2 }  &\le & E|| {\bm X} - \bsy{\theta}||^2 
+  E_{\bsy{\theta}} \left[\frac{\lambda}{|| {\bm X}||^2} \left(\lambda - 2 \sigma^2 (d-2) 
\right)  \right] + 2 B_\lambda. 
\ena
For ease of notation, the dependence of the bound $B_\lambda$ on the negative moments of $\|{\bm X}\|^2$ and the choice of kernel $T$ is suppressed.

One condition needed for some of the results that follow is that for some $m \ge 1$ there exists a constant $C$ such that
\begin{align} \label{eq:inv.mom.Sd}
a)\, E_\theta \left[ \frac{d}{\|{\bm X}\|^2} \right]^m \le C \qmq{or} b)\, \max_{1 \le i \le d}E_{\bsy{\theta}}\left( \frac{d}{\| {\bm X}^{\neg i}\|^2} \right)^m \le C,
\end{align} 
where we recall the notation in \eqref{eq:X.neg.i} for part b). These conditions are handled in Section \ref{ssec_inv_norms}, 
and shown to be satisfied, for instance, under 
log-concavity in Proposition \ref{eq:prop.log.concave}, 
the technical condition \eqref{eq:MSd.bound} given in Lemma \ref{lem:inverse.mean} and discussed in Remark \ref{rem:gamma.inverse.moment}. 

Lemma \ref{lem:ker.sat.assumption}, that gives conditions under which
Assumption \ref{assumptionWonly.kernel} holds, requires us to work in dimension at least 5, while the critical dimension in the Gaussian case is just 3. Under suitable integrability conditions, one would expect the critical dimension implicit  in Theorem 
\ref{thm:shrinkage.non_indep} to decrease to 3 as the sample size $n$ in Example \ref{ex:take.average.n} tends to infinity. However, to ensure that the shrinkage estimator is allowed in the Stein identity with only the required weaker integrability condition, one would have to only use $L^1$ estimates on its gradient, instead of $L^2$,
which by duality would require us to work with $L^{\infty}$ Stein kernels, as in \ref{thm_shrinkage_SLC}. 
But we do not expect bounded Stein kernels to exist for simply log-concave distributions for example, and so in that more general setting we shall assume $d \geq 5$.

Applying Theorem \ref{thm:shrinkage.non_indep} to Model \ref{model:indep} we see how shrinkage may be to advantage in non-Gaussian situations. A short proof of Corollary \ref{cor:shrinkage.via.T.univ} is given in Supplement B, illustrating how he conditions of Theorem \ref{thm:shrinkage.non_indep} can be easily verified.
\begin{corollary} \label{cor:shrinkage.via.T.univ}
Suppose that ${\bm X} \in {\mathbb R}^d$ satisfies the conditions of Model \ref{model:indep}, where the components $Y_{s,i}$  of ${\bm Y}_s$ are independent and $\Sigma_s=\sigma^2 \Id$ for all $s \in S$, the Stein kernels $T_{s,i}$ of the components ${\bm Y}_s$ satisfy
$\sup_{s \in S, 1 \le i \le d} E[T_{s,i}^2]< \infty$, and that for all $s \in S$,   \eqref{eq:inv.mom.Sd}a with $d=2$ is satisfied by $\|{\bm X}_s\|=\|{\bm Y}_s+\bsy{\theta}\|$.
Then if $\lambda=\sigma^2(d-2)$ and $\|\bsy{\theta}\|^2 = O(d)$, 
the shrinkage estimator \eqref{eq:shrinkage.est}
has strictly smaller mean squared error than $S_0$ for all $d$ sufficiently large.
\end{corollary}

\begin{example} \label{ex:take.average.n}
If $\Tk_{{\bm X}_i-\bsy{\theta}},i=1,\ldots,n$ are Stein kernels for an 
independent sample of vectors ${\bm X}_i,i=1,\ldots,n$, each with mean $\bsy{\theta}$ and covariance matrix $\sigma^2 \Id$, then, as noted in \cite{CaPa92} in one dimension, a Stein kernel $\Tk$ for their average 
\beas
{\bm X} = \frac{1}{n}\sum_{i=1}^n {\bm X}_i \qmq{is given by} \Tk=\frac{1}{n}\sum_{i=1}^n \Tk_{{\bm X}_i-\bsy{\theta}} ,
\enas
since, by independence,
\begin{multline*} 
E_{\bsy{\theta}}[\langle {\bm X} - \bsy{\theta}, {\bm f}({\bm X} ) \rangle ] 
 = \left[ \frac{1}{n}\sum_{i=1}^n E_{\bsy{\theta}} \langle {\bm X}_i - \bsy{\theta}, {\bm f}({\bm X}) \rangle \right]\\
  = \frac{1}{n}\sum_{i=1}^n E_{\bsy{\theta}}[\langle \Tk_{{\bm X}_i-\bsy{\theta}},\nabla {\bm f}({\bm X})\rangle ]= E_{\bsy{\theta}}[\langle \Tk,\nabla{\bm f}({\bm X})\rangle ].
\end{multline*}

Under conditions on the measures of the average, results of \cite{CFP19} guarantee existence and uniqueness of Stein kernels within  $W_{\nu,0}^{1,2}$, the set of functions in $W^{1,2}(\nu)$ with $\nu-$mean zero.
As $E[\Tk_i]=\operatorname{Var}({\bm X}_i)$ via \eqref{def:multivariate_stein_kernel}, we see that $\operatorname{Var}(\operatorname{Tr}(\Tk))$ and $E\|\Tk-\sigma^2 \Id \|^2_{\HS}$, and hence the quantities in the last two conditions of \eqref{thm:shrinkage.non_indep} in Theorem 
\ref{thm:shrinkage.non_indep}, 
will decrease in $n$ under mild moment conditions. 
\end{example} 

Next we illustrate Theorem \ref{thm:shrinkage.non_indep} using an explicit example of a random vector having dependent coordinates, allowing us to obtain precise results for this particular case.
\begin{example} \label{ex:Student}
Consider ${\bm X} ={\bm Y}+{\bm \theta}$ in $\R^d$ with ${\bm Y}$ from the family of multivariate {central} Student-$t$ distributions, taken here with $k \ge 5$ degrees of freedom, shape given by a symmetric, positive definite matrix
$\Upsilon$ in $\R^{d\times d}$
and $d=2m \ge 6$, even. 
These distributions are the subfamily of the elliptical distributions introduced in Example \ref{elliptical},  obtained by taking
\begin{align}\label{eq:def.phi.Student}
\phi(t) = (1+2t/k)^{-(k+d)/2}; 
\end{align}
{the covariance matrix of ${\bm Y}$ is 
$\Sigma=(k/(k-2)) \Upsilon$. Here we take $\Upsilon = \Id$ so that $\Sigma=\sigma^2\Id$ for 
$\sigma^2 = k/(k-2)$.
}
Using that $d+k>2$, from \eqref{eq:def.phi.Student} followed by \eqref{eq:elliptickernel}, we obtain that
  \begin{align} \label{eq:T.student}
      \frac{1}{\phi(t/2)} \int_{t/2}^{+\infty}  \phi(u) \mathrm{d}u = \frac{t+k}{d+k-2},\,\,
      \mbox{and hence} \quad \Tk =  \left( \frac{{\bm Y}^\transpose {\bm Y}+k \sigma^2}{d+k-2} \right) 
      \Id 
  \end{align} 
is a Stein kernel for the multivariate Student distribution; see also \cite{MRS} and \cite{LVY15}. We obtain the following bounds for the terms controlling $B_\lambda$ in the right-hand side of inequality \eqref{eq:upper.bound.B.lambda} of Theorem \ref{thm:shrinkage.non_indep} (see Supplement B for details):

\begin{align*}
E_{\bsy{\theta}}[d^2\|{\bm X}\|^{-4}]
\le   \frac{d^2 {(k-2)^2} (k+2)}{  (d-2)(d-4) 
{k^3}}, \quad \quad 
\operatorname{Var}(\operatorname{Tr}(\Tk))
=  \frac{2 d^3 k^4}{(d+k-2) (k-2)^4(k-4)}
\end{align*}
and 
\begin{align*}
E[\| \Tk -\sigma^2 \Id\|^2]
=\frac{2 d^2 k^4}{(d+k-2) (k-2)^4(k-4)}.
\end{align*}

Both these last two terms are $o(d^2)$ as long as
$1/k=o(1)$, in which case all conditions in \eqref{eq:3conditions.thm:shrinkage.non_indep} hold.
\end{example}

We note that in view of Example \ref{elliptical}, Example \ref{ex:Student}
can be translated into a Student distribution with $k$ degrees of freedom, with any parameter $\Sigma$, having covariance matrix $(k/(k-2))\Sigma$.\\

A particular instance of Theorem \ref{thm:shrinkage.non_indep} arises when the (centered) distribution $\nu$ with finite covariance matrix $\Sigma$ satisfies a Poincar\'e inequality
\begin{align} \label{eq:Poin.ineq}
\operatorname{Var}_{\nu}(f) \leq C_P E_{\nu}[\|\nabla  f({\bm X})\|^2] \hspace{3mm}
\end{align}
for all functions $f$ for which the quantities make sense and the right-hand side is finite.
In \cite{CFP19}, using techniques that already appeared in \cite{Ute89}, under Assumption \ref{assumptionW.kernel}
it is shown, see Eq.(6) {\it ibid.}, 
 that in this situation
a Stein kernel $\Tk$ for $\nu$ exists that satisfies 
\begin{align*} 
E_\nu[||\Tk ||_{\HS}^2] \leq C_P \operatorname{Tr}{(\Sigma)}.
\end{align*}
Some caveats are addressed through Assumption \ref{assumptionW.kernel} below. Hence, using $ E_\nu[ \operatorname{Tr} (T)] = \operatorname{Tr} (\Sigma) \le  \kappa  d $ and $\operatorname{Tr} (\Sigma^2) \ge ( \operatorname{Tr} (\Sigma))^2 /\operatorname{rank}(\Sigma)$, we obtain
\begin{align}
\label{eq:disc.boundedby.CP-sigma^2}
E_\nu[||\Tk - \Sigma||_{\HS}^2] =  E_\nu[ {\|T\|^2} -\operatorname{Tr} (\Sigma^2)]  \leq \kappa d \left(C_P -  \frac{\operatorname{Tr} (\Sigma)}{\operatorname{rank}(\Sigma)} \right). 
\end{align}
In particular, if the Poincar\'e constant is independent of the dimension (which is the case for product measures), then the (squared) Stein discrepancy on the left hand side is of order $d$, which is negligible compared to $d^2$, so that the third requirement of Theorem \ref{thm:shrinkage.non_indep} would be satisfied. Another family of examples that can be included using these methods are log-concave distributions, see Corollary \ref{cor_sk_shrinkage_lc} below. 

\medskip 
Here we
clarify some issues related to the space of admissible functions for the Stein identity proved in \cite{CFP19}. The argument there proves existence of a Stein kernel such that the identity \eqref{def:multivariate_stein_kernel} holds for functions lying in the closure of smooth, compactly supported test functions with respect to the Sobolev norm. This closure may not be the whole space $W^{1,2}(\nu)$. To bypass this issue, we make the following extra assumptions, which in full generality is stronger than Assumption \ref{assumptionWonly.kernel}, but equivalent in most situations (such as when $\nu$ has a continuous density with full support, or a support with smooth boundary and density bounded away from zero \cite[Section 5.3.3]{Evans}).
\begin{assumption}\label{assumptionW.kernel}
The function ${\bm g}_0({\bm x})={\bm x}/\|{\bm x}\|^2$ is in the closure 
of $C^{\infty}_c(\R^d)$, the set of infinitely differentiable functions with compact support, with respect to the Sobolev norm \eqref{def:W.norm.ker}.
\end{assumption}

While a bound on the squared Stein discrepancy $E[||\Tk - \Sigma||_{\HS}^2]$ is not enough to get suitable control on the variance of $\operatorname{Tr}(\Tk)$, the conclusion of 
Theorem \ref{thm:shrinkage.non_indep} continues to hold if we assume a Poincar\'e inequality with a constant growing sufficiently slowly, and a stronger moment assumption:

\begin{theorem} \label{thm:shrinkage.poinca}
Let Assumption \ref{assumptionW.kernel} be satisfied for the measure $\nu$ of a random vector ${\bm X}$ with mean $\bsy{\theta}$, covariance matrix  $\Sigma$ with largest eigenvalue $\kappa$, and for which the Poincar\'e inequality holds with constant $C_P$ as given in \eqref{eq:Poin.ineq}.
If for some $d_0$ we have
\bea \label{eq:thm:shrinkage.poinca.estimates}
\sup_{d \geq d_0}E_{\bsy{\theta}}[d^{{3}}||{\bm X}||^{-{6}}] < \infty, \qmq{and} C_P = o(\sqrt{d}),
\ena
$\lambda =O(d)$, then $B_\lambda=o(d)$ and when $ \lambda \in [0,2 (\operatorname{Tr}(\Sigma) - 2 \kappa))]$ the risk of $S_\lambda$ is no larger than that of $S_0$ asymptotically, and if $\lambda/(2 ( \operatorname{Tr}(\Sigma) - 2 \kappa)))$ is bounded away from zero and one as $d$ tends to infinity and the conditions in \eqref{bds.theta.traceSigma} hold, 
the shrinkage estimator \eqref{eq:shrinkage.est}
has strictly smaller mean squared error than $S_0({\bm X})={\bm X}$ for all $d$ sufficiently large.
\end{theorem}

This latter result, whose proof is given in Supplement B, contains the case of a vector with independent coordinates and finite Poincar\'e constant, since that is then dimension free. 
Another family of examples is given by the following corollary, which can be applied to certain isotropic random vectors.

\begin{corollary} \label{cor_sk_shrinkage_lc}
For any $A > 0$, there exists a critical dimension $d_0$ that only depends on $A$ such that if  $d \geq d_0$, then for any measure $\nu$ of an isotropic log-concave random vector ${\bm X}-\bsy{\theta}$ with mean zero and covariance matrix  $\Id$ that satisfies Assumption \ref{assumptionW.kernel}, the risk of $S_\lambda$ is strictly smaller than the risk of $S_0$ for
$\lambda = d-2$, corresponding to the classical James-Stein estimator, as long as $\|{\bm \theta}\|^2 \leq A d$.
\end{corollary}

Considering non-isotropic random vectors, the following corollary also holds and still offers a non-parametric generalization of the Gaussian case.
\begin{corollary} \label{cor_shrinkage_SLC_gene_cov}
Consider the measure $\nu$ of a strongly log-concave random vector ${\bm X}-\bsy{\theta}$ with mean zero, covariance matrix $\Sigma$ and two times differentiable potential $\phi$. 
Assume that there are constants $c, l > 0$ such that in the partial order on symmetric matrices ${\rm Hess}(\phi)({\bm x}) \geq c \, {\rm Id}$, for any ${\bm x}\in \mathbb{R}^d$, and that
the smallest eigenvalue of the covariance matrix is greater than $l$. Then there exists a critical dimension $d_0$, such that if  $d \geq d_0$, the risk of $S_\lambda$ is strictly smaller than the risk of $S_0$ for
$\lambda = O(d)$, as long as $\|{\bm \theta}\|^2 =O(d)$.
\end{corollary}

The proof of Corollaries \ref{cor_sk_shrinkage_lc} and \ref{cor_shrinkage_SLC_gene_cov} can be found in Supplement B. These results show that there exist shrinkage estimators that are asymptotically better
than the MLE (which is indeed $S_0$ for an unconditional strictly log-concave noise vector), even in non-isotropic situations and without estimating the covariance matrix. However, taking into account the covariance structure in the estimator may lead to better performances. We leave this question for future work. \\ 

Before beginning the proofs of Theorems \ref{thm_shrinkage_SLC}, \ref{thm:shrinkage.non_indep} and \ref{thm:shrinkage.poinca}, we note that by using the form of the Jacobian \eqref{eq:f.jacobian.f.shrinkage}, for any non-negative definite matrix $M$ with largest eigenvalue bounded by $\kappa$, and ${\bm f} \in W_{1,2}(\nu)$,
\begin{multline}\label{eq:E[Sig.Grade]}
E_{\bm \theta}\left[
\langle
M, \nabla {\bm f}
\rangle
\right]= -\lambda E_{\bsy{\theta}}\left[ \langle M, \frac{1}{\|{\bm X}\|^2}\Id-\frac{2}{\|{\bm X}\|^4}{\bm X}{\bm X}^\transpose \rangle \right] 
\\=
-\lambda E_{\bsy{\theta}}\left[ 
\frac{  \operatorname{Tr}(M)  }{\|{\bm X}\|^2}  -\frac{2  \operatorname{Tr}(M {\bm X \bm X^\transpose}) }{\|{\bm X}\|^4} \right]
\le 
-\lambda E_{\bsy{\theta}}\left[ 
\frac{  \operatorname{Tr}(M) - 2 \kappa }{\|{\bm X}\|^2}   
\right],
\end{multline}
where we used that the trace of a product of two non-negative definite matrices can be bounded the largest eigenvalue of one multiplied by the trace of the other.

\subsection{Zero Bias}
We now derive analogous shrinkage results based on zero-biasing.  
We will write the Euclidean norm as $\|{\bm x}\|_2$ when it appears in proximity to other $p$ norms.
In the following we will take advantage of the fact that differences of an expression involving only ${\bm X}$ with one involving only ${\bm X}^{ij}$ depend only on the (marginal) distributions of ${\bm X}$ and ${\bm X}^{ij}$, and hence not on the choice of coupling. For this reason, in expressions such as \eqref{eq:zb.bound.diff.partials} we may assume that these vectors are jointly given without specifying a joint distribution. 
Conditions under which the zero bias vectors required in the following result exist are given in
Proposition \ref{prop:mvariate.zbias}. 
\begin{theorem}\label{zerobiasshrinkage}
Let  ${\bm X}={\bm Y}+\bsy{\theta}$ where ${\bm Y} \in \mathbb{R}^d$ has covariance matrix $\Sigma$ with largest eigenvalue $\kappa$, and suppose that for all pairs $i,j$ such that $\sigma_{ij} \not =0$ the zero bias vectors ${\bm X}^{ij}$ exist as in \eqref{eq:mzb.identity}.
Then
\bea \label{eq:mse.bound.shrink.zb}
E_{\bsy{\theta}}|| S_\lambda ({\bm X}) - \bsy{\theta}||^2 &\le& E_{\bsy{\theta}}|| {\bm X} - \bsy{\theta}||^2 
+E_{\bsy{\theta}} \left[\frac{\lambda}{|| {\bm X}||^2} \left(\lambda - 
 2 \left(  \operatorname{Tr}(\Sigma) - 2 \kappa \right) \right) \right] +  2 B^*_\lambda
\ena 
where, with ${\bm f}$ as in \eqref{def:f.g0},
\begin{align} \label{eq:zb.bound.diff.partials}
B_\lambda^*= \Bvert E_{\bm \theta}\sum_{i,j=1}^d   \sigma_{ij}\left[ \partial_j f_i({\bm X}^{ij}) - \partial_j f_i({\bf{X}}) \right] \Bvert.
\end{align}

When $\Sigma=\sigma^2 \Id$, 
\bea \label{bstar}
B_\lambda^* &=& \lambda \sigma^2 
\Bvert E_{\bsy \theta}\sum_{i=1}^d   \left(
\frac{\|{\bm X}^i\|^2-2(X_i^i)^2}{\|{\bm X}^i\|^4}\right)
- \left(\frac{\|{\bm X}\|^2-2X_i^2}{\|{\bm X}\|^4} \right) \Bvert. 
\ena 
Assuming $\lambda \in [0,2\sigma^2(d-2)]$ and that $\|\bsy{\theta}\|^2=O(d)$,
the following two scenarios lead to simplified bounds.
\begin{enumerate}
    \item \label{item:gen.Blambda*} Let ${\bm X}$ satisfy
 \eqref{eq:inv.mom.Sd}b with $m=4$ and constant $C_{-4}$, and have components that for $r=4$ and $8$ satisfy $\sup_{i=1,2,\ldots,d}E(X_i-\theta_i)^r \le C_r$. Then with notation as in \eqref{eq:X.neg.i}, for all $d \ge 3$,
\begin{multline} \label{eq:Blambda*.bd}
B^*_\lambda \le  \lambda \sum_{i=1}^d
     |\operatorname{Cov}_{\bsy \theta}((X_i-\theta_i)^2,\|{\bm X}^{\neg i}\|^{-2})| \\
    +\frac{6\lambda \sqrt{C_{-4}}}{d^2} \left(
   d\left(\sigma^2 \sqrt{C_4}+\sqrt{C_8}/3\right)
    +\|\bsy{\theta} \|^2\left(\sigma^2+C_4 \right)
   \right).
\end{multline} 
When 
$\sum_{i=1}^d|\operatorname{Cov}((X_i-\theta_i)^2,\|{\bm X}^{\neg i}\|^{-2})|=o(1)$ 
then $B_\lambda^* = o(d)$, and 
when the ratio $\lambda/(2\sigma^2(d-2))$ stays bounded away from 0 and 1 the shrinkage estimator has strictly smaller mean squared error than $S_0$ for all $d$ sufficiently large.

\item \label{item:mix.Blambda*} Let  ${\bm X}$ follow Model \ref{model:indep},  where ${\bm Y}$ is the $\mu$ mixture of ${\bm Y}_s, s \in \mathcal{S}$, with ${\bm Y}_s$ having mean zero, covariance matrix $\sigma^2 \Id$ and independent components, and let \eqref{eq:inv.mom.Sd}b hold for ${\bm X}_s={\bm Y}_s+\bsy{\theta}$ for $m=2$ for all $s \in \mathcal{S}$ and $i =1,2,\ldots,d$ with constant $C_{-2}$. Then for any $X_{i,s}^*$ on a joint space with ${\bm X}_s$, independent of ${\bm X}_s^{\neg i}$ and  having the $X_{i,s}$-zero bias distribution, 
 \bea \label{eq:Model2.Blamstar.mix}
B_\lambda^*
 \le \frac{25 C_{-2} \lambda \sigma^2}{8d^2} \sum_{i=1}^d  \int_{\mathcal{S}} E_{\bsy{\theta},s} |  X_{s,i}^2 - (X_{s,i}^*)^2 |d \mu.
 \ena 
If $\max_{1 \le i \le d}E[Y_{s,i}^4]\leq C_4$  for all $s \in \mathcal{S}$ then
\bea \label{eq:Model2.Blamstar.refined}
B^*_\lambda \le 
\frac{25C_{-2}\lambda }{8d^2} \left(  dC_4/3 + C_4^{3/4}\|\bsy{\theta}\|_1 +2\sigma^2\|\bsy{\theta}\|_2^2 +d\sigma^4\right), 
\ena
and $B_\lambda^* = O(1)$. When the ratio $\lambda/(2\sigma^2(d-2))$ stays bounded away from 0 and 1 the shrinkage estimator has strictly smaller mean squared error than $S_0$ for all $d$ sufficiently large.
\end{enumerate}
\end{theorem}
We note that the bounds \eqref{eq:zb.bound.diff.partials}, \eqref{bstar} and  \eqref{eq:Model2.Blamstar.mix} are tight, returning zero when the observation vector is Gaussian. We next present Remark  \ref{rem:covariance.sum},
discussing the covariance term in the bound in Part \ref{item:gen.Blambda*}, followed by Examples \ref{ex:spherical}, 
\ref{eq:student.zb} and \ref{ex:noise.corruption}. The first two of these examples illustrate the computation of the bounds presented in the main part of the theorem, followed by an application of our results to corrupted Gaussian observations.
The proof of Theorem \ref{zerobiasshrinkage} is deferred to Supplement B.

\begin{remark}\label{rem:covariance.sum}
In the simplest case, 
the covariance terms in the sum in the bound 
\eqref{eq:Blambda*.bd} in Part  \ref{item:gen.Blambda*} Theorem \ref{zerobiasshrinkage} will be  zero when ${\bm X}$ follows Model \ref{model:indep} with ${\bm Y}$ the mixture of ${\bm Y}_s, s \in \mathcal{S}$ having independent components. This term will also be of the desired order $o(1/d)$ when the dependence between the components of ${\bm X}$ is sufficiently weak. For example, this order obtains when the components of ${\bm X}$ are locally dependent with sufficiently small dependency neighborhoods, and
certain technical moment bounds on inverse norms, discussed in Section \ref{ssec_inv_norms}, are in force.

 Precisely, say for each $i=1,\ldots,d$ there exists a dependency neighborhood $\{i\} \subset \mathcal{N}_i \subset \{1,\ldots,d\}$ of size $\eta$ such that
 $X_i$ is independent of $\{X_j, j \not \in \mathcal{N}_i\}$. Then, suppressing the dependence on $i$ when defining $U$ and $V$ we may write
 \begin{align*}
 \|{\bf X}^{\neg i}\|^2  = U+V \qmq{where} U=\sum_{j \in \mathcal{N}_i\setminus\{i\}} X_j^2 \qmq{and} 
 V=\sum_{j \not \in \mathcal{N}_i}
 X_j^2\\
 \mbox{and let} \quad W=(X_i-\theta_i)^2-\sigma^2.
 \end{align*}
 Then, using that $W$ has mean zero and is independent of $V$, and that $U$ and $V$ are non-negative, applying H\"older's inequality in the final step we obtain
 \begin{align*}
|\operatorname{Cov}_{\bsy \theta}((X_i-\theta_i)^2,\|{\bm X}^{\neg i}\|^{-2})|=\left| E \left[ \frac{W}{U+V} \right] \right| = \left| E \left[ \frac{W}{U+V}  -\frac{W}{V} \right] \right|
= \left| E \left[ \frac{WU}{V(U+V)}  \right] \right|\\
\le E \left[ \frac{|W|U}{V(U+V)} \right]
\le E \left[ \frac{|W|U}{V^2} \right] \le E[W^4]^{1/4}E[U^4]^{1/4} E[V^{-4}]^{1/2}.
 \end{align*}
When the eighth moments of the components of ${\bm Y}$ are uniformly bounded by $C_8$,
$$
EW^4 \le 8(C_8+\sigma^8) \qmq{and}
EU^4 \le 64 \eta^4 \left(
C_8 + \|\theta\|_{\infty}^8
\right).
$$
Hence, when $V$ satisfies \eqref{eq:inv.mom.Sd}b for $m=4$ when playing the role of $\|{\bm X}^{\neg i}\|^2$, 
we obtain 
 $$
|\operatorname{Cov}_{\bsy \theta}((X_i-\theta_i)^2,\|{\bm X}^{\neg i}\|^{-2})| \le \frac{8\eta  [(C_8+\sigma^8)(C_8+
\|\theta\|_{\infty}^{8})]^{1/4}}{(d-\eta)^2} = o(1/d)
$$
as desired, when $\eta$ and  $\eta \|\theta\|_{\infty}^2$ are both $o(d)$.
\end{remark}

\begin{example}\label{ex:spherical}
We illustrate how coupling can allow for a computation of a bound on \eqref{bstar}. Let ${\bm U}$ and ${\bm U}^*$ be the uniform distributions on $S^{d-1}$ and $B^d$, the sphere and unit ball in  $\mathbb{R}^d$, respectively.
The divergence theorem, and ${\rm Area}(S^{d-1})/{\rm Vol}(B^d)=d$, yield that for smooth ${\bm f}:\mathbb{R}^d \rightarrow \mathbb{R}^d$,
$$
E[\langle {\bm U}, {\bm f}({\bm U})\rangle ] = \frac{1}{d}E[\nabla \cdot {\bm f} ({\bm U}^*)],
$$
and hence $E[{\bm U}]=0, {\rm Cov}({\bm U})=\Id/d$ and 
the zero bias vectors ${\bm U}^i=_d{\bm U}^*$ for all $i=1,\ldots,d$. 
For $\sigma>0$ letting 
$$
{\bm X}={\bm \theta}+\sigma \sqrt{d}{\bm U} \qmq{and} 
{\bm X}^*={\bm \theta}+\sigma \sqrt{d}{\bm U}^*
$$
we have ${\bm X}^i={\bm X}^*, i=1,\ldots,n, E[{\bm X}]={\bm \theta}$ and ${\rm Cov}({\bm X})=\sigma^2 \Id$.

The distribution of $R^\sharp=\|{\bm U}^*\|$ is given by
\begin{align*}
P(R^\sharp \le r) = P(\|{\bm U}^*\| \le r) = 
\frac{{\rm Vol}(rB^d)}{{\rm Vol}(B^d)}
= r^d \qmq{for $0 \le r \le 1$,}
\end{align*}
and hence has density in the unit interval given by
\begin{align*}
\frac{d}{dr}P(R^\sharp \le r) = dr^{d-1}
\qmq{and so} E[R^\sharp]=\int_0^1 dr^d dr
=\frac{d}{d+1}.
\end{align*}
Now, letting ${\bm U}$ `pick' a uniform direction, and an independent variable $R^\sharp$ with the density above `pick' a relative magnitude, we obtain the coupling
\begin{align*} 
{\bm X}={\bm \theta}+\sigma \sqrt{d}{\bm U} \qmq{and} {\bm X}^* = {\bm \theta}+ \sigma  \sqrt{d} R^\sharp {\bm U},
\end{align*}
and that
\begin{align}\label{eq:diffUUstar}
E\|{\bm U}-{\bm U}^*\|= E\|{\bm U}-R^\sharp{\bm U}\|=E(1-R^\sharp)=1-\frac{d}{d+1}\le \frac{1}{d}.
\end{align}
Hence, for $h:\mathbb{R}^d \rightarrow \mathbb{R}$ and $\alpha=\sup_S \|\nabla h({\bm x})\| < \infty$ for $S$ the union of the supports of ${\bm X}$ and ${\bm X}^*$, 
\begin{align}\label{eq:using.alpha}
E| h({\bm X})-h({\bm X}^*)|
\le \alpha E\|{\bm X}-{\bm X}^*\| = \alpha \sigma \sqrt{d}
E\|{\bm U}-{\bm U}^*\|\le \frac{\alpha 
\sigma}{\sqrt{d}}.
\end{align}

Specializing
\eqref{bstar} to the case where  ${\bm X}_i={\bm X}^*$ for all $i=1,\ldots,d$, now taking $d \ge 3$, we obtain
\begin{multline}
B_\lambda^*= \lambda \sigma^2 
\Bvert E_{\bsy \theta} \sum_{i=1}^d   \left(
\frac{\|{\bm X}^*\|^2-2(X_i^*)^2}{\|{\bm X}^*\|^4}\right)
- \left(\frac{\|{\bm X}\|^2-2X_i^2}{\|{\bm X}\|^4} \right) \Bvert
\\
= \lambda \sigma^2 (d-2) \Bvert E_{\bsy{\theta}}  \left( \frac{1}{\|{\bm X}^*\|^2}-\frac{1}{\|{\bm X}\|^2} \right) \Bvert. \label{eq:spec.B.lambda}
\end{multline}

Taking $\|{\bm \theta}\|^2 \ge c\sigma^2 d$ for some $c>1$ we have $\min (\|{\bm X}\|^2, \|{\bm X}^*\|^2) \ge (\|{\bm \theta}\|-\sigma \sqrt{d})^2  \ge (\sqrt{c}-1)^2 \sigma^2 d$
and hence, over the supports of ${\bm X}$ and ${\bm X}^*$,
\begin{align*}
\left\| \nabla \frac{1}{\|{\bm x}\|^2} \right\| = \left\| 
\frac{2{\bm x}}{\|{\bm x}\|^4}
\right\| =2\|{\bm x}\|^{-3} \le \frac{2}{(\sqrt{c}-1)^{3}\sigma^3 d^{3/2}}.
\end{align*}

Now, using \eqref{eq:spec.B.lambda} and \eqref{eq:using.alpha} with $\alpha$ as the upper bound above, and taking $\lambda=\sigma^2(d-2)$, we obtain
\begin{align}\label{eq:bd.Blambda.star.uniform.S}
2B_\lambda^* 
\le
2\sigma^4(d-2)^2 \times \frac{\alpha \sigma}{\sqrt{d}}
=
\frac{4 \sigma^2 (d-2)^2}{(\sqrt{c}-1)^3d^2}.
\end{align}
When $\|{\bm \theta}\|^2 \le C\sigma^2 d$
$$
\|{\bm X}\| = \|{\bm \theta}+ \sigma \sqrt{d}{\bm U}\| \le \|{\bm \theta}\|+\sigma \sqrt{d} \le (\sqrt{C}+1)^2 \sigma \sqrt{d},
$$ 
and \eqref{eq:shrink.gain.gen} yields the upper bound on the second term in the bound \eqref{eq:mse.bound.shrink.zb} 
\begin{align} \label{eq:shrink.gain}
-\sigma^4(d-2)^2E_{\bsy{\theta}}\left[
\frac{1}{\|{\bm X}\|^2} 
\right]  \le \frac{-\sigma^2(d-2)^2}{(\sqrt{C}+1)^2d}.
\end{align}
Comparing to \eqref{eq:bd.Blambda.star.uniform.S}, we see shrinkage will strictly improve the mean squared error when
$$
d>\frac{4(\sqrt{C}+1)^2}{(\sqrt{c}-1)^3}.
$$
For instance, when $c=4$ and $C=9$ shrinkage is advantageous when $d>64$.

More generally, we illustrate in Supplement B the use of the bound \eqref{eq:zb.bound.diff.partials} for a case where the observations has a non-diagonal covariance matrix. 
\end{example}

\begin{example}\label{eq:student.zb}
For ${\bm Y}$ having the mean zero,
variance $\sigma^2 \Id$ Student distribution with $k$ degrees of freedom, $d=2m$ even, 
an explicit zero bias coupling can be constructed using the representation of the distribution of ${\bm Y}$ as 
 ${\bm Y}_\gamma =  {\gamma^{-1/2}}\sigma {\bm N}$, with ${\bm N} \sim \mathcal{N}_d(0, \Id)$, mixed over 
 $\gamma \sim \Gamma(k/2,k/2)$,  
as outlined in Example \ref{ex:Student}.
Part \ref{prop:mvariate.zbias.mix} of Proposition \ref{prop:mvariate.zbias} gives that, for $i=1,\ldots,d$, the zero bias vectors ${\bm Y}^i$ are given by the mixture ${\bm Y}_\delta$ where the distribution of $\delta$ has Radon-Nikodym derivative with respect to the distribution of $\gamma$ equal to ${\rm Var}(Y_{\gamma,i})/{\rm Var}(Y_{\gamma})$, that is, proportional to $\gamma^{-2}$ and hence, 
$\delta \sim \Gamma(k/2-1,k/2)$. Now letting $\epsilon \sim \Gamma(1,k/2)$ be independent of $\delta$, with both variables independent of ${\bm N}$, a coupling of $\gamma$ and $\delta$ is achieved by setting $\gamma=\delta+\epsilon$. 
Hence Part \ref{prop:mvariate.zbias.mix} of Proposition \ref{prop:mvariate.zbias}, and the fact that the normal is fixed by the zero bias transformation, yield the couplings
\begin{align}\label{eq:student.coupling}
{\bm X}=
{\bm \theta}+\frac{\sigma}{\sqrt{\delta+\epsilon}} {\bm N}
\qmq{and}{\bm X}^i={\bm \theta}+\frac{\sigma}{\sqrt{\delta}} {\bm N}, i=1,\ldots,d.
\end{align}
Using the latter coupling to bound the right-hand side of \eqref{bstar}, after computations that are detailed in Supplement B, one can get, for ${\bm \theta}=0$,
$$B_\lambda^* \le \frac{2 \lambda}{k}.$$ 
This bound is $o(d)$ when $\lambda = O(d)$ and $1/k = o(1).$ In the case where ${\bm \theta} \ne {\bm 0},$
$$B_\lambda^* \le \frac{8 \lambda (d+k-2)}{(d-2) k}$$
and if $\lambda = O(d)$ and $1/k = o(1)$, this bound is $o(d)$ as desired. 
\end{example} 

\begin{example}\label{ex:noise.corruption} 
Let ${\bm X}={\bm Y}+{\bm \theta}$ where ${\bm Y}$ is a Gaussian vector ${\bm Y}_0 \sim \mathcal{N}_d(0,\sigma^2 \Id)$ corrupted by a mean zero, variance $\sigma^2 \Id$ outlier vector ${\bm Y}_1$ that 
satisfies assumption \eqref{eq:mgalelike}. By Part \ref{prop:mvariate.zbias.ex.un} of Proposition \ref{prop:mvariate.zbias}, the zero bias vectors of ${\bm Y}_1$ exist. 

One corruption model is 
additive, where for some $\epsilon \in [0,1]$, ${\bm Y}=\sqrt{1-\epsilon}{\bm Y}_0+\sqrt{\epsilon}{\bm Y}_1$. By Part  \ref{prop:mvariate.zbias.sum} of Proposition \ref{prop:mvariate.zbias}, the zero bias vectors of ${\bm Y}$ exist and can be coupled to ${\bm Y}$ via
\begin{align*}
{\bm Y}^i = 
\left\{
\begin{array}{cc}
\sqrt{1-\epsilon}{\bm Y}_0^i+\sqrt{\epsilon}{\bm Y}_1 &\\
\sqrt{1-\epsilon}{\bm Y}_0+\sqrt{\epsilon}{\bm Y}_1^i& 
\end{array}
\right.
=
\left\{
\begin{array}{cc}
\sqrt{1-\epsilon}{\bm Y}_0+\sqrt{\epsilon}{\bm Y}_1 & \mbox{with probability $1-\epsilon$}\\
\sqrt{1-\epsilon}{\bm Y}_0+\sqrt{\epsilon}{\bm Y}_1^i& \mbox{with probability $\epsilon$,}
\end{array}
\right.
\end{align*} 
where we have used the normality of ${\bm Y}_0$ to replace ${\bm Y}_0^i$ by ${\bm Y}_0$. With ${\bm Y}_1$ additionally 
satisfying the conditions in Part \ref{item:mix.Blambda*} of Theorem \ref{zerobiasshrinkage}, the bound 
\eqref{bstar} holds with the reduction  factor of $\epsilon$ over its value for  ${\bm Y}_1$.

Another way that the outlier can enter is via mixing, where with probability $1-\epsilon$ the vector ${\bm Y}$ is the Gaussian ${\bm Y}_0$, and with probability $\epsilon$ equals ${\bm Y}_1$. 
By Part \ref{prop:mvariate.zbias.mix} of Proposition \ref{prop:mvariate.zbias}, the zero bias vector ${\bf Y}^i$ is the same $1-\epsilon, \epsilon$ mixture of  of ${\bf Y}_0^i={\bf Y}_0$ and ${\bf Y}_1^i$. In particular, the bound \eqref{bstar} takes the value zero with probability $1-\epsilon$, and therefore equals $\epsilon$ of the value it has for ${\bm Y}_1$.
\end{example}

In summary, the 
four main results of this section, Theorems \ref{thm_shrinkage_SLC}, \ref{thm:shrinkage.non_indep}, \ref{thm:shrinkage.poinca} and \ref{zerobiasshrinkage}, 
provide  bounds for the mean squared error of the shrinkage estimator $S_\lambda$, which is shown to be strictly smaller than that of $S_0$ under a variety of conditions. The first
three of these results depend on Stein kernels, and the
fourth on zero biasing. The application at hand determines which of the two types of results would be more straightforward to apply. For instance, Theorem \ref{thm:shrinkage.non_indep} requires the existence of a Stein kernel $\Tk$ and that it satisfies certain moment conditions, whereas Theorem  \ref{zerobiasshrinkage} requires that the observation ${\bm X}$ itself satisfies \eqref{eq:mgalelike}, as well as additional moment assumptions. Examples \ref{ex:Student} and \ref{eq:student.zb} illustrate that in some situations when both types of results can be applied, they may yield subtly different bounds. 

\section{SURE: Stein Unbiased Risk Estimate}\label{sec_SURE}
In this section, we demonstrate that in some settings the bias incurred when using standard forms of SURE as in \eqref{eq:intro.sure} when the observation ${\bm X}$ is not Gaussian can be controlled, and is sufficiently small so as to yield estimates useful for the selection of tuning parameters. 

We start by reviewing the Gaussian case for SURE. Though the classical is where the covariance of the observation is $\sigma^2 \Id$ with known $\sigma^2$, we continue with the more general instance where the covariance $\Sigma$ is a known matrix in order to illustrate the range of the results obtained, and foreshadow shrinkage results in cases where the covariance can be consistently estimated. Suppose then that for a known positive definite matrix $\Sigma$ we observe ${\bm X}$ with distribution ${\cal N}_d(\bsy{\theta},\Sigma)$, a normal distribution in ${\mathbb R}^d$ with unknown mean $\bsy{\theta}\in \mathbb{R}^d$. With ${\bm f} \in W^{1,2}(\nu)$, here taking $\nu$ to be the measure of this multivariate Gaussian, we want to compute an unbiased estimate of the mean squared error, or risk, of an estimator of $\bsy{\theta}$ of the form $S({\bm x})={\bm x}+{\bm f}({\bm x})$, that is, an unbiased estimate of the expectation of 
\begin{multline} \label{eq:mse.expansion}
\|S({\bm X})-\bsy{\theta}\|^2 \\= \|{\bm X}-\bsy{\theta} + {\bm f}({\bm X})\|^2 
= \|{\bm X}-\bsy{\theta}\|^2 + \|{\bm f}({\bm X})\|^2 + 2\langle {\bm f}({\bm X}),{\bm X}-\bsy{\theta}\rangle.
\end{multline}
Unbiased estimates of the first two terms are easily constructed, as the expectation of the first term is ${\rm Tr}(\Sigma)$, a quantity assumed known, and $\|{\bm f}({\bm X})\|^2$ is an unbiased estimator of its own expectation. 
Applying the Stein identity \eqref{eq:Stein.Identity.d} to the last term of  \eqref{eq:mse.expansion} eliminates the unknown $\bsy{\theta}$ via
\beas
E[\langle {\bm X}-\bsy{\theta} , {\bm f}({\bm X}) \rangle ] = E [ \langle \Sigma, \nabla {\bm f}({\bm X}) \rangle ] 
\enas
and leads to the conclusion that
\bea \label{eq:sure}
{\rm SURE}({\bm f},{\bm X}) :=
\operatorname{Tr} (\Sigma) + \|{\bm f}({\bm X})\|^2 + 2\sum_{i,j =1}^d \sigma_{ij} \partial_j f_i ({\bm X})
\ena
is unbiased for the risk; the resulting expression is also computable from the data using the known form of the estimator. 

We turn now to the case where ${\bm X}$ continues to have unknown mean $\bsy{\theta}$ and known covariance $\Sigma$, but is not necessarily Gaussian. 
When a Stein kernel $\Tk_{{\bm X}-\bsy{\theta}}$ exists for ${\bm X}$, by applying identity \eqref{def:multivariate_stein_kernel}
to the final term on the right hand side of \eqref{eq:mse.expansion}
we 
arrive at the form
\bea \label{eq:kernel.replacement.gen.Sigma}
{\rm SURE}_k ({\bm f},{\bm X}) = \operatorname{Tr} (\Sigma) + || {\bm f}({\bm X})||^2 + 2  \langle
T_{{\bm X} - \bsy{\theta}}, \nabla {\bm f}  ({\bm X})
\rangle,
\ena
which is
unbiased for the risk; the subscript $k$ denotes that this version is the Stein kernel form. 
Alternatively,
when the appropriate zero bias vectors exist, from identity \eqref{eq:mzb.identity} we have 
\begin{align}
\label{eq:zero.replacement}
{\rm SURE}_z({\bm f},{\bm X}) := \operatorname{Tr} (\Sigma) + \|{\bm f}({\bm X})\|^2 + 2 \sum_{i,j=1}^d \sigma_{ij} \partial_j f_i({\bm X}^{ij})
\end{align}
is again unbiased for the risk of $S({\bm X})$.

To use the forms  \eqref{eq:kernel.replacement.gen.Sigma} and \eqref{eq:zero.replacement} in practice we would need to be able to generate the Stein kernel, and zero bias vectors, respectively,  upon observing ${\bm X}$, which is not possible without knowledge of the mean being estimated. Nevertheless, when ${\bm X}$ is close to normal then heuristically 
the Stein kernel $\Tk_{{\bm X} - \bsy{\theta}}$ is close to $\Sigma$, and  the zero bias vectors ${\bm X}^{ij},i,j=1,\ldots,d$ are close in distribution to ${\bm X}$. These observations motivate the use of SURE as in \eqref{eq:sure} with the observed ${\bm X}$, which in the approximate normal case should give a risk estimator that is Approximately the Same as SURE (ASSURE), in that it has small bias for the estimate of risk. Propositions \ref{prop:AsureBiasnotindep} and \ref{prop:AsureBias.z} bound the bias
\begin{align} \label{def_in_prop:SurehX}
{\rm Bias}_{\bsy{\theta}}({\rm SURE}({\bm f},{\bm X}))=E_{\bsy{\theta}}[{\rm SURE}({\bm f},{\bm X}))]-
E_{\bsy{\theta}}\|S({\bm X})-\bsy{\theta}\|^2
\end{align}
of ASSURE, that is, of Stein's unbiased risk estimate \eqref{eq:sure}
when applied in non-Gaussian frameworks.

Proposition \ref{prop:AsureBiasnotindep} applies the multivariate Stein kernel framework to determine a bound on the bias of SURE; as in Theorem \ref{thm:shrinkage.non_indep}, a Stein discrepancy
makes an appearance in the bound. A proof is given in Supplement C.
\begin{proposition} \label{prop:AsureBiasnotindep}
Let ${\bm X}$ have mean $\bsy{\theta}$ and covariance $\Sigma$, let $\Tk_{{\bm X} - \bsy{\theta}}$ be a
Stein kernel for ${\bm X} -  \bsy{\theta}$ in the sense of \eqref{def:multivariate_stein_kernel},  and suppose that ${\bm f} \in W^{1,2}(\nu)$. Then 
\bea \label{biassure1}
\left| {\rm Bias}_{\bsy{\theta}}({\rm SURE}({\bm f},{\bm X})) \right| 
&\le& 2 | E [\langle \Sigma- \Tk_{{\bm X} - \bsy{\theta}},  \nabla {\bm f} ({\bm X}) \rangle] |. \ena 
If for all $i,j=1,\ldots,d$ the supremum norms $\|\partial_j f_i\|$ over the support of ${\bm X}$ are bounded, then letting $T_{ij}$ denote the $i,j^{th}$ entry of $\Tk_{{\bm X} - \bsy{\theta}}$,
\beas
|{\rm Bias}_{\bsy{\theta}}({\rm SURE}({\bm f},{\bm X})) |
&\le& 2  \sum_{i,j=1}^d  ||\partial_j f_i|| E [| \sigma_{ij}   - T_{ij}|].
\enas 
\end{proposition}

\begin{proof} Taking the difference of \eqref{eq:sure} and \eqref{eq:kernel.replacement.gen.Sigma} yields \eqref{biassure1}.
The second assertion now follows from the first by expanding out the inner product and applying the given bound on the partial derivatives. 
\end{proof}

\medskip 
Proposition \ref{prop:AsureBiasnotindep} has the following analog through the use of zero biasing. 
For $g:\mathbb{R}^d \rightarrow \mathbb{R}$ let $\|g\|_{{\rm Lip}}$ denote the usual Lipschitz semi-norm of $g$, and for $i=1,\ldots,d$ let $\|g\|_{{\rm Lip},i}$ be the smallest $L$ such that for all real $x_1,\ldots,x_{i-1},x_{i+1},\ldots,x_d$ and $u,v$, 
\begin{align*}
 |g(x_1,\ldots,x_{i-1},u,x_{i+1},\ldots,x_d)-g(x_1,\ldots,x_{i-1},v,x_{i+1},\ldots,x_d)| \le L|v-u|.  
\end{align*}

\begin{proposition} \label{prop:AsureBias.z}
Let ${\bm X}=\bsy{\theta}+{\bm Y}$ where 
${\bm Y}$ has mean zero, covariance 
$\Sigma$, and whose zero bias vectors exist. Then, when ${\bm f} \in W^{1,2}(\nu)$,
\begin{equation} \label{eq:zb.bias.partialhi}
|{\rm Bias}_{\bsy{\theta}}({\rm SURE}({\bm f},{\bm X}))| \le 2 \left| \sum_{i,j=1}^d \sigma_{ij} E_{\bm \theta}\left( \partial_j f_i({\bm X}^{ij})- \partial_j f_i({\bm X})\right) \right|,
\end{equation}
and when $\partial_j f_i$ is Lipschitz for all $i,j=1,\ldots,d$,
\begin{equation}  \label{eq:zb.bias.partialhi.lip.gen}
|{\rm Bias}_{\bsy{\theta}}({\rm SURE}({\bm f},{\bm X}))| \le 2   \sum_{i,j=1}^d |\sigma_{ij}|\|\partial_j f_i\|_{\rm Lip}d({\bm X},{\bm X}^{ij}),
\end{equation}
where $d(\cdot,\cdot)$ is the Wasserstein-1 distance. 
In addition, under Model \ref{model:indep} where ${\bm Y}_s, s \in \mathcal{S}$ has independent components and $\Sigma_s={\rm diag}(\sigma_1^2,\ldots,\sigma_d^2)$ is non-singular, 
\begin{equation}  \label{eq:zb.bias.partialhi.lip}
|{\rm Bias}_{\bsy{\theta}}({\rm SURE}({\bm f},{\bm X}))| \le 2  \sum_{i=1}^d \sigma_i^2\|\partial_i f_i\|_{{\rm Lip},i} \int_{\mathcal{S}} d(Y_{s,i}^*,Y_{s,i}) d\mu,
\end{equation}
where $Y_{s,i}^*$ has the $Y_{s,i}$-zero bias distribution.
\end{proposition}

Proposition \ref{prop:AsureBias.z} is proved in Supplement C. The bounds \eqref{biassure1} and \eqref{eq:zb.bias.partialhi} above measure deviation from normal through the deviation of the Stein Kernel, and zero bias distribution,  respectively. Indeed, 
if the data are normally distributed, then both results return a bound of zero, recovering the Gaussian case.

\subsection{Sure applied to Shrinkage} 
Now specializing to the 
shrinkage estimator given by \eqref{eq:shrinkage}, 
Corollary \ref{prop:sure.shrinkage.z.mult} gives two results on the bias of SURE for the shrinkage estimator when applied in non-Gaussian settings, one using Stein kernels, and the other zero biasing. Both claims follow as immediate consequences of our results in Section \ref{sec:shrinkage} upon noting that $| \operatorname{Bias}_{\bsy{\theta}}(\operatorname{SURE}({\bm f},{\bm X}))|$, as given in \eqref{biassure1} and \eqref{eq:zb.bias.partialhi}, correspond to quantities whose bounds are provided by Theorems \ref{thm:shrinkage.non_indep} and \ref{zerobiasshrinkage}.

\begin{corollary} \label{prop:sure.shrinkage.z.mult} Let ${\bm f}$ be given by \eqref{eq:f.jacobian.f.shrinkage}, ${\rm SURE}({\bm f},{\bm X})$ as in \eqref{eq:sure}, 
${\rm Bias}_{\bsy{\theta}}({\rm SURE}({\bm f},{\bm X}))$
as in \eqref{def_in_prop:SurehX}, and ${\bm X}=\bsy{\theta} + {\bm Y} \in \mathbb{R}^d$, where ${\bm Y}$ has mean zero and positive definite covariance matrix $\Sigma$.

\begin{enumerate}
\item 
If ${\bm Y}$ has Stein kernel $\Tk$, then with $B_\lambda$
as given in \eqref{eq:upper.bound.B.lambda} 
\beas 
\left| {\rm Bias}_{\bsy{\theta}}({\rm SURE}({\bm f},{\bm X})) \right| 
\le  2B_\lambda,
\enas 
and if the conditions in  \eqref{eq:3conditions.thm:shrinkage.non_indep} hold and $\lambda \in [0,2 (\operatorname{Tr}(\Sigma) - 2 \kappa)]$ then this bound is of order $o(d)$. 

\item If the zero bias vectors of ${\bm Y}$ exist, then with $B_\lambda^*$ as given in  \eqref{eq:zb.bound.diff.partials} 
\beas 
\left| {\rm Bias}_{\bsy{\theta}}({\rm SURE}({\bm f},{\bm X})) \right| 
\le  2B_\lambda^*.
\enas
\end{enumerate}
\end{corollary}
The first claim is immediate via Theorem \ref{thm:shrinkage.non_indep} and 
comparing  \eqref{eq:defB} there to \eqref{biassure1}, and the second claim likewise follows from Theorem \ref{zerobiasshrinkage}; conditions that guarantee $B_\lambda^*=o(d)$ are detailed following the statement of that latter result. 

\subsection{SURE applied to soft-thresholding}\label{eq:SureToSoftThresh}
Thresholding of statistical quantities, meaning keeping only ``important''
quantities, as indicated by their estimates having exceeded some threshold, is widely used in practice. Such procedures are
at the core of breakthroughs using wavelet estimation, and their
adaptivity in Besov spaces, see \cite{DoJo95}. 

To obtain optimality from the minimax viewpoint, a careful, data-driven
selection of the threshold is generally needed. It has been
shown that the use of a SURE estimate of the mean squared risk of
thresholded estimators indeed leads to adaptivity in the Gaussian
setting \cite{DoJo95}. 

Outside the Gaussian setting, it is also known that an optimal theoretical
value of the threshold gives minimax rates of estimation under
some moment assumptions on the noise, in the case of independent coordinates \cite{AvHo03,AvHo05,DeJu96}. But in this general framework, to our knowledge, the validity of threshold
selection via minimizing the SURE estimate of the associated
risks remains an open question.

To begin to approach the problem of obtaining adaptivity results for wavelet
estimation using SURE outside the Gaussian setting, we present some conclusions for the selection of a threshold $\lambda>0$ when estimating the mean of a random vector ${\bm X}={\bm Y}+{\bm \theta}$, where ${\bm Y}$ is a centered random vector 
with covariance matrix $\sigma^{2}{\rm Id}$, by the estimate  ${\bm S}_{\lambda}\left({\bm X}\right)$ whose coordinates are given by soft-thresholding the coordinates of ${\bm X}$ via $({\bm S}_{\lambda}\left({\bm x}\right))_{i}={\rm sgn}\left(x_i\right)\left(\left|x_i\right|-\lambda\right)_{+}$ for $x_i \in \mathbb{R}, \lambda>0$. By letting
${\bm f}_{\lambda}\left({\bm x}\right)=S_{\lambda}\left({\bm x}\right)-{\bm x}$, the SURE estimate of the risk has the simple following formula:
\begin{equation}\label{formula_SURE_ST}
    {\rm SURE}\left({\bm f}_{\lambda},{\bm X}\right)=d\sigma^2+\sum_{i=1}^d \min\{ X_i^2,\lambda^2\}-2\cdot {\rm Card}\{i:|X_i|\leq \lambda\}.
\end{equation}
Assume now that ${\bm Y}$ admits a Stein kernel $T$.
By Proposition \ref{prop:AsureBiasnotindep}, when ${\bm f}_{\lambda}\in W^{1,2}(\nu)$,
\begin{align}
\left|{\rm Bias}_{{\bm \theta}}\left({\rm SURE}\left({\bm f}_{\lambda},{\bm X}\right)\right)\right| & \leq2\left|E_{\bm \theta}\left[\left\langle \sigma^2 \Id
-T,\nabla{\bm f}_{\lambda}({\bm X})\right\rangle \right]\right|\nonumber \\
 & =2\left|\sum_{i}^{d}E_{\bm \theta}\left[\left(\sigma^2  -T_{ii}\right)\mathbf{1}_{\left\{ \left|X_{i}\right|\leq\lambda\right\} }\right]\right|\nonumber \\
 & =2\left|\sum_{i=1}^{d}E_{\bm \theta}\left[\left(\sigma^2-T_{ii}\right)\mathbf{1}_{\left\{ \left|X_{i}\right|>\lambda\right\} }\right]\right|,\label{eq:upper_bias_kernel}
\end{align}
where the first equality follows from the identities $\partial_{i}{\bm f}_{\lambda,j}({\bm X})=-\delta_{ij}{\bm 1}_{\left\{ \left|X_{i}\right|\leq\lambda\right\} }$
and the second equality by using the fact that $E\left[\sigma^{2}-T_{ii}\right]=0$.

Writing out the expression for the risk of ${\bm S}_{\lambda}$ yields
\[
E_{\bm \theta}\left[\left\Vert S_{\lambda}\left({\bm X}\right)-{\bm \theta}\right\Vert ^{2}\right]=\sum_{i=1}^{d}E_{\bm \theta}\left[\left({\rm sgn}\left(X_{i}\right)\left(\left|X_{i}\right|-\lambda\right)_{+}- \theta_{i}\right)^{2}\right]. 
\]
Reducing to one dimension, 
for $\theta\in\mathbb{R}$, letting
\[
p(\lambda,\theta)=E\left[\left({\rm sgn}\left(Y+\theta\right)\left(\left|Y+\theta\right|-\lambda\right)_{+}-\mathbf{\theta}\right)^{2}\right], 
\]
one may verify that
\begin{align} \label{eq:p.lambda.theta.ineq}
p\left(\lambda,\theta\right)\geq P\left(\left|Y+\theta\right|>\lambda+\left|\theta\right|+1\right).
\end{align}
The latter bound will be useful for controlling the bias of SURE by the risk in  the strongly log-concave case.

\subsubsection{Strongly log-concave case}
Assume furthermore that ${\bm Y}$ has a positive strongly log-concave
density on $\mathbb{R}^d$. In this case, $T$ is uniformly bounded \cite{Fathi19}, so
there exists a positive constant $L$ such that $\max_{i}\left|\sigma^{2}-T_{ii}\right|\leq L$
$a.s.$ In addition, ${\bm Y}$ has sub-Gaussian tails, in the sense that there exists a constant $a>0$ such that for any $t>0$, $P(\lVert {\bm Y} \rVert \geq t) \leq a e^{-t^2/C}$.  Note that when such property is in force,
it is sufficient according to \cite{AvHo03}, to search among values of $\lambda$ in the range $I=[0,\sqrt{C\log d}]$.
Hence, inequalities (\ref{eq:upper_bias_kernel}) and 
\eqref{eq:p.lambda.theta.ineq} 
give
\begin{multline}\label{upper_bias}
\left|{\rm Bias}_{{\bm \theta}}\left({\rm SURE}\left({\bm f}_{\lambda},{\bm X}\right)\right)\right|  \leq 2L\sum_{i=1}^{d}P\left(\left|X_{i}\right|>\lambda\right)
  \leq2LM\cdot E_{\bm \theta}\left[\left\Vert S_{\lambda}\left({\bm X}\right)-{\bm \theta}\right\Vert ^{2}\right],\\
  \mbox{where} \quad M=\sup_{\lambda \in I, i=1,\ldots,d}\frac{P\left(\left|X_{i}\right|>\lambda\right)}{P\left(\left|X_{i}\right|>\lambda+A+1\right)} \qmq{and} A=\|{\bm \theta}\|_\infty.
\end{multline}
We are now ready to state our main result about adaptive soft-thresholding calibration. Theorem \ref{thm:soft.thres.log.concave} is proved in Supplement C.

\begin{theorem} \label{thm:soft.thres.log.concave}
Assume that ${\bm X}={\bm \theta}+{\bm Y}$ with ${\bm Y}$ mean zero and  covariance matrix $\sigma^2 \Id$, and has a strongly log-concave distribution with independent coordinates, and that $LM<1/2$, with the constants $L$ and $M$ defined above. Consider the selection of the soft-thresholding parameter via SURE,
\bea \label{lambda_hat}
\widehat{\lambda}\in \arg \min_{\lambda \in I}{\rm SURE}\left({\bm f}_{\lambda},{\bm X}\right),
\ena 
where $I=[0,\sqrt{C\log d}]$, for $C$ the scaling sub-Gaussian
constant of ${\bm Y}$. Then
\[
(1-2LM)E_{\bm \theta}\left[\left\Vert S_{\widehat{\lambda}}\left({\bm X}\right)-{\bm \theta}\right\Vert ^{2}\right]
\leq (1+2LM)\min_{\lambda \in I}E_{\bm \theta}\left[\left\Vert S_{\lambda}\left({\bm X}\right)-{\bm \theta}\right\Vert ^{2}\right]+B\sqrt{d\log^3d},
\]
where $B$ is a positive constant depending only on $C$.
\end{theorem}

In the setting of wavelet estimation, one considers a vector of coefficients computed by applying the wavelet transform on an input signal vector. But independence of the noise terms in the signal is lost in general - except in the Gaussian case - when taking linear combinations of the transformed signal. This difficulty would be the first one to overcome in order to generalize our results to the wavelet estimation problem.

\subsection{Adaptivity under classical asymptotics}\label{subsec:adaptivity}

Let us consider a $d-$dimensional vector ${\bm X}$ with mean $\bsy{\theta}$ and variance $\sigma^2_d \Id$ where $\sigma^2_d=\sigma^2/d$, with $\sigma^2$ an absolute positive constant. This is a case of interest in statistics, where the vector ${\bm X}$ might be the mean of $d$ i.i.d. vectors with finite variance $\sigma^2$. This setting is also naturally linked to a non-parametric regression model, see for instance \cite{Wass06}, Section 7.3.

Pinsker's  theorem  \cite{pinsker1980optimal} (see also \cite{Nussbaum96}) gives the exact asymptotic minimax risk over $\ell_2-$balls in the Gaussian case. More precisely, letting $\mathcal{G}_d (c)=\left\{\mathcal{N}(\bsy{\theta},\sigma^2_d \Id):\| \bsy{\theta} \| \leq c\right\}$, we have
\begin{align} \label{eq:Pinsker.Gaussian}
\lim_{d\rightarrow+\infty}
\inf_{\widehat{\bsy{\theta}}}
\sup_{P \in\mathcal{G}_{d}(c)}
E_{P}[||\widehat{\bsy{\theta}}-\bsy{\theta}||^{2}]=\frac{\sigma^{2}c^{2}}{\sigma^{2}+c^{2}}
\end{align}
where the infimum is taken over all estimators of $\bsy{\theta}$, that is, over all measurable functions of ${\bm X}$ for which $\bsy{\theta}=E_P[{\bm X}]$.

\sloppypar The asymptotic value of the Gaussian minimax risk can actually be extended to the whole class of distributions $\mathcal{P}_d (c)=\left\{P\in \mathcal{M}_1^+:\| E_P[{\bm X}] \| \leq c, {\rm Var}_P[{\bm X}]=\sigma^2_d \Id \right\}$, where $\mathcal{M}_1^+$ is the set of all probability measures on $\mathbb{R}^d$. More precisely, for any collection of distributions $\mathcal{P}$ such that $\mathcal{G}_d (c) \subset \mathcal{P} \subset \mathcal{P}_d (c)$, it holds
\begin{align} \label{eq:Pinsker.nonGaussian}
\lim_{d\rightarrow+\infty}\inf_{\widehat{\bsy{\theta}}}\sup_{P\in\mathcal{P}}E_{P}[||\widehat{\bsy{\theta}}-\bsy{\theta}||^{2}]
=\frac{\sigma^{2}c^{2}}{\sigma^{2}+c^{2}}.
\end{align}
Indeed, by \eqref{eq:Pinsker.Gaussian} the left-hand side of \eqref{eq:Pinsker.nonGaussian}
is at least as large as the right, since $\mathcal{G}_d (c) \subset \mathcal{P}$. The reverse inequality is achieved by considering the estimator $\widehat{\bsy{\theta}}=c^{2}{\bm X}/(\sigma^{2}+c^{2})$, which satisfies $E_{P}[||\widehat{\bsy{\theta}}-\bsy{\theta}||^{2}]\leq\sigma^{2}c^{2}/(\sigma^{2}+c^{2})$ whenever $P \in  \mathcal{P}_d (c)$.

In the Gaussian case, the James-Stein estimator $S_\lambda({\bm X})$ in \eqref{eq:shrinkage},
with $\lambda = (d-2)\sigma_d^2$, is known to be adaptive, in the sense that it asymptotically recovers
the minimax risk for any $c>0$, without the knowledge of $c$. Hence, a
natural question is: under what more general distributional assumptions is the James-Stein estimator adaptive to $c$? That is, for which collections of distributions $\left\{ \mathcal{P}_c : c>0 \right\}$, where $\mathcal{G}_d (c) \subset \mathcal{P}_c \subset \mathcal{P}_d (c)$, does the James-Stein estimator recover the asymptotic minimax risk for any fixed value of $c>0$? We answer this question with the following results, starting with the use of Stein kernels.

We note that a variance decay of rate $\sigma^2/d$ corresponds to a decay on the error of the form ${\bm Y}/\sqrt{d}$, and in the proof of the following result, that is given in Supplement C, we invoke Theorem \ref{thm:shrinkage.non_indep} with that form for the error. Note, correspondingly, that scaling corresponds to a decay on the Stein kernel $T$ of ${\bm Y}$ to $T/d$.

\begin{theorem} \label{prop:adpt_dep}
Let ${\bm X}-{\bm \theta}$ be a mean zero vector with covariance matrix $\sigma_d^2\Id$ with $\sigma_d^2 = \sigma^2/d$, and let  $\Tk/d$ be a Stein kernel for ${\bm X} -  \bsy{\theta}$ in the sense of \eqref{def:multivariate_stein_kernel}. 
Then, with $\lambda=(d-2)\sigma_d^2$, it holds that 
\bea \label{bd:thm:prop:adpt_dep}
E\left[\|S_\lambda({\bm X})-\bsy{\theta}\|^{2}\right]&\leq &
d\sigma_{d}^{2}-\frac{(d-2)^{2}\sigma_{d}^{4}}{\|\bsy{\theta}\|^{2}+d\sigma_{d}^{2}} + 2B_\lambda,
\ena
where $B_\lambda$ is as in \eqref{eq:upper.bound.B.lambda}. If
the conditions in \eqref{eq:3conditions.thm:shrinkage.non_indep} or in \eqref{eq:thm:shrinkage.poinca.estimates} hold for ${\bm X}$ and $\Tk$, then $B_\lambda$ is of order $o(1)$.
\end{theorem}

Again Lemma \ref{lem:inverse.mean} can 
be applied to obtain bounds on the expectations of the inverse moments under Model \ref{model:indep}. In the case of a log-concave vector, Theorem \ref{prop:adpt_dep} gives the following corollary, proved in Supplement C.

\begin{corollary}\label{cor:mse.scaled.le.1}
Let $\mathcal{P}(c)$ be the set of distributions of vectors ${\bm X}$ belonging to the set $\mathcal{P}_d(c)$ defined above, such that ${\bm X}-{\bm \theta}$ is a mean zero isotropic log-concave vector. 
Then the James-Stein estimator $S_\lambda({\bm X})$ in \eqref{eq:shrinkage},
with $\lambda = (d-2)\sigma_d^2$, is asymptotically adaptive to $c$ in the set $\mathcal{P}(c)$, in the sense that it recovers
the minimax risk over $\mathcal{P}(c)$ for any $c>0$, without the knowledge of $c$, when the dimension $d$ grows to infinity.
\end{corollary}

We now turn to using the zero bias distribution to obtain parallel results. For ${\bm \theta} \in \mathbb{R}^d$ and $C_4,C_{-2}$ positive constants, 
let $\mathcal{P}({\bm \theta},C_4,C_{-2})$ be the set of distributions of vectors ${\bm X}={\bm Y}+{\bm \theta}$ that satisfy the assumptions of Part \ref{item:mix.Blambda*} of Theorem \ref{zerobiasshrinkage},  where 
 ${\bm Y}_s, s \in \mathcal{S}$ has covariance matrix $\sigma_d^2 \Id$ and $\sup_{s \in \mathcal{S}} \max_{1 \le i \le d}E[Y_{s,i}^4]\leq C_4$.

\begin{theorem}\label{thm_adpt_indep}
If the distribution of ${\bm X}$ is a member of  $\mathcal{P}({\bm \theta},C_4,C_{-2})$, 
then with $\lambda=(d-2)\sigma_d^2$, 
\begin{multline} \label{eq:MSE.bound.adapt}
E\left[\|S_\lambda({\bm X})-\bsy{\theta}\|^{2}\right]\leq 
\frac{  \sigma^2 \|\bsy{\theta}\|^2} {\sigma^2+\|\bsy{\theta}\|^2}
\left( 1 + \frac{4\sigma^2}{d\|{\bm \theta}\|^2} \right)
+L \lambda \left(
\frac{1}{d}+\frac{\| \bsy{\theta} \|_1}{d^2}+
\frac{ \| \bsy{\theta} \|_2^2}{d^3} \right),
\end{multline} 
where the constant $L$ only depends on $\sigma^2,C_4$ and $C_{-2}$. Moreover, letting
\beas
\mathcal{P}(c) =\left\{ P\in \mathcal{P}({\bm \theta},C_4,C_{-2}): \|{\bm \theta}\| \le c \right\}  
\enas
the James-Stein estimator $S_\lambda({\bm X})$ in \eqref{eq:shrinkage} is asymptotically adaptive to $c$ in the set $\mathcal{P}(c)$, in the sense that it recovers
the minimax risk over $\mathcal{P}(c)$ for any $c>0$, without the knowledge of $c$, when the dimension $d$ grows to infinity.
\end{theorem}

Theorem \ref{thm_adpt_indep} shows that the James-Stein estimator is adaptive in this case, in the sense that it asymptotically recovers the minimax risk over $\ell_2-$balls, without requiring that $c$ be known. Its proof can be found in Supplement C.

\section{Multivariate Zero Bias}
\label{sec:prop:mvariate.zbias} We collect the properties of the zero bias distribution in the following result, proved in Supplement D. We first note that when ${\bm Y}$ satisfies \eqref{eq:mgalelike} its mean is necessarily zero. 

\begin{proposition} \label{prop:mvariate.zbias}
Let ${\bm Y} \in \mathbb{R}^d$ have mean zero and positive definite covariance matrix $\Sigma$. 

\begin{enumerate}

\item \label{prop:mvariate.zbias.ex.un} If \eqref{eq:mgalelike} holds
then $\Sigma={\rm diag}(\sigma_1^2,\ldots,\sigma_d^2)$, 
and the laws for random vectors ${\bm Y}^i, i=1,\ldots,d$ satisfying \eqref{eq:mzb.identity.theta.zero} exist and are unique. Conversely, if \eqref{eq:mzb.identity.theta.zero} holds then \eqref{eq:mgalelike} holds.

When \eqref{eq:mgalelike} holds the collection of zero bias random vectors may be constructed as follows. 
With $\nu$ the distribution of ${\bm Y}$, for each $i=1,\ldots,d$ let ${\bm Y}^{\square,i}$ have distribution
\begin{align} \label{eq:square.bias.measures}
d\nu^{\square, i}=\frac{y_i^2}{\sigma_i^2}d\nu.
\end{align}
For ${\bm y} \in \mathbb{R}^d, u \in \mathbb{R}$ and $i=1,\ldots,d$, let
\begin{align*} 
D_{i,u}{\bm y}=(y_1,\ldots,y_{i-1},uy_i,y_{i+1},\ldots,y_d)^\transpose,
\end{align*}
that is, $D_{i,u}{\bm y}$ is formed by multiplying the $i^{th}$ component of ${\bm y}$ by $u$.  Then for $U_i$ a uniformly distributed variable on $[0,1]$, independent of ${\bm Y}^{\square,i}$,
the collection of vectors
\begin{align}\label{eq:X^i.by.D}
{\bm Y}^i=D_{i,U_i} {\bm Y}^{\square, i} \quad \mbox{for $i=1,\ldots,d$}
\end{align}
satisfies \eqref{eq:mzb.identity.theta.zero}. \\

\item \label{prop:mvariate.zbias.support} When ${\bm Y}$ satisfies \eqref{eq:mgalelike}, and $S$ and $S^i$ are the supports of ${\bm Y}$ and ${\bm Y}^i$ respectively,  
then with $\operatorname{cl}$
denoting closure, 
\begin{align} \label{eq:support.S^i}
S^i = \operatorname{cl}(U^i(S))
\qmq{where}
U^i(S)=\left\{D_{i,u}{\bm y}: {\bm y}\in S, y_i \not =0, u \in [0,1] \right\}.
\end{align}
\\

\item \label{prop:mvariate.zbias.sum} When
${\bm Y} = \sum_{j=1}^n {\bm Y}_j$
where ${\bm Y}_j, j=1,\ldots, n$ are independent, mean zero $\R^d$ valued random vectors with covariance matrices $\Sigma_j={\rm diag}(\sigma_{j,1}^2,\ldots,\sigma_{j,d}^2)$ and associated zero bias vectors ${\bm Y}_j^i,i=1,\ldots,d$, then ${\bm Y}$ has zero bias vectors ${\bm Y}^i, i=1,\ldots,d$ whose distributions are the mixtures of ${\bm Y}-{\bm Y}_j+{\bm Y}_j^i, j=1,\ldots,n$, where ${\bm Y}_j^i$ is the $i^{th}$ zero bias vector of ${\bm Y}_j$, taken independently of ${\bm Y}_k, k \not = j$ with probability $\sigma_{j,i}^2/\sigma_i^2$, where $\sigma_i^2=\sum_{j=1}^n \sigma_{j,i}^2$.
\\[1ex]

\item \label{prop:mvariate.zbias.mix} When ${\bm Y}$ is the $\mu$ mixture 
of $\{{\bm Y}_s\}_{s \in \mathcal{S}}$, a collection of mean zero random vectors in $\mathbb{R}^d$ with ${\bm Y}_s, s \in \mathcal{S}$ having non singular covariance matrices ${\rm Var}({\bm Y}_s)={\rm diag}(\sigma_{s,1}^2,\ldots,\sigma_{s,d}^2)$ and zero bias vectors ${\bm Y}_s^i, i=1,\ldots,d$, then 
the zero bias distribution of ${\bm Y}$ exists, and the distribution of ${\bm Y}^i$ is the $\nu^i$ mixture of ${\bm Y}_s^i, s \in \mathcal{S}$, 
where $d \nu^i/d \mu = \sigma_{s,i}^2/\sigma_i^2$
where $\sigma_i^2={\rm Var}(Y_i)$.
In particular, $\nu^i=\mu$
if and only if $\sigma_{s,i}^2$ is a constant $\mu$ a.s. over $s \in \mathcal{S}$.
\\[1ex]

\item \label{prop:pos.mult.trans} When ${\bm Y}=A{\bm U} \in \mathbb{R}^d$ for some mean zero ${\bm U} \in \mathbb{R}^m$
with positive definite covariance matrix $\Gamma$ and whose zero bias vectors $\bm{U}^{kl}$ exist for all $1 \le k,l \le m bb$ for which $\gamma_{kl} \not =0$,
$A=(a_{ik})_{1 \le i \le d, 1 \le k \le m} \in \mathbb{R}^{d \times m}$ 
 and  $\sigma_{ij}:={\rm Cov}(Y_i,Y_j) \ge 0$ for all $1 \le i,j \le d$, and for all $i,j$ such that $\sigma_{ij}>0$ we have 
$a_{ik}\gamma_{kl}a_{jl} > 0$ for $1 \le k,l \le m$, 
then the zero bias vectors for ${\bm Y}$ exist, with the distribution of ${\bm Y}^{ij}$ for such $i,j$ pairs obtained by mixing the distributions of $A{\bm U}^{kl}$ with measure
$\mu_{ij}(kl)=a_{ik}\gamma_{kl}a_{jl}/\sigma_{ij}$.
\\[1ex]

\item \label{prop:mvariate.zbias.density} When ${\bm Y}$ satisfies \eqref{eq:mgalelike} and has density $p({\bm y})$ then for all $i=1,\ldots,d$, the integral
\begin{align}\label{def:p^i}
p^i({\bm y})=\frac{1}{\sigma_i^2}\int_{y_i}^\infty u p(y_1,\ldots,y_{i-1},u,y_{i+1},\ldots,y_d)du
\end{align}
exists a.e. 
and $p^i(\cdot)$ is the density of ${\bm Y}^i$.
If there exists $g \in L^1(\mathbb{R})$ such that $|y_i|p({\bm y})\le g(y_i)$ then $p^i({\bm y})$ is bounded over $\mathbb{R}^d$.
\end{enumerate}
\end{proposition}

Taking $U \sim \mathcal{U}[0,1]$ and ${\bm Y}$ independent, item \ref{prop:mvariate.zbias.ex.un} provides the  alternative identity
\begin{align*}
E[f({\bm Y}^i)] = \frac{1}{\sigma_i^2}E[Y_i^2f(D_{i,U} {\bm Y})] 
\end{align*}
to 
\eqref{eq:mzb.identity}
for computing expectations with respect to ${\bm Y}^i$. Generally, when ${\bm X}^i$ exists for ${\bm X}={\bm Y}+{\bsy \theta}$, we obtain 
\begin{align} \label{eq:mult.sq.bias.theta}
E[f({\bm X}^i)] = \frac{1}{\sigma_i^2}E[(X_i-\theta_i)^2f(D_{i,U} ({\bm X}-{\bsy \theta})+{\bsy{\theta}})].
\end{align}

The necessity for excluding $y_i=0$ in \eqref{eq:support.S^i} can be made apparent by considering the zero bias distribution of the measure in $\R^2$ that puts equal mass on the five points $(\pm 1,\pm 1)$ and $(0,0)$, and for the necessity of taking the closure, consider ${\bm Y}$ with the uniform distribution on the boundary of the $L^1$ ball in $\R^2$.

\section{Boundedness in mean of inverse norms}\label{ssec_inv_norms}
In this section we present two results to bound the expectation of powers of inverse norms of a vector ${\bm X}$, a quantity on which our results here depend. The first provides a simple sufficient condition on the moment generating function of the squared norm of the vector whose inverse moment is being taken, and the other that can be applied when the vector has a log-concave distribution. The proof of Lemma \ref{lem:inverse.mean} can be found in Supplement E.

Before stating the first result, we recall that random variables $V_1,\ldots,V_d$ are said to be negatively associated (see \cite{JoagProschan1983}) when for all disjoint subsets $A$ and $B$ of $\{1,\ldots,d\}$,
$$
{\rm Cov}(f(V_i, i \in A),g(V_i, i \in B)) \le 0 
$$
when $f$ and $g$ are both non-decreasing (or both non-increasing) functions. Clearly, collections of independent random variables are negatively associated.

\begin{lemma} \label{lem:inverse.mean}
Let $S_d, d \ge 1$ be a non-negative random variable such that for some $\mu, q$ and $C$, all positive, 
\bea \label{eq:MSd.bound}
M_d(t) \le \frac{C}{(1-\mu t/q)^{qd}} \qmq{for all $t \le 0$, where}  M_d(t)=E[e^{tS_d}].
\ena
Then for all $m \ge 1$, if $d \ge 2m/q$ there exists a constant $C_{\mu,m}$, depending only on $\mu$ and $m$ such that 
\begin{align*} 
E\left[ \frac{d}{S_d} \right]^m \le C_{\mu,m}.
\end{align*}

When $S_d=\sum_{i=}^d V_i$, a sum of  non-negative, negatively associated random variables such that for some $\mu$ and $q$ positive the moment generating functions of $V_i,i,\ldots,d$ obey the bound \eqref{eq:MSd.bound} with $C=1$ and $d=1$, then  \eqref{eq:inv.mom.Sd}a holds for all $d \ge 2m/q$.
\end{lemma}

\begin{remark} \label{rem:gamma.inverse.moment} 
When $S_d=\|{\bm X}\|^2$ for ${\bm X} \in \mathbb{R}^d$ then Lemma \ref{lem:inverse.mean} provides a sufficient condition for the satisfaction of the negative moment condition \eqref{eq:inv.mom.Sd}a in terms of the moment generating function of $S_d$. 
The lemma can be applied to vectors having negatively associated components, and hence in particular to those with independent components, and then by the extension as done in Corollary \ref{cor:shrinkage.via.T.univ}, for vectors having non-independent coordinate distributions that are covered by Model \ref{model:indep}.

For instance, when the marginal distribution $\mathcal{L}(Y)$ of the components of the error vector ${\bm Y}$ is $\mathcal{N}(0,\sigma^2)$ then $V=(\theta+Y)^2$ has a non-central $\chi^2$ distribution with moment generating function
\begin{align*}
M_V(t)=E[e^{t(\theta+Y)^2}]= \frac{\exp\left(
\frac{t\theta^2}{1-2t\sigma^2}
\right)}{(1-2t\sigma^2)^{1/2}}.
\end{align*}
As the numerator is bounded by 1 for all $t \le 0$, and $M_V(t)$ does not otherwise depend on $\theta$, we see that \eqref{eq:MSd.bound} is satisfied for 
all $\theta \in \mathbb{R}$   with $C=1,q=1/2,\mu=\sigma^2$ and $d=1$. If the distribution of the absolute value of $\theta$ plus the coordinate error $Y$ stochastically dominates the same quantity for a $\sigma$ scaling  of the standard normal, that is, when $|\theta+\sigma Z | \le_{\rm st} |\theta + Y|$, the same bound will hold.

Though condition \eqref{eq:MSd.bound} appears related to the sub-gamma property of random variables, that condition is concerned about the behavior of the moment generating function in a positive neighborhood of zero. Note also that when any mass of a distribution is moved to zero it only would `help' the satisfaction of the sub-gamma property, but \eqref{eq:MSd.bound} will immediately be violated. Indeed, if $V$ has a point mass of probability $p>0$ at zero, 
then $M_V(t) \ge p$ for all $t \le 0$, and hence $M_V(t)$ cannot tend to zero as $t \rightarrow -\infty$, as does any moment generating function that satisfies \eqref{eq:MSd.bound}.
\end{remark}

\begin{remark} \label{rem:followingex:finite.support}
We continue the discussion in Example \ref{ex:finite.support} where the error vector ${\bm Y}$ has a discrete distribution with finite support $S_{\bm Y}$. Taking any $\delta \in (0,1]$ and 
\begin{align}\label{def:finite.support.shifts}
{\bm \theta} \in \bigcap_{{\bm y} \in S_{\bm Y}} \bigcup_{I \subset \{1,\ldots,d\}, |I| \ge \delta d}\{
{\bm \psi}: \psi_i \not = -y_i, i \in I
\},
\end{align}
for any ${\bm y} \in S_{\bm Y}$ there must exist a set of indices $I \subset \{1,\ldots,d\}$ satisfying $|I| \ge \delta d$ such that
$$
\tau_{\bm y} := \min_{i \in I}|\theta_i+y_i| >0.
$$
Now, letting $\tau=\min_{{\bm y} \in S_{\bm Y}} \tau_{\bm y}$,
a quantity that must be positive as $S_{\bm Y}$ is finite, we obtain $\|{\bm X}\|^2=\|{\bm \theta}+{\bm Y}\|^2 \ge \delta d \tau^2$ almost surely. With only minor changes to the argument to handle ${\bm X}^{\neg i}$,we see that the negative moment condition  \eqref{eq:inv.mom.Sd}b is satisfied. 
Similar to the conclusion in Example \ref{ex:finite.support}, the exceptional set of shifts ${\bm \theta}$ not satisfying \eqref{def:finite.support.shifts} has Lebesgue measure zero. 
\end{remark}

\begin{remark} \label{rem:Li85}
The assumption that
there exists a positive constant $K$ such that for any $a\geq0$ and any $i\in \left\{ 1,\ldots,d \right\}$,
\beas 
\sup_{u\in \mathbb{R}} P\left\{ u-a \leq Y_i \leq u+a \right\} \leq Ka 
\enas
is made in the proof of Theorem 3.1 in \cite{Li85} in order to tackle negative moment estimates, in the case where the components of the observation vector are independent, an instance subsumed by Model \ref{model:indep}. 
The main example achieving this condition is the case where $Y_i, i = 1,\ldots,d$ has a uniformly bounded density with respect to  Lebesgue measure, with a bound independent of $i$. In such a case, the vector ${\bm Y}$ also has a bounded density on $\mathbb{R}^d$, since its coordinates are independent, and negative moments $\mathbb{E}[\Vert {\bm X}\Vert^{-2m}]$ are finite for any $m\geq 1$ for dimensions $d\geq 2m+1$. We note that \cite{Li85} does not need to control a decay rate of these moments with respect to the ambiant dimension - as we do for instance in \eqref{eq:inv.mom.Sd} -, but there are two essential differences between our analysis and that in \cite{Li85}: first, the setting of 
\cite{Li85}
is asymptotic only and second, \cite{Li85} 
considers the consistency in probability of SURE towards the loss, while we investigate the bias of SURE compared to the risk - which is the integrated loss. 
\end{remark}

We now give a bound on means of inverse norms for log-concave random vectors, which is a corollary of \cite[Theorem 6.2]{AdGuEtal} (itself a variant of results of \cite{Pao12}). 

\begin{proposition} \label{eq:prop.log.concave}
There is a dimension $d_0$ such that for any log-concave distribution in $\mathbb{R}^d$ with 
covariance matrix $\sigma^2\Id$ for some $\sigma^2 > 0$, for 
$d \geq d_0$ we have
$$E[||{\bm X}||^{-6}] \leq cd^{-3}$$
where $c$ is a constant independent of the distribution and of $d$.
\end{proposition}

In this statement, the exponent $6$ does not play a significant role, beyond affecting the value of $c$ and the dimension $d_0$. The value of $d_0$ depends on the values of some universal constants used in \cite{AdGuEtal}, that were not made explicit, though it must be larger than $6$. 

Strictly speaking, \cite[Theorem 6.2]{AdGuEtal} is only given for centered random variables. In the non-centered case, we can consider the projection ${\bf \widetilde{X}}$ of ${\bm X}$ onto the $(d-1)$-dimensional subspace orthogonal to the mean vector $\bsy{\theta}$. Then ${\bf \widetilde{X}}$ is a $(d-1)$-dimensional centered log-concave vector, still with covariance matrix $\sigma^2\Id$. We can then apply the centered result to ${\bf \widetilde{X}}$, and use the fact that $||{\bm X}||^{-8} \leq ||{\bf \widetilde{X}}||^{-8}$. 


\bigskip
\begin{acks}[Acknowledgments]
This research was started during the workshop {\it Stein's method and applications in high-dimensional statistics} at the American Institute of Mathematics (AIM), San Jose, California, and we acknowledge the generous support and the fertile research environment provided by AIM. MF was additionally supported by ANR-11-LABX-0040-CIMI within the program ANR-11-IDEX-0002-02, as well as Projects EFI (ANR-17-CE40-0030) and MESA (ANR-18-CE40-006) of the French National Research Agency (ANR). GR is partially supported by EPSRC grants EP/R018472/1 and EP/T018445/1. The authors sincerely thank the associate editor and our two reviewers for their hard work in providing us with detailed and insightful reviews. The fourth author warmly thanks Lionel Truquet for instructive discussions related to martingale difference random fields. 
\end{acks}

\begin{supplement}
\stitle{Organisation}
\sdescription{We gather in this supplementary material the proofs of the main results stated in the main part of the paper, as well as some technical details for some remarks and examples and further considerations. Each section of the supplement, from A to E, is related to a section of the main part of the article, from \ref{sec:SteinIdentity} to \ref{ssec_inv_norms}, and follows the order of appearance of the results, examples and remarks. Supplement F provides in addition some insights towards a generalization of our results related to Stein kernels to other situations than shrinkage, via an analysis of possible extensions of Assumption \ref{assumptionW.kernel}.
}
\end{supplement}

\begin{supplement}\label{suppA}
\stitle{Supplement A: Proofs for Section \ref{sec:SteinIdentity}}
\sdescription{
\begin{proof}[Proof of Lemma \ref{lem:ker.sat.assumption}]  By \eqref{eq:Jacg0}, 
\begin{align*}
||{\bm g}_0||_{W^{1,2}(\nu)}^2 = ||{\bm g}_0||_{L^2(\nu)}^2 + ||\nabla {\bm g}_0||_{L^2(\nu)}^2=
E_\nu[\|{\bf X}\|^{-2}+d\|{\bf X}\|^{-4}].
\end{align*}
With $p$ the density, for positive $c$ and $t$ such that $p({\bm x}) \le c$ for all $\|{\bm x}\|\le t$, the proof of the 
first assertion of the lemma is completed by noting that whenever $d>q$ 
\begin{multline*}
E_\nu\|{\bf X}\|^{-q} = E_\nu[\|{\bf X}\|^{-q}; \|{\bf X}\|\le  t] + E_\nu[\|{\bf X}\|^{-q}; \|{\bf X}\|> t] \\
\le c\int_{\|{\bf x}\| \le t} \frac{1}{\|{\bm x}\|^q}d{\bm x}+t^{-q} = \frac{2c\pi^{d/2}}{\Gamma(d/2)}\int_0^t r^{d-1-q} dr+t^{-q} < \infty. 
\end{multline*}
The final claim regarding translates follows directly. 
\end{proof}
\begin{proof}[Proof of Lemma \ref{lem:zb.sat.assumption}]
The proof of the first part of the lemma may proceed as for Lemma \ref{lem:ker.sat.assumption} using 
Part \ref{prop:mvariate.zbias.density} of Proposition \ref{prop:mvariate.zbias}, and noting that the condition that $|y_i| p({\bm y})$ is dominated by an $L^1$ function is invariant under translation.
Next, as ${\bm g}_0$ and its Jacobian in \eqref{def:f.g0} and \eqref{eq:Jacg0} respectively are bounded outside any neighborhood of zero, 
when the supports of $\nu$ and $\nu^i,i=1,\ldots,d$ have empty intersection with a ball around the origin of radius $\delta$, we may again argue as in the proof of Lemma \ref{lem:ker.sat.assumption}, though the integral over the volume centered at zero, in this case, vanishes.  
Lastly, when \eqref{eq:x.neg.i.bdbl.2delta} holds then $S$ has empty intersection with a ball of positive radius centered at the origin, and for every $i=1,\ldots,d$ the same must hold for any point ${\bm z} \in U^i(S)$ given in \eqref{eq:support.S^i}, and hence for its closure. Hence, the final claim holds by  \eqref{eq:support.S^i}.
\end{proof}
}
\end{supplement}

\begin{supplement}
\stitle{Supplement B: Proofs for Section \ref{sec:shrinkage}}
\sdescription{
\begin{proof}[Proof of Theorem \ref{thm_shrinkage_SLC}]
Note first that, since the eigenvalues of the Stein kernel $T$ are uniformly bounded from above, Identity (\ref{def:multivariate_stein_kernel}) holds for any ${\bm f} \in L^2(\nu)$ such that $\nabla{{\bm f}}\in L^1(\nu)$.  This is ensured for ${\bm f}=- \lambda{\bm g}_0$ since we assume that $E_{\bsy{\theta}}\left[ \|{\bm X}\|^{-2} \right]<\infty$. The first term in the mean squared error expression for $S_\lambda({\bm X})$ in \eqref{eq:MSE.expansion.3.1+3.2} agrees with the first term of the bound \eqref{eq:risk_T_infinity}.
Applying the Stein identity \eqref{def:multivariate_stein_kernel} the second term of \eqref{eq:MSE.expansion.3.1+3.2} is given by twice  \eqref{eq:E[Sig.Grade]}  with $M=T$.
For the first term there we use $\operatorname{Tr}(\Tk)\geq \alpha_- d$, and for the second that $\kappa \le \alpha_+$. 
Now incorporating the final term of \eqref{eq:MSE.expansion.3.1+3.2} gives \eqref{eq:risk_T_infinity}.
It remains to prove the second part of the theorem, concerning strongly log-concave vectors with controlled potential. For $d\geq 3$, $E_{\bsy{\theta}}\left[ \|{\bm X}\|^{-2} \right]<\infty$ since the density of ${\bm X}$ is of the form $\exp(-\varphi)\exp(-c_-\|{\bm X}-\bsy{\theta}\|^{2}/2)$ with $\varphi$ a finite convex function, which implies that integrability is controlled by the Gaussian case. 
For the Stein kernel $T$ constructed (via so-called moment maps) in \cite{Fathi19}, which indeed exists for log-concave vectors with smooth potential, see \cite[Theorem 2.3]{Fathi19}, it
holds that $\|T\|_{op}\leq 1/c_-$ (\cite[Corollary 2.4]{Fathi19}). In addition, from the proof of Theorem 3.4 in \cite{KolKos17}, 
$\operatorname{Tr}(\Tk)\geq d/(1+c_+)$, which concludes the proof.
\end{proof}
\begin{proof}[Technical details for Example \ref{ex:Student}]
First, to bound the inverse fourth moment $E_{\bm \theta}|| \mathbf{X}||^{-4}$, we represent the distribution of ${\bm Y}$ by the Gamma ${\gamma} \sim \Gamma(k/2,k/2)$ variance mixture  ${{\bm Y}} =_d {\gamma^{-1/2}}\sigma {\bm N}$ of an independent standard ${\bm N} \sim \mathcal{N}_d(0, \Id)$ normal vector, using the shape-rate parameterization of the Gamma distribution, that is, where $\gamma$ has density proportional to $\gamma^{k/2-1} e^{-k  \gamma/2}$,
(see for example Eq.\,9 in \cite{kirkby2019moments}).  
Thus,
\begin{align}\label{eq:gamma2.x.cond.exp}
E_{\bm \theta}\|{\bm X}\|^{-4} = E\|{\bm \theta}+\gamma^{-1/2}\sigma {\bm N}\|^{-4}
= \frac{1}{\sigma^4}E\left[\gamma^2E\left[\|\frac{\sqrt{\gamma}}{\sigma}{\bm \theta}+{\bm N}\|^{-4}\Bvert \gamma\right] \right].
\end{align}
Given $\gamma$, and relying on the fact that we have taken $d=2m \ge 6$, hence even, the squared norm inside the conditional expectation follows a non-central chi-squared distribution with non-centrality parameter $\kappa=\gamma \|{\bm \theta}\|^2/\sigma^2$.
As $m \ge 3$,
we can apply (29.32 c) from \cite{jkb1994}
to obtain
\begin{multline*}
     E \left\{ \|{\bm N}+ \frac{\sqrt{\gamma}}{\sigma} {\bm \theta}\|^{-4} {\Bvert \gamma} \right\} =  
     \frac14 e^{- \kappa} \sum_{j=0}^\infty  \frac{\kappa^j }{j!(j+m -2) (j+m-1)} 
   \\ \le \frac{1}{4 (m-2)(m-1)} = 
   \frac{1}{(d-2)(d-4)}. 
  \end{multline*} 
Applying this bound in \eqref{eq:gamma2.x.cond.exp},
and then using 
that for all $p>-\alpha$ the $p^{th}$ 
moment of $\Gamma(\alpha, \beta)$ is $ \Gamma(\alpha+p)/{\beta^p}\Gamma(\alpha)$, we obtain
$$
E[{\gamma^2}]=\left(\frac{2}{k}\right)^2 \frac{\Gamma(k/2+2)}{\Gamma(k/2)}=\frac{k+2}{k},
$$
yielding,  with $\sigma^2 = k/(k-2)$,
\begin{align*}
E_{\bsy{\theta}}[d^2\|{\bm X}\|^{-4}]
\le   \frac{d^2 {(k-2)^2} (k+2)}{  (d-2)(d-4) 
{k^3}}.
\end{align*}
This term is thus bounded uniformly in $d$ \and $k$. For odd dimensions $d$ the calculations would be similar, though more involved, see \cite{BoJuYa84}.
The terms  $\operatorname{Var}(\operatorname{Tr}(\Tk))$ and $E[\| \Tk -\sigma^2 \Id\|^2]$ can be bounded directly. As ${\bf Y}$ has mean zero, the second moment of its components is the variance $\sigma^2=k/(k-2)$. Moments of order 4 we calculate using the mixture property (see also \cite{kirkby2019moments}), as follows. Let ${\bm m}=(m_1, \ldots, m_d)$ be such that
$m:=\sum_i m_i < k$. Then for non-negative even integers $m_i \ge 0, i=1,\ldots,d$,
\begin{multline*}
E [ {\bm Y}^{{\bm m}}] :=  E \left\{ \prod_{i=1}^d Y_i^{m_i} \right\} 
 = E \{ \gamma^{-m/2}  \sigma^m \prod_{i=1}^d N_i^{m_i}  \} \\
=  \left( \frac{k}{k-2}\right)^{\frac{m}{2}} \frac{\prod_{i=1}^d(m_i!)}{2^{\frac{m}{2}} \prod_{j=1}^d \left(\frac{m_i}{2}\right)! }
 \left(\frac{k}{2} \right)^{\frac{m}{2}} \frac{\Gamma \left(\frac{k - m}{2} \right)}{\Gamma \left( \frac{k}{2}\right)} .
 \end{multline*} 
In particular, 
\begin{align*}
EY_i^4 =\frac{3 k^4}{(k-2)^3(k-4)} \qmq{and}
E[Y_1^2Y_2^2]= \frac{ k^4}{(k-2)^3(k-4)};
\end{align*}
we assume that $k \ge 5$ {in order} {that these quantities be finite}. 
Hence
\begin{multline*}
{\rm Var}\left(\sum_i Y_i^2\right) = E\left(\sum_i Y_i^2\right)^2-\left(\sum_i EY_i^2\right)^2\\
=\frac{dk^4}{(k-2)^3}\left[\frac{d+2}{k-4}- \frac{d}{k-2}\right]
= \frac{2 dk^4(d+k-2)}{(k-2)^4(k-4)}
\end{multline*}
and now, using \eqref{eq:T.student},
 \begin{eqnarray*} 
\operatorname{Var}(\operatorname{Tr}(\Tk))
=  \frac{d^2}{(d+k-2)^2} \operatorname{Var} \left( \sum_i Y_i^2 \right)  
=  \frac{2 d^3 k^4}{(d+k-2) (k-2)^4(k-4)} .  \end{eqnarray*} 
Similarly, using that $E[T]=\sigma^2 \Id$,
\beas
E[\| \Tk -\sigma^2 \Id\|^2]
= \frac{d}{(d+k-2)^2} \operatorname{Var} \left( \sum_i Y_i^2 \right)
=\frac{2 d^2 k^4}{(d+k-2) (k-2)^4(k-4)}.
\enas
\end{proof}
\begin{proof}[Proof of Theorems \ref{thm:shrinkage.non_indep} and \ref{thm:shrinkage.poinca}]
The first term in the mean squared error expression for $S_\lambda({\bm X})$ in \eqref{eq:MSE.expansion.3.1+3.2} agrees with the first term of the bound \eqref{eq:Slambdarisk.upper.bound.quantitative}.
Applying the Stein identity \eqref{def:multivariate_stein_kernel}  
on the second term of \eqref{eq:MSE.expansion.3.1+3.2}, and writing $\Tk$ as short for $\Tk_{{\bm X}-{\bsy{\theta}}}$,
yields twice 
\begin{multline} \label{eq:LHS.Shinkage.T}
E_{\bsy{\theta}}\left[\left\langle {\bm X}-\bsy{\theta},{\bm f}({\bm X})\right\rangle\right]= E_{\bsy{\theta}}[\langle \Tk, \nabla {\bm f}({\bm X}) \rangle] \\= E_{\bsy{\theta}}[\langle \Sigma  , \nabla {\bm f}({\bm X}) \rangle]+ E_{\bsy{\theta}}[\langle T- \Sigma, \nabla {\bm f}({\bm X}) \rangle].
\end{multline}
Bounding twice the first term by twice the bound \eqref{eq:E[Sig.Grade]} with $M=\Sigma$
and combining with the last term of \eqref{eq:MSE.expansion.3.1+3.2} yields the second term in the bound of \eqref{eq:Slambdarisk.upper.bound.quantitative}. The inequality \eqref{eq:Slambdarisk.upper.bound.quantitative} holds for $B_\lambda$ in \eqref{eq:defB}
as it upper bounds the final term  of \eqref{eq:LHS.Shinkage.T}. 
To obtain the bound \eqref{eq:upper.bound.B.lambda} from the final term in \eqref{eq:LHS.Shinkage.T} we apply the second equality in \eqref{eq:E[Sig.Grade]} with $M=T-\Sigma$ and apply the Cauchy-Schwarz inequality to the resulting two terms. 
For the first we note that $E_{\bsy{\theta}}[\Tk]=\Sigma$ by \eqref{def:multivariate_stein_kernel}, and for the second that
\begin{align}\label{eq:bound.for.Tsigm.inner.J}
|{\rm Tr}((T-\Sigma){\bm X}{\bm X}^\transpose)|=\vert \langle \Tk-\Sigma, {\bm X}{\bm X}^\transpose\rangle \vert \le \|{\bm X}\|^2 \| \Tk-\Sigma\|_{\HS}.
\end{align}
Thus, moving
to the final assertions of  Theorem \ref{thm:shrinkage.non_indep}, it is easy to see that the upper bound on $B_\lambda$ given in \eqref{eq:upper.bound.B.lambda} is $o(d)$ when the  conditions in \eqref{eq:3conditions.thm:shrinkage.non_indep} hold and $\lambda =O(d)$. Next, the second term in the bound \eqref{eq:Slambdarisk.upper.bound.quantitative} is non-positive when $\lambda \in [0,2 (\operatorname{Tr}(\Sigma) -2 \kappa)]$, and as $B_\lambda=o(d)$, $S_\lambda$ cannot have a larger asymptotic risk than $S_0$, given by the first term in the bound, which has order $O(d)$. The final claim of  Theorem \ref{thm:shrinkage.non_indep} was shown in the discussion immediately above the theorem statement.
Similarly, to prove Theorem \ref{thm:shrinkage.poinca} it suffices to exhibit an upper bound on $B_\lambda$, which can be seen to be the absolute value on the right hand of the last equality in expression \eqref{eq:E[Sig.Grade]} with $M=T-\Sigma$,
and show that it is $o(d)$
under the given conditions. Squaring the first expression inside the parentheses of that term, since $E_{\bsy{\theta}}[\operatorname{Tr}(T) - \operatorname{Tr}(\Sigma) ] = 0$, we obtain the bound
\begin{multline*}
\left( E_{\bsy{\theta}}[(||{\bm X}||^{-2} -E_{\bsy{\theta}}[||{\bm X}||^{-2}]) (\operatorname{Tr}(T) - \operatorname{Tr}(\Sigma))]\right)^2 \leq \operatorname{Var}_{\bsy{\theta}}(||{\bm X}||^{-2}) \operatorname{Var}(\operatorname{Tr}(T))\\
\leq 
4C_P E_{\bsy{\theta}}[||{\bm X}||^{-6}] 
\operatorname{Var}(\operatorname{Tr}(T))
\le 
4d^{-1}C_P E_{\bsy{\theta}}[d^3||{\bm X}||^{-6}] 
\kappa \left(C_P -  \frac{\operatorname{Tr} (\Sigma)}{\operatorname{rank}(\Sigma)} \right),
\end{multline*}
by first applying the Cauchy Schwarz inequality to the first term, followed by the Poincar\'e inequality \eqref{eq:Poin.ineq} with $f({\bm x})=||{\bm x}||^{-2}$, and then
using \eqref{eq:disc.boundedby.CP-sigma^2} to obtain
$$\operatorname{Var}(\operatorname{Tr}(\Tk)) 
\leq d\sum_{i=1}^d E[(T_{ii}-\sigma_{ii})^2]  \le d E[||\Tk - \Sigma||_{\HS}^2] \leq \kappa d^2 \left(C_P -  \frac{\operatorname{Tr} (\Sigma)}{\operatorname{rank}(\Sigma)} \right).$$ 
Taking a square in the remaining term in \eqref{eq:E[Sig.Grade]}, recalling here $M=T-\Sigma$, and starting with two applications of the Cauchy-Schwarz inequality, we have four times
\begin{multline*}
\left( E_{\bsy{\theta}} \left[ \frac{\langle
\Tk-\Sigma,{\bm X}{\bm X}^\transpose \rangle}{\|{\bm X}\|^4}  \right] \right)^2
\le
\left( E_{\bsy{\theta}}\left[
\|{\bm X}\|^{-2} \|\Tk-\Sigma\|
\right]\right)^2\\
\le E_{\bsy{\theta}}[\|{\bm X}\|^{-4}] 
E_{\bsy{\theta}}[\|\Tk-\Sigma\|_{\HS}^2] \le E_{\bsy{\theta}}[\|{\bm X}\|^{-4}]  \kappa d \left(C_P -  \frac{\operatorname{Tr} (\Sigma)}{\operatorname{rank}(\Sigma)} \right)
\\
\le E_{\bsy{\theta}}[d^3\|{\bm X}\|^{-6}]^{2/3} d^{-1} \kappa \left(C_P -  \frac{\operatorname{Tr} (\Sigma)}{\operatorname{rank}(\Sigma)} \right)
\end{multline*}
by virtue of \eqref{eq:bound.for.Tsigm.inner.J}, \eqref{eq:disc.boundedby.CP-sigma^2}
and Lyapunov's inequality. By the conditions in \eqref{eq:thm:shrinkage.poinca.estimates} the squares of both these expressions
are $o(1)$, and the same is true of their square roots, which appear via Cauchy-Schwarz. Hence, as $\lambda=O(d)$ by assumption, multiplication by $\lambda$ yields $B_\lambda=o(d)$, as claimed.
\end{proof}
\begin{proof} [Proof of Corollary \ref{cor:shrinkage.via.T.univ}]
Firstly,  $\Sigma = \sigma^2 \Id$ by the conditional variance formula; additionally, as the components of ${\bm Y}_s$ are independent, the Stein kernel of ${\bm Y}_s$ exists and is given by ${\rm diag}(T_{s,1},\ldots,T_{s,d})$ for all $s \in S$, by virtue of the properties mentioned in Section \ref{subsec:mult.dim.ker}. By the discussion in Model \ref{model:indep}, $T$ is the Stein kernel of ${\bm Y}$ when  $({\bf Y},T)$ is the $\mu$ mixture of $({\bf Y}_s,T_s), s \in \mathcal{S}$.
Hence,  \eqref{eq:Blambda.sigma^2I} is in force. As for the conditions in \eqref{eq:3conditions.thm:shrinkage.non_indep},
the first one holds as it is satisfied uniformly over $s \in \mathcal{S}$ for the vectors ${\bm X}_s$ that make up the mixture.
For the second, by the conditional variance formula
$$
{\rm Var}(\operatorname{Tr}(\Tk)) = E[{\rm Var}(\operatorname{Tr}(\Tk_I))|I] + {\rm Var}(E[\operatorname{Tr}(\Tk_I)|I]),
$$
where $I$ has the mixing distribution $\mu$. For the conditional expectation, we have ${\rm Var}(\operatorname{Tr}(\Tk_s))=\sum_{i=1}^d {\rm Var}(T_{s,i}) \le \sum_{i=1}^d E[T_{s,i}^2]=O(d)$, uniformly in $s$. Similarly, for the conditional variance, $E[\operatorname{Tr}(\Tk_s)]=\sum_{i=1}^d \sigma^2=d\sigma^2=O(d)$ uniformly over $s \in S$, thus showing the second condition is satisfied. Lastly, likewise we have
$$
E[\|T_s-\sigma^2 \Id\|^2] = E\left[ \sum_{i=1}^d (T_{s,i}-\sigma^2)^2\right] \le 2\sum_{i=1}^d \left( E[T_{s,i}^2]+\sigma^4\right)=O(d),
$$
again uniformly in $s \in S$. Thus,  all conditions in \eqref{eq:3conditions.thm:shrinkage.non_indep} hold. 
\end{proof}
\begin{proof}[Proof of Corollary \ref{cor_sk_shrinkage_lc}]
First, for $\Sigma=\Id$ the conditions in \eqref{bds.theta.traceSigma} hold when
$\|{\bm \theta}\|^2=O(d)$, and we only
need to show that the two assumptions in \eqref{eq:thm:shrinkage.poinca.estimates} hold. The sharp bound on negative moments of order $6$ (or indeed of any order, as long as the dimension is large enough) was proven in \cite{Pao12, AdGuEtal}, with a universal prefactor for all isotropic log-concave distributions. For the bound on the Poincar\'e constant, \cite{Chen21} showed that for any $\epsilon > 0$ there exists a constant $C$ that only depends on $\epsilon$ such that for any log-concave isotropic measure, $C_P \leq Cd^{\epsilon}$. In particular, $C_P = o(\sqrt{d})$ uniformly for all isotropic log-concave distributions, which is the bound we need.
\end{proof}
\begin{proof}[Proof of Corollary \ref{cor_shrinkage_SLC_gene_cov}]
First note that Assumption \ref{assumptionW.kernel} is satisfied, since the density has full support and is two times differentiable. 
By 
the Brascamp-Lieb inequality, the largest eigenvalue of the covariance matrix $\Sigma$ is smaller than $1/c$. Consequently, as $c$ is fixed (independent from $d$), the first condition in \eqref{bds.theta.traceSigma} holds when
$\|{\bm \theta}\|^2=O(d)$. The second condition in \eqref{bds.theta.traceSigma} also holds since the eigenvalues of $\Sigma$ are greater than $l>0$. Hence, we only
need to show that the two assumptions in \eqref{eq:thm:shrinkage.poinca.estimates} hold. Since the bound on negative moments of order $6$ is valid for all isotropic log-concave distributions, it is easy to see that the bound is also valid in our case, by increasing the bound by a factor $1/l$ compared to the isotropic case. 
The bound on the Poincar\'e constant
is simply a consequence of the Brascamp-Lieb inequality, that gives $C_P \leq 1/c$.
\end{proof}
\begin{proof}[Proof of Theorem \ref{zerobiasshrinkage}] We show how bound \eqref{eq:mse.bound.shrink.zb} follows from \eqref{eq:MSE.expansion.3.1+3.2}. First note that the first terms of these two expressions are equal.
Letting
${\bm f}$ be as in \eqref{eq:f.jacobian.f.shrinkage},
applying zero biasing as in \eqref{eq:mzb.identity} to twice the inner product of the second term of \eqref{eq:MSE.expansion.3.1+3.2} 
gives
\begin{multline*} 
2E_{\bsy{\theta}} [\langle {\bm X}-\bsy{\theta},{\bm f}({\bm X}) \rangle]
=   2E_{\bm \theta} \left[\langle \Sigma, \nabla {\bm f}\rangle \right]+2R, \\
\mbox{where} \quad R= E_{\bsy{\theta}} \sum_{i,j=1}^d \sigma_{ij}  [ \partial_j f_i({\bf{X}})- \partial_j f_i({\bm X}^{ij}) ].
\end{multline*}
Applying the bound \eqref{eq:E[Sig.Grade]} with $M=\Sigma$ on the first term and 
combining it with the last term of \eqref{eq:MSE.expansion.3.1+3.2} leads to the second term of bound \eqref{eq:mse.bound.shrink.zb},
 and taking the absolute value of $R$ gives \eqref{eq:zb.bound.diff.partials}.  Specializing \eqref{eq:zb.bound.diff.partials} to the case $\Sigma=\sigma^2 \Id$ and 
using \eqref{eq:f.jacobian.f.shrinkage} 
yields the form of $B_\lambda^*$ in \eqref{bstar}. 
We now derive the remaining bounds on $B^*_\lambda$ for the two cases considered. Applying the triangle inequality we may bound the difference in a typical summand in \eqref{bstar} by
\begin{align}\label{eq:diff.partials.Xi.X.gen}
 \Bvert E_{\bsy{\theta}} \left[ \frac{1}{\|{\bm X}^i\|^2}-\frac{1}{\|{\bm X}\|^2} \right] \Bvert +  2 \Bvert E_{\bsy{\theta}}  \left[\frac{X_i^2}{\|{\bm X}\|^4} 
 -\frac{(X_i^i)^2}{\|{\bm X}^i\|^4}   
 \right]
 \Bvert.
\end{align}
\begin{enumerate}
  \item  
To handle the first expression in \eqref{eq:diff.partials.Xi.X.gen}, and recalling the notation in \eqref{eq:X.neg.i}, write
\begin{multline} \label{eq:three.terms}
\frac{1}{\|{\bm X}^i\|^2}
- \frac{1}{\|{\bm X}\|^2} \\= \left( \frac{1}{\|{\bm X}^{i, \neg i}\|^2}
- \frac{1}{\|{\bm X}^{\neg i}\|^2}\right)- \left( \frac{1}{\|{\bm X}\|^2}-\frac{1}{\|{\bm X}^{\neg i}\|^2} \right)+\left( \frac{1}{\|{\bm X}^i\|^2}-\frac{1}{\|{\bm X}^{i, \neg i}\|^2} \right).
\end{multline}
Using \eqref{eq:mult.sq.bias.theta} and that $h(D_{i,u}({\bm x}-{\bm \theta})+{\bm \theta})=h({\bm x})$ when $h$ does not depend on $x_i$, 
we may write the absolute expected value of the first term in \eqref{eq:three.terms} as
\begin{multline*}
\Bvert E_{\bsy{\theta}} \left( \frac{(X_i-\theta_i)^2}{\sigma^2 \|{\bm X}^{\neg i}\|^2}-\frac{1}{\|{\bm X}^{\neg i}\|^2} \right) \Bvert = \frac{1}{\sigma^2}\Bvert
 E_{\bsy{\theta}} \left(
\frac{
(X_i-\theta_i)^2-\sigma^2
}{\|{\bm X}^{\neg i}\|^2}
\right)
\Bvert \\
= \frac{1}{\sigma^2}|\operatorname{Cov}_{\bsy{\theta}}((X_i-\theta_i)^2,\|{\bm X}^{\neg i}\|^{-2})|.
\end{multline*}
For the absolute expectation of the second term of \eqref{eq:three.terms}, simplifying and applying the Cauchy-Schwarz inequality and \eqref{eq:inv.mom.Sd}b, we obtain the bound
\begin{multline} \label{eq:second.in.eq:three.terms}
\Bvert E_{\bsy{\theta}}\left( \frac{X_i^2}{\|{\bm X}^{\neg i}\|^2\|{\bm X}\|^2} \right) \Bvert \le E_{\bsy{\theta}}\left( \frac{((X_i-\theta_i)+\theta_i)^2}{\|{\bm X}^{\neg i}\|^4}\right) \\ \le 
2E_{\bsy{\theta}}\left( \frac{(X_i-\theta_i)^2+\theta_i^2}{\|{\bm X}^{\neg i}\|^4}\right)
\le \frac{2\sqrt{C_{-4}}}{d^2} \left(
\sqrt{C_4} + \theta_i^2 
\right).
\end{multline}
For the last term of \eqref{eq:three.terms}, again using \eqref{eq:mult.sq.bias.theta} as for the first term,  we similarly obtain the bound
\begin{multline} \label{eq:last.in.eq:three.terms}
E_{\bsy{\theta}}\left( \frac{(X_i^i)^2}{\|{\bm X}^{i, \neg i}\|^4}\right) = \frac{1}{\sigma^2} 
E_{\bsy{\theta}}\left( \frac{(X_i-\theta_i)^2(U(X_i-\theta_i)+\theta_i)^2}{\|{\bm X}^{\neg i}\|^4}\right) \\
\le
\frac{2}{\sigma^2} 
E_{\bsy{\theta}}\left( \frac{U^2(X_i-\theta_i)^4+(X_i-\theta_i)^2\theta_i^2}{\|{\bm X}^{\neg i}\|^4}\right) \le \frac{2\sqrt{C_{-4}}}{d^2\sigma^2}\left(
\frac{1}{3}\sqrt{C_8}+\theta_i^2 \sqrt{C_4}.
\right)
\end{multline}
Lastly, using the triangle inequality and $\|{\bm x}^{\neg i}\| \le \|{\bm x}\|$, we see that the final expression in \eqref{eq:diff.partials.Xi.X.gen} can be bounded by twice the resulting quantities in \eqref{eq:second.in.eq:three.terms} and \eqref{eq:last.in.eq:three.terms}, respectively. 
Summing over $i$ and multiplying by $\lambda \sigma^2$ in \eqref{bstar}  yields the bound \eqref{eq:Blambda*.bd}. The final claim follows in light of Remark \ref{rem:shrinkage.succeeds}.
\item  We recall that when taking an expectation that depends only on the  marginal distributions of ${\bm X}^i$ and ${\bm X}$, as is the case with the first expression below, any coupling may be used to produce an upper bound; by taking $X_{s,i}^*$ independent of all variables but $X_{s,i}$ we achieve the factorization for the first inequality. 
Beginning with the first term in \eqref{eq:diff.partials.Xi.X.gen}, 
first writing out the expectation of the mixture as in  \eqref{eq:mix.s.over.mu} and simplifying to obtain the first equality, applying  \eqref{eq:defY^i.indep} to the mixture components, which yields $({\bm X}_s^i)^{\neg i} = {\bm X}_s^{\neg i}$, then noting that $\| {\bm x}^{\neg i}\| \le \| {\bm x}\|$ and using  independence,
we obtain the bound
\begin{multline} \label{eq:part2.bd.1}
E_{\bsy{\theta}} \Bvert  \frac{1}{\|{\bm X}^i\|^2}-\frac{1}{\|{\bm X}\|^2}  \Bvert = 
\int_S E_{\bsy{\theta},s}
\Bvert \frac{ \|{\bm X}_s\|^2 -\|{\bm X}_s^i\|^2}{\|{\bm X}_s^i\|^2\|{\bm X}_s\|^2 }\Bvert d\mu \\=  \int_S E_{\bsy{\theta},s}
\left[ \frac{ |X_{s,i}^2 - (X_{s,i}^*)^2|}{\|{\bm X}_s^i\|^2\|{\bm X}_s\|^2 }\right] d\mu  \le  \int_S
E_{\bsy{\theta},s}\left[ \frac{|X_{s,i}^2 - (X_{s,i}^*)^2|}{\|{\bm X}_s^{\neg i}\|^4}\right] d \mu \\ \le 
\frac{C_{-2}}{d^2} 
\int_S E_{\bsy{\theta}, s} 
\left| X_{s,i}^2 - (X_{s,i}^*)^2\right| 
 d \mu
: = C_{i,\mu},
\end{multline}
where we have applied \eqref{eq:inv.mom.Sd}b in the final inequality. 
For the numerator that results when combining the two expressions in the second term in \eqref{eq:diff.partials.Xi.X.gen}, using $\|{\bm x}\|^4=(\|{\bm x}^{\neg i}\|^2+x_i^2)^2$ and that $({\bm X}_s^i)^{\neg i}={\bf X}_s^{\neg i}$, we obtain
\begin{multline} 
X_{s,i}^2  \|{\bm X}_s^i\|^4-(X_{s,i}^*)^2\| {\bm X}_s\|^4 \\
= X_{s,i}^2  \left( \|{\bm X}_s^{\neg i}\|^4 + 2(X_{s,i}^*)^2\|{\bm X}_s^{\neg i}\|^2 + (X_{s,i}^*)^4 \right)\\ -(X_{s,i}^*)^2 \left( \| {\bm X}_s^{\neg i}\|^4 +2X_{s,i}^2\|{\bm X}_s^{\neg i}\|^2 + X_{s,i}^4\right)\\ \label{eq:expand.quartics}
= (X_{s,i}^2-(X_{s,i}^*)^2) \|{\bm X}_s^{\neg i}\|^4+
(X_{s,i}^*)^2 X_{s,i}^2 ((X_{s,i}^*)^2-X_{s,i}^2).
\end{multline}
Arguing as in \eqref{eq:part2.bd.1},
we may bound the term arising from the first expression in \eqref{eq:expand.quartics} by
 \begin{eqnarray*}
 2\int_S  E_{\bsy{\theta},s}
     \left[  \frac{|X_{s,i}^2 - (X_{s,i}^*)^2| \| {\bm X}_s^{\neg i}\|^4 }{\| {\bm X}_s^{i}\|^4 \| {\bm X}_s\|^4}  \right]  d \mu \le 2\int_S  E_{\bsy{\theta},s}
    \left[ \frac{|X_{s,i}^2 - (X_{s,i}^*)^2|  }{\| {\bm X}_s^{\neg i}\|^4 }\right]  d \mu \le 2C_{i,\mu}.
 \end{eqnarray*}
Similarly, for the term arising from the second expression in \eqref{eq:expand.quartics}, as for any ${\bm v} \in \mathbb{R}^d$ we have $v_i^2/\|{\bm v}\|^4 \le 1/(4\|{\bm v}^{\neg i}\|^2)$, we obtain
\begin{eqnarray*}
 2\int_S  E_{\bsy{\theta},s} \left[
 \frac{(X_{s,i}^*)^2 X_i^2|X_{s,i}^2 - (X_{s,i}^*)^2| }{\| {\bm X}_s^{i}\|^4 \| {\bm X}_s\|^4} \right]  d \mu &\le&  \frac{1}{8} C_{i,\mu}.
 \end{eqnarray*}
Summing these three bounds, and then summing over $i=1,\ldots,d$ and taking the product with $\lambda \sigma^2$ as in \eqref{bstar} yields \eqref{eq:Model2.Blamstar.mix}. 
For the final claim, let a random variable $Y$ satisfy $E[Y]=0, {\rm Var}(Y)=\sigma^2$, and $E[Y^4] \le C_4$ for some constant $C_4$; let  $X=Y+\theta$. Then, applying the triangle inequality followed by \eqref{eq:def.non.zero.b}, we obtain
\begin{multline*} 
\sigma^2 E|  X^2 - (X^*)^2 | \le 
\sigma^2 E \left( (X^*)^2+X^2 \right)\\=
E[(X-\theta)X^3/3]+ \sigma^2 EX^2
=E[Y(Y+\theta)^3/3]+\sigma^2 E(Y+\theta)^2\\
= E(Y^4+3Y^3\theta+3Y^2\theta^2)/3+\sigma^4+\sigma^2 \theta^2
\leq 
C_4/3+C_4^{3/4}|\theta|+2\sigma^2 \theta^2 +\sigma^4.
\end{multline*}
Applying this bound in \eqref{eq:Model2.Blamstar.mix} and summing over $i$ yields \eqref{eq:Model2.Blamstar.refined}. That $B_\lambda^*=O(1)$ then follows by taking into account
that $\|\bsy{\theta}\|_1 \le \sqrt{d}\|\bsy{\theta}\|_2$, and the final claim in light of Remark \ref{rem:shrinkage.succeeds}.
\end{enumerate} 
\end{proof}
\begin{proof}[Example \ref{ex:spherical} continued]
We illustrate the use of the bound \eqref{eq:zb.bound.diff.partials} for a case where the observations has a non-diagonal covariance matrix. When ${\bm Y}=\sigma\sqrt{d} A{\bm U}$ for some $A \in \mathbb{R}^{d \times d}$ its distribution is supported on the surface of an ellipsoid, and has covariance matrix $\Sigma=\sigma^2AA^\transpose$.
Part \ref{prop:pos.mult.trans} of Proposition 
\ref{prop:mvariate.zbias} yields that for all $i,j$ with $\sigma_{ij}>0$ that ${\bm Y}^{ij}= \sigma\sqrt{d} A{\bm U}^*, 
$ and thus ${\bm X}^{ij}={\bm \theta}+\sigma\sqrt{d} A{\bm U}^*=:{\bm X}^*$.
Consider the case where the error vector ${\bm Y}$ exhibits only `local correlation' in that $AA^\transpose$ is the tridiagonal matrix formed by taking the sum of the identity and the matrix having $\rho \in (0,(c-1)/2)$ on the super and sub diagonals. As tridiagonal Toeplitz matrices,
%
with $a$ on diagonal, $b$ above and $d$ below the diagonal, have eigenvalues
$$
\lambda_k = a + 2\sqrt{bd}\cos\left(
\frac{k \pi}{d+1}
\right),
$$
we see that the eigenvalues of $AA^\transpose $ lie in the interval $[1-2\rho,1+2\rho]$. Hence, letting $\|A\|_{\rm op}$ denote the operator norm of $A$, applying \eqref{eq:diffUUstar}, we obtain
\begin{align*}
E_{\bm \theta}\|{\bm X}^*-{\bm X}\| =
\sigma \sqrt{d}E \|A{\bm U}^*-A{\bm U}\|
\le \sigma \sqrt{d}\|A\|_{\rm op}E \|{\bm U}^*-{\bm U}\|  \le  \sigma \sqrt{\frac{1+2\rho}{d}}.
\end{align*}
In addition, for ${\bm x}$ in the union of the supports of ${\bm X}$ and ${\bm X}^*$, when $c\sigma^2d \le \|{\bm \theta}\|^2 \le C\sigma^2 d$,
\begin{align} \label{eq:norm.x.bounds}
\sigma \sqrt{d} \left( \sqrt{c}-\sqrt{1+2\rho}\right) \le 
\|{\bm x}\| \le \sigma \sqrt{d} \left( \sqrt{C}+\sqrt{1+2\rho}\right).
\end{align}
Writing $B_\lambda^*=B_{\lambda,{\rm diag}}^*+B_{\lambda,{\rm offdiag}}^*$ as the sum of the diagonal and off diagonal contributions to $B_\lambda^*$ in \eqref{eq:zb.bound.diff.partials}, expression
\eqref{eq:bd.Blambda.star.uniform.S} is replaced by
\begin{align} \label{B.lambda.diag}
B_{\lambda,{\rm diag}}^* \le \frac{4 \sigma^2 (d-2)^2\sqrt{1+2\rho}}{(\sqrt{c}-\sqrt{1+2\rho})^3d^2}.
\end{align}
For the off diagonal terms we control
\begin{align}
B_{\lambda,{\rm offdiag}}^*=E_{\rm \theta}\left|\sum_{|j-i|=1}\sigma_{ij} [\partial_j f_i({\bm X}^{ij})-\partial_j f_i({\bm X})]\right| = \rho \lambda \sigma^2 E_{\rm \theta}\left|\sum_{|j-i|=1} [\partial_j g_i({\bm X}^*)-\partial_j g_i({\bm X})]\right|\nonumber \\
\le \rho \lambda \sigma^2 \sum_{|j-i|=1} \sum_{k=1}^d \|\partial_{kj} g_i\| E_{\rm \theta}\|{\bm X}^*-{\bm X}\|,\label{eq:elliptical.off.diag}
\end{align}
where the sup norm on the partial derivatives is taken over the union of the supports of ${\bm X}$ and ${\bm X}^*$. 
For $i \not =j$ we obtain $\partial_j g_i({\bm x})= 
\partial_j (x_i/\|{\bm x}\|^2)
=-2x_ix_j/\|{\bm x}\|^4$, and so
\begin{align*}
\partial_{k,j} g_i({\bm x}) = \left( 
 -2\left(\delta_{ik}x_j+
\delta_{jk}x_i \right) \|{\bm x}\|^4
+8x_ix_jx_k\|{\bm x}\|^2\right)/\|{\bm x}\|^8.
\end{align*}
Letting $x_0=0$ in the third expression below, we may write 
\begin{align*}
\sum_{|j-i|=1}\sum_{k=1}^d \left| \partial_{k,j} g_i({\bm x}) \right|\le   \sum_{|j-i|=1}\left( 
 2\left(|x_j|+
|x_i| \right) \|{\bm x}\|^4
+8|x_ix_j|\sum_{k=1}^d|x_k|\|{\bm x}\|^2\right)/\|{\bm x}\|^8\\
\le 4 \sum_{i=1}^d |x_i| \left( \|{\bm x}\|^{-4} + 2\|{\bm x}\|^{-6}\sum_{i=1}^d|x_ix_{i-1}|\right)\le 12  \|{\bm x}\|^{-4} \sum_{i=1}^d |x_i|\le 12 \sqrt{d}\|{\bm x}\|^{-7/2}.
\end{align*}
Hence, again taking $\lambda=\sigma^2 (d-2)$, from \eqref{eq:elliptical.off.diag}, and with $C_1$ a constant depending only on $c$, we obtain the bound
$$
B_{\lambda,{\rm offdiag}}^* \le 12 \rho \sigma^2 \lambda \sqrt{d} \left(\sigma \sqrt{d} \left( \sqrt{c}-\sqrt{1+2\rho}\right)\right)^{-7/2}  \times \sigma \sqrt{\frac{1+2\rho}{d}}
\le 12 C_1\sigma^{3/2}d^{-3/4}.
$$
Summing with \eqref{B.lambda.diag} we see that $B_\lambda^*$ is $O(1)$, as is \eqref{eq:bd.Blambda.star.uniform.S}.
As the gain from shrinkage is given by the second term in the bound in \eqref{eq:mse.bound.shrink.zb} as in \eqref{eq:shrink.gain}, and
$$
{\rm Tr}(\Sigma)-2 \kappa \ge \sigma^2(d-2) - 2(1+2\rho)\sigma^2 \ge \sigma^2(d-2) - 6\sigma^2 ,
$$
applying \eqref{eq:norm.x.bounds}, we have the lower bound
\begin{multline*}
E_{\bsy{\theta}} \left[\frac{\lambda(\lambda - 2 ({\rm Tr}(\Sigma)-2 \kappa) )}{|| {\bm X}||^2}  \right] \\ \le 
-E_{\bsy{\theta}}\left[
\frac{\sigma^4(d-2)^2-12 \sigma^2(d-2)}{\|{\bm X}\|^2} 
\right] 
\le -\frac{\sigma^4(d-2)^2(1+O(d^{-1}))}{(\sqrt{C}+\sqrt{1+2\rho})^2d}.
\end{multline*}
As the gain increases at least on the order of $d$,  shrinkage will be effective here for all sufficiently large dimensions.  Lastly, we note that in this example the observation distribution is not absolutely continuous with respect to Lebesgue measure on $\R^d$ and hence we do not provide an analogous Stein kernel result.
\end{proof}
\begin{proof}[Technical details of Example \ref{eq:student.zb}]
When ${\bm \theta}=0$, by exchangeability, 
\beas  E_{\bsy \theta} \left( \frac{(X_i^i)^2}{\|{\bm X}^i\|^4} \right) 
&=& \frac{1}{d}  E_{\bsy \theta} \left(\frac{1}{\|{\bm X}^i\|^2} \right) \qmq{and} 
E_{\bsy \theta} \left( \frac{X_i^2}{\|{\bm X}\|^4} \right)
= \frac{1}{d}  E_{\bsy \theta} \left(\frac{1}{\|{\bm X}\|^2} \right). 
\enas 
Hence, using \eqref{eq:student.coupling} for the second equality, 
that ${\bm N}$ and $\epsilon$ are independent,
and
that the first 
moment of $\epsilon$ is $2/k$, we obtain
\begin{multline*}
 E_{\bsy \theta} \left\{ \frac{\|{\bm X}^i\|^2-2(X_i^i)^2}{\|{\bm X}^i\|^4}
-\frac{\|{\bm X}\|^2-2X_i^2}{\|{\bm X}\|^4} \right\} 
= \left( 1 - \frac{2}{d} \right) 
 E_{\bsy \theta}
  \left(\frac{1}{\|{\bm X}^i\|^2} - \frac{1}{\|{\bm X}\|^2} \right)\\
= \left( \frac{d-2}{d} \right) 
 E_{\bsy \theta}
  \left(\frac{\delta }{ \sigma^2 \|{\bm N}\|^2} - \frac{\delta  + \epsilon}{\sigma^2 \|{\bm N}\|^2} \right)
=  \left( \frac{d-2}{d} \right)  \frac{1}{\sigma^2}  E_{\bsy \theta} 
  \left(\frac{- \epsilon }{\|{\bm N}\|^2} \right) \\
  = - \frac{2 (d-2)}{k \sigma^2 d} \frac{1}{2^{\frac{d-2}{2}} \Gamma (d/2)} \int_0^\infty r^{d-3} e^{-\frac{r^2}{2}}  dr 
  = - \frac{2 (d-2)}{k \sigma^2 d} \frac{2^{\frac{d-4}{2}} \Gamma ((d-2)/2)}{2^{\frac{d-2}{2}} \Gamma (d/2)}
  = - \frac{2}{k \sigma^2 d } . 
\end{multline*}
Thus for ${\bm \theta} = {\bm 0}$ 
from \eqref{bstar} we obtain
$$B_\lambda^* \le \frac{2 \lambda}{k}.$$ 
This bound is $o(d)$ when $\lambda = O(d)$ and $1/k = o(1).$
For ${\bm \theta} \ne {\bm 0},$ we apply both  equation (18) and the inequalities in (30) from \cite{Mo07}, which hold for this case, to obtain that for $d=2m$,  being even, 
$$ 2 E \left(\frac{\gamma}{\gamma || {\bm \theta} ||^2 + m\sigma^2 } \right) \le  E_{\bsy \theta}  \left( \frac{1}{\|{\bm X}\|^2} \right) 
\le  2 E\left( \frac{\gamma}{ \gamma || {\bm \theta} ||^2 + (m  -1)\sigma^2 } \right),
$$ 
and similarly, that
\beas 
 2E \left( \frac{\delta}{ \delta || {\bm {\theta}} ||^2  + m\sigma^2 } \right) \le 
E_{\bsy \theta} \left( \frac{1}{\|{\bm X}^i\|^2} \right)
&\le & 2E\left( \frac{\delta}{ \delta|| {\bm {\theta}} ||^2  + (m -1)\sigma^2 } \right) .
\enas 
As $\gamma = \delta + \epsilon > \delta$, it follows that 
\begin{multline*}
\left| E_{\bsy \theta} \left( \frac{1}{\|{\bm X}^i\|^2} \right) -  E_{\bsy \theta} \left( \frac{1}{\|{\bm X}\|^2} \right)  \right|
\le 2
 E \left\{ 
\frac{\delta + \epsilon}{ (\delta + \epsilon) || {\bm {\theta}} ||^2 + (m -1)\sigma^2 }
 - 
 \frac{\delta}{ \delta || {\bm {\theta}} ||^2  + m\sigma^2 } 
 \right\} 
 \\
 =  2 
 E \left\{ \frac{ \sigma^2(\epsilon m + \delta)}{((\delta+\epsilon) || {\bm {\theta} }||^2 + (m-1) \sigma^2)( \delta  || {\bm {\theta} }||^2 + m \sigma^2) } \right\}
 \le  \frac{2}{m(m-1) \sigma^{{2}}} 
 E (\epsilon m + \delta) \\
 = \frac{2}{m(m-1) \sigma^2} 
\left( \frac{2m}{k} + \frac{k-2}{k}\right) 
= \frac{8 (d+k-2)}{d(d-{2}) k \sigma^2}.
\end{multline*} 
Thus, from \eqref{bstar},
$$B_\lambda^* \le \frac{8 \lambda (d+k-2)}{(d-2) k}$$
and if $\lambda = O(d)$ and $1/k = o(1)$, this bound is $o(d)$ as desired. 
\end{proof}
}
\end{supplement}

\begin{supplement}
\stitle{Supplement C: Proofs for Section \ref{sec_SURE}}
\sdescription{
\begin{proof}[Proof of Proposition \ref{prop:AsureBias.z}] Inequality \eqref{eq:zb.bias.partialhi} follows directly by taking the difference between
\eqref{eq:sure} and \eqref{eq:zero.replacement}. Coupling ${\bm X}^{ij}$ to ${\bm X}$ so that it achieves $d({\bm X},{\bm X}^{ij})$, we obtain inequality \eqref{eq:zb.bias.partialhi.lip.gen} 
by noting that for every $i,j=1,\ldots,d$, 
\begin{align*}
E_{\bm \theta}|\partial_j f_i({\bm X}^{ij}) - \partial_j f_i({\bm X}) | \le \|\partial_j f_i \|_{\rm Lip}E_{\bm \theta}\|{\bm X}^{ij}-{\bm X}\| = \|\partial_j f_i \|_{\rm Lip} d({\bm X},{\bm X}^{ij}).
\end{align*}
To prove inequality \eqref{eq:zb.bias.partialhi.lip}, as the expectation in \eqref{eq:zb.bias.partialhi} does not depend on the joint distribution of $({\bm X},{\bm X}^i)$,
under Model \ref{model:indep}
we may choose ${\bm X}^i_s$ to equal ${\bm X}_s$ in coordinates $j \not = i$ and take its $i^{th}$ coordinate $X_{s,i}^i=X_{s,i}^*=Y_{s,i}^*+\theta_i$ to have the $X_{s,i}$-zero bias distribution, independent of $X_{s,j}, j \not = i$, and to achieve the minimal $L^1$ distance to $X_{s,i}$. So doing, we obtain
\begin{align*}
E_{\bm \theta}|\partial_i f_i({\bm X}^i) - \partial_i f_i({\bm X}) | \le \|\partial_i f_i \|_{{\rm Lip},i}E_{\bm \theta}|X_i^i-X_i| = \|\partial_i f_i \|_{{\rm Lip},i}\int_{\mathcal{S}}E_{\bm \theta,s}|X_{s,i}^i-X_{s,i}|d \mu.
\end{align*}
The proof is completed by noting that the Wasserstein distance is preserved by translation,  and applying the triangle inequality in \eqref{eq:zb.bias.partialhi}. 
\end{proof}
\begin{proof}[Proof of Theorem \ref{thm:soft.thres.log.concave}]
Note first that, according to formula (\ref{formula_SURE_ST}) and compactness of $I$, the minimum values of SURE in display (\ref{lambda_hat}) indeed exist. The minimum values along $I$ of the expectation of SURE exist for the same reasons, and by continuity of the risk with respect to $\lambda$, the minimum values of the risk also exist. Let $R(\lambda)=E_{\bm \theta}\left[\left\Vert S_{\lambda}\left({\bm X}\right)-{\bm \theta}\right\Vert ^{2}\right]$, $M(\lambda)=E_{\bm \theta}[{\rm SURE}\left({\bm f}_{\lambda},{\bm X}\right)]$ and ${\rm Bias}(\lambda)=R(\lambda)-M(\lambda)$. Then, for all $\lambda>0$ it holds that
\beas
R(\lambda)\leq M(\lambda)+|{\rm Bias}(\lambda)| \leq M(\lambda)+2LM R(\lambda),
\enas
where (\ref{upper_bias}) was used in the second inequality. Substituting $\widehat{\lambda}$ for $\lambda$ and rearranging, we obtain
\beas
(1-2LM) R(\widehat{\lambda})&\leq & M(\widehat{\lambda})\\
&=&E_{\bm \theta}[{\rm SURE}\left({\bm f}_{\widehat{\lambda}},{\bm X}\right)]\ignore{+ E[\sup_{\lambda \in I} |{\rm SURE}\left({\bm f}_{\lambda},{\bm X}\right)- M(\lambda)|]}\\
& = & E_{\bm \theta}[\min_{\lambda \in I}{\rm SURE}\left({\bm f}_{\lambda},{\bm X}\right)]\ignore{+ E_{\bm \theta}[\sup_{\lambda \in I} |{\rm SURE}\left({\bm f}_{\lambda},{\bm X}\right)- M(\lambda)|]}\\
& \leq & \min_{\lambda \in I} M(\lambda)+E[\sup_{\lambda \in I} |{\rm SURE}\left({\bm f}_{\lambda},{\bm X}\right)- M(\lambda)|]\\
& \leq & \min_{\lambda \in I} \left\{R(\lambda)+|{\rm Bias}(\lambda)|\right\}+E_{\bm \theta}[\sup_{\lambda \in I} |{\rm SURE}\left({\bm f}_{\lambda},{\bm X}\right)- M(\lambda)|]\\
& \leq & (1+2LM) \min_{\lambda \in I} R(\lambda)+E_{\bm \theta}[\sup_{\lambda \in I} |{\rm SURE}\left({\bm f}_{\lambda},{\bm X}\right)- M(\lambda)|],
\enas 
where (\ref{upper_bias}) was used in the final inequality. 
To conclude the proof, it remains only to control the second term in the right-hand side of the last expression. This can be done exactly as in the proof of Proposition 1 in \cite{DoJo95}, using
only independence, Hoeffding's inequality, the fact that the random variables involved in  ${\rm SURE}\left({\bm f}_{\lambda},{\bm X}\right)$ are uniformly bounded, and a discretization of the index set $I$.
\end{proof}
\begin{proof}[Proof of Theorem \ref{prop:adpt_dep}.] 
The first two terms in the follow directly from \eqref{eq:Blambda.sigma^2I},
the  first term 
corresponding, and the second obtained by substituting the given value of $\lambda$ and applying Jensen's inequality.
The claim about the order of $B_\lambda$ follows from Theorems \ref{thm:shrinkage.non_indep} and \ref{thm:shrinkage.poinca}
by noting that $o(d)$ conclusions there for  $B_\lambda$ in  \eqref{eq:upper.bound.B.lambda} become $o(1)$ here due to the scaling, eg. $\sqrt{{\rm Var}(T/d)}=\sqrt{{\rm Var}(T)}/d$.
\end{proof} 
\begin{proof}[Proof of Corollary \ref{cor:mse.scaled.le.1}.] As the conditions  \eqref{eq:thm:shrinkage.poinca.estimates} of Theorem \ref{thm:shrinkage.poinca} apply for a general isotropic log-concave vector - see Corollary \ref{cor_sk_shrinkage_lc} - the result follows from Theorem \ref{prop:adpt_dep}. Indeed, if the distribution of $\bm X$ belongs  $\mathcal{P}(c)$, then the limit of the right-hand side of (\ref{bd:thm:prop:adpt_dep}) corresponds to the right-hand side of (\ref{eq:Pinsker.nonGaussian}). 
\end{proof}
\begin{proof}[Proof of Theorem \ref{thm_adpt_indep}]
The first term of the bound \eqref{eq:MSE.bound.adapt} follows from \eqref{eq:mse.bound.shrink.zb} and \eqref{eq:Jen.Enorm.sq}, in the same way as the first two terms in the bound \eqref{bd:thm:prop:adpt_dep} appear. 
The last term corresponds to the bound \eqref{eq:Model2.Blamstar.refined} on $B_\lambda^*$. For the final claim, 
it suffices to note that if the distribution of $\bm X$ belongs to $\mathcal{P}(c)$, then the limit of the right-hand side of \eqref{eq:MSE.bound.adapt} equals
the right-hand side of \eqref{eq:Pinsker.nonGaussian} when the dimension $d$ tends to infinity.
\end{proof}
}
\end{supplement}

\begin{supplement}
\stitle{Supplement D: Proofs for Section \ref{sec:prop:mvariate.zbias}}
\sdescription{
\begin{proof}[Proof of Proposition \ref{prop:mvariate.zbias}]
Starting with Part \ref{prop:mvariate.zbias.ex.un}, assume that \eqref{eq:mgalelike} holds. Then as for all $i \not =j$ we have 
$$
E[Y_iY_j]=E[E[Y_iY_j|Y_k, k \not = i]]=E[Y_jE[Y_i|Y_k, k \not = i]]=0,
$$
the covariance matrix $\Sigma$ must be diagonal. 
Next we prove existence of the zero bias vectors by showing that the construction \eqref{eq:X^i.by.D} produces vectors ${\bm Y}^i$ satisfying \eqref{eq:mzb.identity.theta.zero}. For any given ${\bm f} \in W_z^{1,2}$ and  $i=1,\ldots,d$, first changing measure from ${\bm Y}^{\square,i}$ to ${\bm Y}$ according to 
\eqref{eq:square.bias.measures},  then integrating over the uniform variable $U_i$ and using that $D_{i,1}{\bm y}={\bm y}$, we obtain
\begin{align*}
E[\sigma_i^2 \partial_i f_i ({\bm Y}^i)]=E[Y_i^2\partial_i f_i(D_{i,U_i}({\bm Y}))] 
=E\left[ Y_i(f_i({\bm Y})-f_i(D_{i,0}{\bm Y}))
\right]=E[Y_i f_i({\bm Y})],
\end{align*}
where we have
applied \eqref{eq:mgalelike} in the final equality, which is in force as $D_{i,0}{\bm y}$
depends only on $\{y_j, j \not =i\}$. Summing over $i$ now yields \eqref{eq:mzb.identity.theta.zero}.
Next, to show uniqueness, for 
$i \in \{1,\ldots,d\}$ and $g: \mathbb{R}^d \rightarrow \mathbb{R}$, $g \in \mathcal{G}$, the class of all continuously differentiable functions with compact support, let 
\begin{align}\label{eq:f.int.g.in.u}
f({\bm y}) = \int_{-\infty}^{y_i}g(y_1,\ldots,y_{i-1},u, y_{i+1}, \ldots, y_d)du.
\end{align}
When \eqref{eq:mzb.identity.theta.zero} holds, substitution using the function ${\bm f}:\mathbb{R}^d \rightarrow \mathbb{R}^d$ given by ${\bm f}({\bm y})=f({\bm y}){\bm e}_i$ yields
\begin{align} \label{char.integral.g}
E[g({\bm Y}^i)] = \frac{1}{\sigma_i^2}E[Y_if({\bm Y})] \qm{for all $g \in \mathcal{G}$,}
\end{align}
thus showing the law of ${\bm Y}^i$ is uniquely determined.
To show the converse, for ${\bm f}({\bm y})=g(y_j, j \not = i){\bm e}_i$ for some $i =1,\ldots,d$, an infinitely differentiable, compactly supported real valued function $g$ of $d-1$ variables and ${\bm e}_i$ the $i^{th}$ unit basis element, we see that \eqref{eq:mzb.identity.theta.zero} implies that $E[Y_ig(Y_j, j \not = i)]=0$, and hence also \eqref{eq:mgalelike}, 
 concluding the proof of \ref{prop:mvariate.zbias.ex.un}.
For claim \ref{prop:mvariate.zbias.support}, 
without loss of generality we may take $i=1$, and for integers $a<b$ we use $a:b$ to denote the set of integers $a,\ldots,b$. In showing the claim,
for a given $r>0$, left implicit in the notation, we let $g_1: \mathbb{R} \rightarrow \mathbb{R}$ and $g_{2:d}:\mathbb{R}^{d-1} \rightarrow \mathbb{R}$ denote given smooth functions that are strictly positive in open balls of radius $r$ centered at the origin and take the value zero otherwise. 
For $a \in [-\infty, \infty)$, let
\begin{align} \label{eq:support.deff}
	f(x)=\int_a^x g_1(u)du, 
\end{align}
which is necessarily a non-decreasing function. 
For all ${\bm v} \in \mathbb{R}^d$, by \eqref{eq:mzb.identity.theta.zero},
\begin{align} \label{eq:support.suffices}
	E[Y_1f(Y_1-v_1)g_{2:d}({\bm Y}_{2:d}-{\bm v}_{2:d})] = \sigma_1^2 E[g(Y_1^1-v_1)g_{2:d}({\bm Y}_{2:d}^1-{\bm v}_{2:d})].
\end{align}
 Letting 
\begin{align*} 
	W=g_{2:d}({\bm Y}_{2:d}-{\bm v}_{2:d}), \qmq{we have $W \ge 0$ and} 
	\gamma = \inf_{\|{\bm y}_{2:d}\|_2 \le r/2}g_{2:d}({\bm y}_{2:d})>0,
\end{align*}
the strict positivity guaranteed by continuity and compactness.
To show the inclusion
$$
S^1 \supset {\rm cl}(U^1(S))
$$
it suffices that $S^1 \supset U^1(S)$, as supports are closed. 
Taking an arbitrary ${\bm v} \in U^1(S)$ we verify its membership in 
$S^1$ by exhibiting $\tau>0$ such that the quantity in \eqref{eq:support.suffices} is positive for all $r \in (0,\tau)$.
The point ${\bm v}$ must have a representation as $(us_1,{\bm v}_{2:d})$ for some ${\bm s} = (s_1,{\bm v}_{2:d}) \in S$ with $s_1 \not =0$ and some $u \in [0,1]$. We first consider $u \in (0,1]$.
Suppose that $s_1>0$. Take $a=-\infty$ in \eqref{eq:support.deff}, and set $\tau=us_1$, which is necessarily positive. For $r \in (0,\tau)$ and $y \le -r$ we have $g_1(y)=0$, and therefore also $f(y)=0$, so $f(y-\tau)=0$ for $y -\tau \le -r$, and as $g_1(y)>0$ for $y \in (-r,r)$ we have $f(y-\tau)$ is positive for all $y-\tau>-r$.
Using that $W \ge 0$,
\begin{multline*}
E[Y_1f(Y_1-\tau)W] =  E[{\bf 1}(Y_1 -\tau > -r)Y_1f(Y_1-\tau)W]\\
\ge (\tau-r) E[{\bf 1}(Y_1 -\tau > -r)f(Y_1-\tau)W]
\ge (\tau-r) E[{\bf 1}(Y_1 -\tau \ge -r/2)f(Y_1-\tau)W].
\end{multline*}
As $f$ is non-decreasing
$$
{\bf 1}(y -\tau \ge -r/2)f(y-\tau) \ge f(-r/2){\bf 1}(y \ge  \tau - r/2)\ge f(-r/2){\bf 1}(y \ge  s_1 - r/2),
$$
and therefore, now using ${\bm s} \in S$ for the final inequality, we have ${\bm v} \in S^1$ via \eqref{eq:support.suffices} and
$$
E[Y_1f(Y_1-v_1 )W] \ge (\tau-r)f(-r/2)\gamma P(Y_1 \ge s_1-r/2,\|{\bm Y}_{2:d}-{\bm v}_{2:d}\| \le r/2)>0.
$$
For $s_1<0$ we argue similarly, concluding the case $u\in (0,1]$. The case $u=0$ follows, as $(0,{\bm v}_{2:d}) \in {\rm cl}(\{(us_1,{\bm v}_{2:d}): u \in (0,1]\}) \subset {\rm cl}(S^1)=S^1$.
\ignore{In the case $u=0$, so ${\bm v}=(0,{\bm s}_{2:d})$, we take $\tau=s_1$. Take $a=0$ in \eqref{eq:support.deff}. We have $yf(y) > 0$ for all $y \not = 0$ and $f(y) \ge c\, {\rm sign}(y)$ for $|y| \ge r$ where $c={\rm min}\{\int_0^r g_1(u)du, \int_{-r}^0 g_1(u)du\}>0$. Hence, 
$$
yf(y) \ge c|y|{\bf 1}(|y| \ge r) \ge cr {\bf 1}(|y| \ge r),
$$
and, since ${\bm s} \in S$, for all $r \in (0,\tau)$, 
$$
E\left[ 
Y_1f(Y_1)W
\right] \ge cr\gamma P( 
Y_1 \ge r, \|{\bf Y}_{2:d}-{\bf s}_{2:d}\|_2 \le r/2
)>0.
$$}
For the opposite inclusion $S^1 \subset {\rm cl}(U^1(S))$, take an arbitrary ${\bm v}=(v_1,{\bm v}_{2:d}) \in S^1$. First consider $v_1 \not = 0$, and take $v_1>0$, the case $v_1<0$ being similar. Take $r>0$ arbitrary and let
$f$ be as in \eqref{eq:support.deff} with $a=-\infty$. As ${\bm v} \in S^1$
the right hand side of  \eqref{eq:support.suffices} is positive,
and as $f(y-v_1)=0$ for all $y - v_1 \le -r$, we see via the left hand side that ${\bm Y}$  has mass in $[v_1-r,\infty) \times B_r({\bm v}_{2:d})$ for all $r>0$. In particular, there exists $s_1 \ge v_1-r$ and ${\bm t}_{2:d}$ within distance $r$ of ${\bm v}_{2:d}$ such that $(s_1,{\bm t}_{2:d}) \in S$, implying that $(v_1-r,{\bm t}_{2:d}) \in U^1(S)$. As $r>0$ is arbitrary, ${\bm v} \in {\rm cl}(U^1(S))$.
For the case $v_1=0$ take $r>0$ arbitrary, and $a=0$ in \eqref{eq:support.deff}, so that $yf(y) \ge 0$ for all $y$ and is strictly positive except for $y=0$. By dominated convergence, using the right hand side of identity in \eqref{eq:support.suffices} to obtain the inequality, we have
\begin{align*}
\lim_{\rho \downarrow 0} E[Y_1f(Y_1){\bf 1}(|Y_1| \ge \rho)W]=E[Y_1f(Y_1)W]>0.
\end{align*}
In particular, for some $\rho>0$, we have that
$$
E[Y_1f(Y_1){\bf 1}(|Y_1| \ge \rho)W]>0.
$$
Thus there exists $s_1 \not = 0$ and ${\bm t}_{2:d}$ within distance $r$ of ${\bm v}_{2:d}$ such that $(s_1,{\bm t}_{2:d}) \in S$, and hence, that $(0,{\bm t}_{2:d}) \in U^1(S)$. As $r>0$ is arbitrary, we obtain that $(0,{\bm v }_{2:d}) \in {\rm cl}(U^1(S))$.
The claims in \ref{prop:mvariate.zbias.sum} and  \ref{prop:mvariate.zbias.mix} regarding the zero bias distribution of sums and mixtures follow by making minor modification to the proofs of those claims in the univariate case, see Lemma 2.1(v) of \cite{GR97} and Theorem 2.1, \cite{Go10}, respectively.
To show Part \ref{prop:pos.mult.trans}, note that from \eqref{eq:mzb.identity.gen.cov}, for a given smooth real valued function
$f$ on $\mathbb{R}^d$, letting $g({\bm u})=f(A{\bm u})$ 
and ${\bm Z}^{kl}=A{\bm U}^{kl}$ for $1 \le k,l \le m$ such that $\gamma_{kl} \not =0$, we have 
\begin{multline} \label{eq:Xk.genf}
E[U_kf({\bm Y})]= E[U_kg({\bm U})]= E[\langle {\bm U},g({\bm U}){\bm e}_k\rangle]\\=E\left[\sum_{l=1}^m \gamma_{kl} \partial_l g({\bm U}^{kl})\right]
= \sum_{l=1}^m \sum_{j=1}^d 
\gamma_{kl}a_{jl} E[\partial_j f({\bm Z}^{kl})].
\end{multline}
Letting ${\bm a}_i^\transpose$ be the $i^{th}$ row of $A$, we obtain the components of $\Sigma$ as
\begin{align*}
	\sigma_{ij} ={\bm a}_i^\transpose \Gamma {\bm a}_j= \sum_{1 \le k,l \le m} a_{ik}\gamma_{kl}a_{jl}.
\end{align*}
Then, as $Y_i={\bm a}_i^\transpose{\bm U}$,
for ${\bm f} \in C_c^\infty(\R^d)$, applying \eqref{eq:Xk.genf}, we obtain
\begin{multline*} 
E [\langle {\bm Y},{\bm f}({\bm Y})\rangle] =	\sum_{i=1}^d E[Y_if_i({\bm Y})]=\sum_{i=1}^d \sum_{k=1}^m a_{ik}
	E[U_kf_i({\bm Y})]\\=
	\sum_{1 \le i,j \le d, 1 \le k,l \le m} a_{ik}
	\gamma_{kl}a_{jl} E[\partial_j f_i({\bm Z}^{kl})]
	\\=\sum_{1 \le i,j \le d: \sigma_{ij}>0} \sigma_{ij} \sum_{1 \le k,l \le m}
	\frac{a_{ik}
	\gamma_{kl}a_{jl}}{\sigma_{ij}}
	E[\partial_j f_i({\bm Z}^{kl})] 
	 = \sum_{1 \le i,j \le d}\sigma_{ij}E[\partial_j f_i ({\bm Y}^{ij})].
\end{multline*}
Moving to \ref{prop:mvariate.zbias.density}, assume now that ${\bm Y}$ has density $p({\bm y})$. Relabeling $i$ by $1$ for convenience, we have
\begin{align*} 
\int_{\mathbb{R}^{d-1}} \int_{-\infty}^\infty |u| p(u,{\bm y}_{2:d})du d{\bm y}_{2:d} = E[|Y_1|] < \infty,
\end{align*}
so that $p^1({\bm y})$ as given in \eqref{def:p^i} exists almost everywhere by Fubini's theorem.
Now, from \eqref{eq:f.int.g.in.u} and \eqref{char.integral.g}, for $g \in \mathcal{G}$ we may write 
\begin{align*}
\sigma_1^2 E[g({\bm Y}^1)] = E[Y_1f({\bm Y})]= \int_{-\infty}^\infty y_1 \int_{-\infty}^{y_1}g(u, {\bm y}_{2:d})du \,  p({\bm y})d{\bm y}.
\end{align*}
We decompose the outer integral as the sum of integrals over the positive and negative half lines, the one over positive values being
\begin{multline*} 
 \int_0^\infty y_1 \int_{-\infty}^{y_1}g(u,{\bm y}_{2:d})du\,  p({\bm y}) d{\bm y} \\
 = \int_0^\infty y_1 f(0,{\bm y}_{2:d})p({\bm y})d{\bm y}
 + \int_0^\infty  \int_0^{y_1}y_1g(u, {\bm y}_{2:d})du\,  p({\bm y}) d{\bm y}.
\end{multline*}
Recalling $g$ is bounded and then invoking Fubini's theorem to change the order of integration, the double integral may be written as
\begin{multline*}
    \int_0^\infty  \int_u^\infty y_1p(y_1,{\bm y}_{2:d})dy_1 g(u,{\bm y}_{2:d}) du \,  d{\bm y}_{2:d} \\
    = \int_0^\infty  \left[ \int_{y_1}^\infty up(u,{\bm y}_{2:d})du \right] g({\bm y}) d{\bm y} = \sigma_1^2 \int_0^\infty p^1({\bm y})g({\bm y})d{\bm y}
\end{multline*}
where for the first equality we have interchanged the labelling of $u$ and $y_1$, and then applied definition \eqref{def:p^i} for the final equality. 
Similarly, the integral over the negative half line yields the sum
\begin{align*}
\int_{-\infty}^0 y_1 f(0,{\bm y}_{2:d})p({\bm y})d{\bm y} + \sigma_1^2 \int_{-\infty}^0  p^1({\bm y}) g({\bm y}) d{\bm y},
\end{align*}
and combining with the integral over the positive half line and applying \eqref{eq:mgalelike} to see that the first term vanishes yields
\begin{align*}
\sigma_1^2 E[g({\bm Y}^1)] = \sigma_1^2 \int_{-\infty}^\infty p^1({\bm y})g({\bm y})d{\bm y} \qm{for all $g \in \mathcal{G}$,}
\end{align*}
thus showing that $p^1({\bm y})$ is the density of ${\bm Y}^1$. For the final claim, as $|u|p(u,\bm{y}_{2:d}) \le g(u)$ for some $g \in L^1$, the density $p^1$ is uniformly bounded as
$$
\sigma_1^2 p^1({\bm y}) \le 
\int_{y_1}^\infty |u|p(u,\bm{y}_{2:d})du
\le \int_{-\infty}^\infty g(u)du.
$$
\end{proof}
}
\end{supplement}

\begin{supplement}
\stitle{Supplement E: Proofs for Section \ref{ssec_inv_norms}}
\sdescription{
\begin{proof}[Proof of Lemma \ref{lem:inverse.mean}] For
$t \in (-\infty,0]$ let
\bea \label{eq:def.Pn.seq}
P_{m,d}(t) = \int_{-\infty}^t P_{m-1,d}(u) du \qmq{for $m \ge 1$, where} P_{0,d}(t)=\frac{C}{(1-\mu t/q)^{qd}}.
\ena
Then, with $(x)_k$ the falling factorial, one may verify via induction that for $m \ge 0$ and $d \ge 2m/q$ the quantities in \eqref{eq:def.Pn.seq} exist, and are given by 
\bea 
P_{m,d}(t) = \frac{C(1-\mu t/q)^{-qd+m}}{(\mu/q)^m (qd-1)_m}
\qmq{for all $t \in (-\infty,0]$.}
\ena
By \eqref{eq:MSd.bound}, and induction again, for such $m$ and $d$ the quantities defined by
\beas 
M_{m,d}(t) = \int_{-\infty}^t M_{m-1,d}(u) du \qmq{for $m \ge 1$ with} M_{0,d}(t)=E[e^{tS_d}],
\enas
exist and satisfy 
\bea \label{eq:Mmd.bound}
M_{m,d}(t) \le P_{m,d}(t) 
\qmq{for all $t \in (-\infty,0]$.}
\ena
Via Fubini's theorem, and induction yet again, we have 
\beas 
M_{m,d}(t) = E\left[\frac{e^{tS_d}}{S_d^m}\right] \qmq{so in particular} M_{m,d}(0) = E\left[\frac{1}{S_d^m}\right],
\enas
and by \eqref{eq:def.Pn.seq} and \eqref{eq:Mmd.bound} these quantities are finite and satisfy
\begin{align*}
E\left[\frac{1}{S_d^m}\right] =  M_{m,d}(0) \le P_{m,d}(0)= \frac{C}{(\mu/q)^m (qd-1)_m}.
\end{align*}
Using $d \ge 2m/q$ in the final inequality, we obtain 
\begin{align*}
(qd-1)_m \ge (qd-m)^m = d^m\left(q-\frac{m}{d}\right)^m \ge \left( \frac{dq}{2} \right)^m
\end{align*}
demonstrating our desired conclusion holds with $C_{\mu,m}=C(2/\mu)^m$. 
When $S_d$ is a sum of non-negative, negatively associated random variables, using that the function $\exp(tv)$ is non-increasing when $t \le 0$ for all $v \ge 0$, its moment generating is bounded by the product of the moment generating functions of the marginals of the summands, and hence satisfies \eqref{eq:MSd.bound} with the constant $C=1$, independent of the dimension $d$. \end{proof}
}
\end{supplement}

\begin{supplement}
\stitle{Supplement F: Remarks on Assumptions \ref{assumptionW.kernel}}
\sdescription{Assumption \ref{assumptionW.kernel} was introduced in Section \ref{sec:shrinkage} in connection with the Poincar\'e inequality. 
This assumption does not always hold even for measures with a density, for example, when the support of the measure has a boundary, and its density does not behave well enough at the boundary. However, the following result shows that the assumption holds under an easily verified condition. The lower bound of 5 in Lemma \ref{lem:ker.sat.assumption}
on the dimension $d$ is to ensure that $\|\nabla ({\bm x}/\|{\bm x}\|^2)\|^2 = d/\|{\bm x}\|^4 \in L^2(\nu)$,  and the assumption of full support is used to guarantee that the smooth approximations constructed in the proof have support strictly inside the support of the measure.
Below, for $r>0$ and ${\bm y} \in \R^d$ we let
\begin{align*} 
 B(r,{\bm y})=\{{\bm x} \in \mathbb{R}^d: \|{\bm x}-{\bm y}\|_\infty \le r\} \qmq{and} B(r)=B(r,{\bm 0}).
\end{align*}
\begin{lemma} \label{lem:assumptionW.kernel2}
Assumption \ref{assumptionW.kernel} holds when $d \ge 5$ and $\nu$ has a density of full support that is bounded in a neighborhood of the origin. 
\end{lemma}
\begin{proof} With ${\bm g}_0({\bm x})={\bm x}/\|{\bm x}\|^2$
and $\epsilon \in (0,1/2]$, consider the smooth compactly supported function ${\bm g}_{\epsilon}({\bm x})= {\bm x}\psi_{\epsilon}({\bm x})/(\epsilon^2 + \|{\bm x}\|^2)$ 
where 
$\psi_{\epsilon}({\bm x})$ is the approximation of unity that takes the value $\prod_i (e^{1-(1-\epsilon^2|x_i|^2)^{-1}})$ for  $\max_{i=1,\ldots,d} |x_i| < \epsilon^{-1}$, 
and zero otherwise. 
We decompose the $L^2$ norm of the difference as
\begin{multline*}
||{\bm g}_0 - {\bm g}_{\epsilon}||_{L^2(\nu)}^2 \\
= \int_{B(\epsilon^{1/2})}\|{\bm g}_0 - {\bm g}_{\epsilon}\|^2 d\nu + \int_{B(\epsilon^{-1/2}) \setminus B(\epsilon^{1/2})}\|{\bm g}_0 - {\bm g}_{\epsilon} \|^2 d\nu + \int_{B(\epsilon^{-1/2})^c} \|{\bm g}_0 - {\bm g}_{\epsilon}\|^2 d\nu
\end{multline*}
where $B(r)^c = \mathbb{R}^d/B(r)$. Since the density of $\nu$ is bounded at the origin the singularity is integrable, so the first term goes to zero as $\epsilon$ goes to zero by dominated convergence. For the last term, since both functions are bounded, the integral indeed goes to zero as $\epsilon$ goes to zero. So all that is left is the second term. On $B(\epsilon^{-1/2}) \setminus B( \epsilon^{1/2})$  we have
\begin{align*}
    \Bvert  \frac{1}{\| {\bm x}\|^2} -  \frac{1}{\epsilon^2+\|{\bm x}\|^2}\Bvert \le \frac{\epsilon^2}{\|{\bm x}\|^4} \le  \frac{\epsilon}{\|{\bm x}\|^2} \qmq{and} |1-\psi_{\epsilon}({\bm x})| \leq 1 - e^{-\frac{\epsilon d}{1-\epsilon}} \leq \frac{\epsilon d}{1-\epsilon},
\end{align*}
where in the second equality we maximize the difference by taking all $|x_i|$ large, hence
\begin{multline*}
 \| {\bm g}_0({\bm x})-{\bm g}_\epsilon({\bm x}) \| = \norm{ {\bm x}
\left[
\left(\frac{1}{\|{\bm x}\|^2} - \frac{1}{\epsilon^2+\|{\bm x}\|^2} \right) + \frac{1}{\epsilon^2+\|{\bm x}\|^2}\left(
1-\psi_\epsilon(\bm{x})
\right)
\right]}\\
\le \|{\bm x}\| \left(\frac{\epsilon}{\|{\bm x}\|^2} +  \frac{2d \epsilon}{\epsilon^2+\|{\bm x}\|^2} \right) \le \epsilon\left( 1+ 2d \right) \|{\bm g}_0({\bm x})\|,
\end{multline*}
demonstrating that the second term 
tends to zero with $\epsilon$.
The same reasoning holds for the approximation of the derivative, using the fact that $\|\nabla \psi_{\epsilon}\|$ is small on $B(\epsilon^{-1/2})$ and $\|\nabla {\bm g}_0 - \nabla \left[{\bm x}/(\epsilon^2 + \|{\bm x}\|^2)\right]\|$ is small on $B( \epsilon^{1/2})^c$. 
\end{proof}
Assumption \ref{assumptionW.kernel} ensures that the function ${\bm g}_0({\bm x})={\bm x}/\Vert {\bm x} \Vert^2$, specifically related to shrinkage, is contained in the closure of $C^{\infty}_c(\mathbb{R}^d)$ with respect to the Sobolev norm
\eqref{def:W.norm.ker} for the given $\nu$. The application of the techniques developed here when used in other situations may necessitate the use of Stein's identity with functions different from ${\bm g}_0({\bm x})$, and successfully carrying out such a program will require that the functions in question belong to these same spaces. 
In this section, we provide conditions which guarantee that these needed inclusions will hold.
Consider a probability measure $\nu$ with positive density $p$
on $\mathbb{R}^{d}$.  Given real valued functions on $\mathbb{R}^d$
let $*$ denote the usual convolution 
with respect to Lebesgue measure ${\rm Leb}$, that is,
\[
(f*g)({\bm x})=\int_{\mathbb{R}^d} f({\bm x}-{\bm y})g({\bm y})d{\bm y}.
\]
When $f \in L^2(\nu)$ and $g \in L^1({\rm Leb})$, letting $*_{p}$ denote the operation
\[
(f*_{p}g)({\bm x})=\frac{1}{\sqrt{p({\bm x})}}\left\{ (f\sqrt{p})*g\right\} ({\bm x})=\frac{1}{\sqrt{p({\bm x})}}\int_{\mathbb{R}^d} f({\bm x}-{\bm y})\sqrt{p({\bm x}-{\bm y})}g({\bm y})d{\bm y},
\]
by standard results on convolution, $f*_{p}g$ is finite almost everywhere and belongs to $L^2(\nu)$ (see for instance \cite[Section 2.15]{LiebLoss01}). 
Recall that the family of functions $(\psi_{h})_{h>0}$ is said to be a mollifier for the convolution $*$ as $h\rightarrow 0$ if  $\psi_{h}$ is non-negative on $\mathbb{R}^d$ for all $h>0$, $\int_{\mathbb{R}^d} \psi_h(x)dx=1$ and
\[
\forall\eta>0,\;\lim_{h\rightarrow0}\int_{\|{\bm x}\|\geq\eta}\psi_{h}({\bm x})d{\bm x}=0.
\]
We have the following result, motivating the introduction of the operation
$*_{p}$:
\begin{proposition}\label{prop:moll.conv.in.L1}
If $f\in L^{2}(\nu)$ and $(\psi_{h})_{h>0}$
is a mollifier for the convolution product $*$ as $h\rightarrow0$
then $f*_{p}\psi_{h}$ converges to $f$ in $L^{2}(\nu)$.
\end{proposition}
\begin{proof}We have that
\[
\left\Vert f-f*_{p}\psi_{h}\right\Vert _{L^{2}(\nu)}=\left\Vert f\sqrt{p}-\left\{ \left(f\sqrt{p}\right)*\psi_{h}\right\} \right\Vert _{L^{2}({\rm Leb})}
\]
and that $f\sqrt{p} \in L^{2}({\rm Leb})$, so the result follows from a simple application
of a classical result on convolution (see for instance \cite [Theorem 2.16]{LiebLoss01}).
\end{proof}
In our next result, we take $\nu$ to have a density $p$ of the form $p({\bm x})=\exp(-\phi({\bm x}))$ for some infinitely differentiable function $\phi$. In this case, we call $\phi$ the potential of $p$, and the score function $\nabla \log p$ of $p$ is $-\nabla p/p=\nabla \phi$. With a slight abuse of notation, meant to emphasise the dependence on $\nu$, we write $\overline{C^{1,2}(\nu)}$ for the closure 
of the space $C^{\infty}_c(\R^d)$ with respect to the Sobolev norm in \eqref{def:W.norm.ker} that defines $W^{1,2}(\nu)$.
\begin{proposition} 
If $\nu$ has a density
$p=\exp(-\phi)$ with full support, is infinitely differentiable
and $\nabla\phi\in L^{\infty}(\nu)$, then $W^{1,2}(\nu) \subset \overline{C^{1,2}(\nu)}$.
 In addition, if $\nabla\phi\in L^{s}(\nu)$ and 
$f\in W^{1,2}(\nu) \cap  L^r(\nu)$
for some $r\in (2,+\infty]$ and $s$ satisfying $2/s+2/r=1$, then $f \in \overline{C^{1,2}(\nu)}$.
\end{proposition}
\begin{proof}
Take $f \in W^{1,2}(\nu)$. Note that the set of functions in $W^{1,2}(\nu)$ that are compactly supported is dense in  $W^{1,2}(\nu)$. It suffices indeed to multiply any element of $W^{1,2}(\nu)$ by smooth cutoff functions, that is functions with values one on a compact set and value zero outside a neighborhood of this compact set, in such a way that their gradient stays uniformly bounded. We can thus assume, without loss of generality, that $f$ is compactly supported. In this case, if we take a mollifier $(\psi_h)_{h>0}$ that is made of infinitely differentiable, compactly supported functions $\psi_h$, then, as $p$ is assumed to be infinitely differentiable, the functions  $f*_{p}\psi_{h}$ belong to  $C^{\infty}_c(\R^d)$. Note that such a mollifier indeed exists. 
To conclude the proof of the first part of the proposition, we first note that $f*_{p}\psi_{h}$ converges to $f$ in $L^2(\nu)$ when $h \rightarrow 0$, by Proposition \ref{prop:moll.conv.in.L1}. It remains to show that for any $i\in {1,\ldots,d},  \partial_{i}\left\{ f*_{p}\psi_{h}\right\} $ converges to $\partial_i f$ in $L^2(\nu)$. As $\nabla\phi\in L^{\infty}(\nu)$, we also have $(f\partial_{i}\phi) \in L^{2}(\nu)$. Furthermore,
\begin{align*}
\partial_i \left\{ f *_p \psi_h\right\} ({\bm x})&=
\frac{1}{\sqrt{p({\bm x})}}\left\{ \partial_{i}(f\sqrt{p})*\psi_h\right\} ({\bm x})
-\frac{\partial_{i}p({\bm x})}{2p^{3/2}({\bm x})}
\left\{ (f\sqrt{p})*\psi_h\right\} ({\bm x})
\\
 & =\left\{ \partial_{i}f*_{p}\psi_h\right\} ({\bm x})-\frac{1}{2}\left\{ f\partial_{i}\phi*_p \psi_h\right\} ({\bm x})+\frac{1}{2}\partial_i \phi({\bm x})
 \left\{ f*_p\psi_h\right\} ({\bm x}).
\end{align*}
Taking the limit when $h\rightarrow 0$ gives that $\partial_{i}f=\lim_{h\rightarrow0}\partial_{i}\left\{ f*_{p}\psi_{h}\right\} $ in $L^2(\nu)$.
The second part of the proposition follows from the same reasoning, using the fact that the assumptions $\nabla\phi \in L^s(\nu)$ and $f\in L^r(\nu)$ imply that $f\nabla\phi \in L^2(\nu)$.
\end{proof}
}
\end{supplement}

\end{document}